\documentclass{article}

\usepackage{amsthm,amsmath,amssymb}
\usepackage{bbm}
\usepackage{epsfig}
\usepackage{geometry}
  \newtheorem{lemma}{Lemma}[section]
  \newtheorem{theorem}{Theorem}[section]
  \newtheorem{corollary}[lemma]{Corollary}

  \newcommand{\sgn}{\text{sgn}}

\usepackage{cite}

\begin{document}
%
%
%
%
%
%
\title{Convergence of the stochastic Euler scheme\\ for locally Lipschitz coefficients}%
%
%

\author{Martin Hutzenthaler$^1$
and Arnulf Jentzen$^2$
\smallskip
\\
\small{$^1$LMU Biozentrum,
Department Biologie II,
University of Munich (LMU),} 
\\
\small{D-82152~Planegg-Martinsried, Germany,
e-mail: hutzenthaler$\,$(at)$\,$bio.lmu.de}
\\
\small{Phone:
+49-89-2180-74179,
Fax: 
+49-89-2180-74104}
\smallskip
\\
\small{$^2$Program in Applied 
and Computational Mathematics, 
Princeton University,}
\\
\small{Princeton, NJ 08544-1000, 
USA, e-mail: ajentzen$\,$(at)$\,$math.princeton.edu}
\\
\small{Phone:
+1-609-258-2654,
Fax: 
+1-609-258-1735}
}


\maketitle
\def\thefootnote{\fnsymbol{footnote}}
\def\@makefnmark{\hbox to\z@{$\m@th^{\@thefnmark}$\hss}}
\footnotesize\rm\noindent
%
%
\footnote[0]{{\it Communicated by
Peter Kloeden.}
{\it AMS\/\ subject
classifications {\rm Primary 65C05;
Secondary 60H35, 65C30}}.}
\footnote[0]{{\it Corresponding 
author: Arnulf Jentzen.}}
\normalsize\rm
\begin{abstract}
  Stochastic differential equations are 
  often simulated with the
  Monte Carlo Euler  method.
  Convergence of this method is well understood 
  in the case of globally
  Lipschitz continuous coefficients
  of the stochastic differential equation.
  The important case of
  superlinearly growing coefficients, however,
  has remained an open question.
  The main difficulty is that numerically 
  weak convergence fails
  to hold in many cases of superlinearly growing 
  coefficients.
%
%
%
  In this paper we overcome this 
  difficulty and
  establish convergence of the Monte 
  Carlo Euler method
  for a large class of one-dimensional 
  stochastic differential equations
  whose drift functions have at most 
  polynomial growth.
\end{abstract}



%
\section{Introduction}%
\label{sec:Introduction}
Many applications require the numerical 
approximation of moments or 
expectations of other functionals
of the solution of a stochastic differential equation (SDE) whose coefficients
are superlinearly growing. 
Moments are often approximated 
by discretizing time using the stochastic
Euler scheme (see 
e.g.\ \cite{KlPl},
\cite{Mil95},
\cite{Y02}) 
(a.k.a.\ Euler-Maruyama scheme) and by approximating expectations
with the Mon\-te Car\-lo method.
This Monte Carlo Euler method has been shown to converge 
in the case of globally Lipschitz continuous 
coefficients of the SDE
(see e.g.~Section~14.1 in \cite{KlPl} and 
Section 12 in~\cite{Mil95}).
The important case of
superlinearly growing coefficients, however,
has remained an open problem.
The main difficulty is that numerically 
weak convergence fails
to hold in many cases of 
superlinearly growing coefficients; see~\cite{HJK2011}.
In this paper we overcome this difficulty and establish
convergence of the Monte Carlo Euler method
for a large class of 
one-dimensional SDEs with
at most polynomial growing drift functions and
with globally Lipschitz continuous diffusion functions;
see 
Section~\ref{sec:Main result} 
for the exact statement.

For clarity of exposition,
we concentrate in this 
introductory section on the
following prominent example.
Let $ T \in (0,\infty) $ 
be fixed and let $(X_t)_{t \in [0,T]}$ 
be the unique
strong solution
of the
one-dimensional SDE
\begin{equation}\label{eq:X_intro}
  d X_t =
  - X_t^3\,dt + \bar{\sigma} \, dW_t,
  \qquad
  X_0 = x_0
\end{equation}
for all $ t \in [0,T] $,
where $(W_t)_{t \in [0,T] }$ is a one-dimensional
standard 
Brownian motion
with continuous sample paths
and where $ \bar{\sigma} \in [0,\infty) $ and
$ x_0 \in \mathbb{R}$ are given constants.
Our goal is then to solve the cubature
approximation problem of the 
SDE~\eqref{eq:X_intro}. 
More formally, we want to compute 
moments and, more generally, the deterministic
real number
\begin{equation}  \label{eq:functional}
  \mathbb{E}\Bigl[f\bigl( X_T \bigr)
              \Bigr]
\end{equation}
for a given smooth
function $ f \colon \mathbb{R} 
\rightarrow \mathbb{R} $
whose derivatives have at most polynomial growth.

A frequently used scheme for solving this problem is 
the Monte Carlo Euler method.
In this method, time is discretized through 
the stochastic Euler scheme
and expectations are approximated 
by the Monte Carlo method.
More formally, the Euler approximation
$(Y_n^N)_{ n \in \left\{ 0,1,\ldots,N \right\} }$ of 
the solution
$\left( X_t \right)_{t \in [0,T]}$
of the SDE \eqref{eq:X_intro}
is defined recursively through
$ Y^N_0= x_0 $ and
\begin{equation}\label{eq:Euler_intro}
  Y^N_{n+1}
  =
  Y^N_n
  -
  \frac{T}{N}
  \left( Y^N_n \right)^3
  +
  \bar{\sigma} \cdot \left(
   W_{ \frac{(n+1) T}{N} }
   -
   W_{ \frac{ n T }{N} }
  \right)
\end{equation}
for every $n \in \left\{ 0,1, \ldots,N-1 \right\} $
and every $ N \in \mathbb{N}:=\{1,2,\ldots\}$.
Moreover, let $ Y_n^{N,m} $, 
$ n \in \left\{ 0,1, \ldots,N \right\}$, 
$ N \in \mathbb{N}$,
for $ m \in \mathbb{N} $
be independent copies of the Euler approximation
defined in~\eqref{eq:Euler_intro}.
The Monte Carlo Euler
approximation of~\eqref{eq:functional} 
with $N \in \mathbb{N}$ time steps and 
$M \in \mathbb{N}$ Monte Carlo runs
is then
the random real number
\begin{equation}  \label{eq:MCE_intro}
  \frac{1}{M}\left( \sum_{m=1}^{M} 
  f\bigl( Y_N^{N,m} \bigr)
  \right).
\end{equation}
In order to balance the error due
to the Euler method
and the error due to the Monte Carlo method, it turns
out to be optimal to have $M$ increasing at the
order of $N^2$; see~\cite{DG95}.
We say that the Monte Carlo Euler method 
converges
if
\begin{equation}  
\label{eq:MCC_intro}
  \lim_{N\to\infty}\left|
    \mathbb{E}\Bigl[f\bigl( X_T \bigr)
              \Bigr]
    -
    \frac{1}{N^2}\sum_{m=1}^{N^2}
       f \bigl( Y_N^{N,m} \bigr) 
  \right| = 0
\end{equation}
holds 
almost surely
for every smooth 
function
$ f\colon \mathbb{R} 
\rightarrow \mathbb{R} $
whose derivatives have 
at most polynomial growth
(see also Appendix~A.1 
in \cite{G04}).

In the literature, convergence of the Monte Carlo 
Euler method
is usually established by estimating 
the bias
and by estimating the statistical 
error (see e.g.\ Section~3.2 in \cite{Ta95}).
More formally, the triangle inequality yields
\begin{equation}  \begin{split}  \label{eq:error_error}
\underbrace{
  \Bigl| \mathbb{E}\Bigl[f\bigl( X_T \bigr)
              \Bigr]
         -\frac{1}{N^2}\sum_{m=1}^{N^2}
              f \bigl( Y_N^{N,m} \bigr)
  \Bigr|
}_{
\substack{
\text{approximation error of the}
\\
\text{Monte Carlo 
Euler method}
}
}
  \leq
  \underbrace{
    \Bigl|
        \mathbb{E}\Bigl[f\bigl( X_T \bigr)
                  \Bigr]
       -\mathbb{E}\Bigl[f\bigl( Y_N^N \bigr)
              \Bigr]
  \Bigr|
  }_{
    \substack{
      \text{absolute value}
    \\
      \text{of the bias }
    }
  }
  +
\underbrace{
  \Bigl| \mathbb{E}\Bigl[f\bigl( Y_N^N \bigr)
              \Bigr]
         -\frac{1}{N^2}\sum_{m=1}^{N^2}
              f \bigl( Y_N^{N,m} \bigr)
  \Bigr|
}_{\text{statistical error}}
\end{split}     \end{equation}
for every $N\in\mathbb{N}$ and
every smooth function
$ f\colon \mathbb{R} \rightarrow \mathbb{R} $
whose derivatives have 
at most polynomial growth.
The first summand  on the right-hand side of~\eqref{eq:error_error}
is the absolute value of the
bias due to 
approximating the exact
solution with Euler's method.
The second summand on the right-hand side of~\eqref{eq:error_error}
is the statistical error which is due to approximating
an expectation with the arithmetic average over independent copies.
The bias is usually
the more difficult part to estimate.
This is why the concept of numerically weak convergence,
which concentrates on that error part,
has been studied intensively in the literature
(see for 
instance 
\cite{Mil78}, 
\cite{KPS94}, 
\cite{Mil95}, 
\cite{Ta95},
\cite{bt96}, 
\cite{H01}, 
\cite{K02}, 
\cite{R03}
or Part VI in \cite{KlPl}).
To give a definition, we say that 
the stochastic Euler scheme converges 
in the numerically weak sense 
(not to be confused with stochastic weak
convergence)
if the bias of the Monte
Carlo Euler method converges to zero,
i.e., if
\begin{equation}    \label{eq:weak_conv_intro}
 \lim_{N\to\infty}
  \Bigl|
        \mathbb{E}\Bigl[f\bigl( X_T \bigr)
                  \Bigr]
       -\mathbb{E}\Bigl[f\bigl( Y_N^N \bigr)
              \Bigr]
  \Bigr|
  =0
\end{equation}
holds for 
every smooth function
$ f\colon \mathbb{R} \rightarrow \mathbb{R} $
whose derivatives have at most polynomial growth.
If the coefficients of the SDE are 
globally Lipschitz continuous,
then numerically weak convergence 
of Euler's method
and convergence of the Monte
Carlo Euler method is 
well-established;
see e.g.\ Theorem 14.1.5 in \cite{KlPl}
and Section 12 in \cite{Mil95}.

The case of superlinearly growing coefficients is more subtle.
The main difficulty in that case
is that numerically weak 
convergence 
usually fails to hold;
see \cite{HJK2011} for a large class of examples.
In particular, the sequence
$\mathbb{E}\bigl[(Y_N^N)^2\bigr]$, $N \in \mathbb{N}$, 
of second moments of the Euler 
approximations \eqref{eq:Euler_intro}
diverges to infinity if $\bar{\sigma}>0$
although the second moment
$\mathbb{E}\bigl[ (X_T)^2\bigr]$ of the exact
solution of the SDE \eqref{eq:X_intro} 
is finite and, hence, we have
\begin{equation}    
\label{eq:weak_expl}
 \lim_{N\to\infty}
  \Bigl|
        \mathbb{E}\Bigl[ ( X_T )^2
                  \Bigr]
       -\mathbb{E}\Bigl[ ( Y_N^N )^2
              \Bigr]
  \Bigr|
  = \infty 
\end{equation}
instead of \eqref{eq:weak_conv_intro}.
The absolute value of the
bias thus diverges to 
infinity in case of SDEs 
with superlinearly growing
coefficients. 
This in turn implies
divergence of the Monte 
Carlo Euler method
in the mean square sense, i.e.,
\begin{equation}
\label{eq:divMCE}
  \underbrace{
  \mathbb{E}\!\left[\bigg|
    \mathbb{E}\Bigl[ ( X_T )^2 
              \Bigr]
    -
    \frac{1}{N^2}\sum_{m=1}^{N^2}
       \big( Y_N^{N,m} \big)^{ 2 } \,
  \bigg|^2 \right] 
  }_{ 
    \substack{
      \text{mean square error of the}
      \\
      \text{Monte Carlo Euler method}
    }
  }
  \! =
  \underbrace{
  \Bigl|
        \mathbb{E}\Bigl[ ( X_T )^2
                  \Bigr]
       -\mathbb{E}\Bigl[ ( Y_N^N )^2
              \Bigr]
    \Bigr|^2
  }_{
    \text{squared bias } \to\infty
  }
  +
  \underbrace{
  \text{Var}\!\left( 
    \frac{1}{N^2}\sum_{m=1}^{N^2}
       \big( Y_N^{N,m} \big)^{ 2 } \,
  \right)
  }_{
    \substack{ 
    \text{variance of the Monte}
    \\
    \text{Carlo Euler method}
    }
  }
  \rightarrow \infty
\end{equation}
as $ N \rightarrow
\infty $.
Clearly, the 
mean square
divergence~\eqref{eq:divMCE}
does not exclude the
almost sure convergence~\eqref{eq:MCC_intro}.
Indeed, the main result
of this article proves 
the 
convergence~\eqref{eq:MCC_intro}
of the Monte Carlo Euler
method.
For proving this result,
we first need to understand
why Euler's method does not converge
in the sense of numerically weak convergence.
In the deterministic case,
that is,
\eqref{eq:X_intro} and \eqref{eq:Euler_intro} with $\bar{\sigma}=0$,
the Euler approximation diverges 
if the starting point is sufficiently
large.
This divergence has been estimated in~\cite{HJK2011} and turns out
to be at least double-exponentially fast.
Now in the presence of noise ($\bar{\sigma}>0$),
the Brownian motion has an
exponentially small chance to push the Euler approximation outside of $[-N,N]$.
On this event, the Euler approximation grows
at least double-exponentially fast due to the deterministic dynamics.
Consequently, as being 
double-exponentially large over-compensates that
the event has an 
exponentially small probability, the $L^2$-norm of the
Euler approximation diverges to infinity
and, hence, numerically weak convergence fails to hold.

Now we indicate for example~\eqref{eq:X_intro} with $x_0=0$ and $\bar{\sigma}=1$
why the Monte Carlo Euler method converges
although the stochastic Euler scheme
fails to converge
in the sense of numerically weak convergence.
Consider the event
$\Omega_N:=\{ \sup_{0 \leq t \leq T}|W_t| 
\leq \sqrt{ N / (2 T) } \}$ and note 
that the probability of $
( \Omega_N )^c$ is exponentially small
in $N \in \mathbb{N}$.
The key step in our proof is to show that the Euler approximation
does not diverge on $\Omega_N$
as $N \in \mathbb{N}$ goes to
infinity.
More precisely, one can show that
the Euler approximations \eqref{eq:Euler_intro} 
are uniformly dominated on $ \Omega_N $ by twice the
supremum of the Brownian motion, i.e.,
\begin{equation}      \label{eq:domination_intro}
  \sup_{N\in\mathbb{N}}
  \bigg( 
    \mathbbm{1}_{ \Omega_N }
    \left| Y_N^N \right| 
  \bigg)
  \leq 
  2\left( \sup_{0 \leq t \leq T}|W_t|
  \right)
\end{equation}
holds.
Consequently, the restricted absolute moments
are uniformly bounded
\begin{equation}   \label{eq:uniform_bound_intro}
  \sup_{N\in\mathbb{N}} \mathbb{E}\biggl[
                        \mathbbm{1}_{\Omega_N}
                        \bigl|  Y_N^N  \bigr|^p
            \biggr]
  \leq
  2^p \cdot 
  \mathbb{E}\biggl[\Bigl(\sup_{0 \leq t
  \leq T}|W_t|\Bigr)^p \, \biggr]
  <\infty  
\end{equation}
for all $ p \in [1,\infty) $.
This estimate complements the divergence
\begin{equation}
  \lim_{N\to\infty}\mathbb{E}
    \Bigl[
      \mathbbm{1}_{
        ( \Omega_{N} )^c
      }
      \bigl| Y_N^N \bigr|^p 
    \Bigr]
  = \infty 
\end{equation}
for all $p\in[1,\infty)$,
which has been established in~\cite{HJK2011}.
Now once the restricted absolute moments
are uniformly bounded,
an adaptation of the arguments of the globally Lipschitz case
leads to the modified numerically weak convergence
\begin{equation} \label{eq:modified_weak_convergence_intro}
  \lim_{N\to\infty}\mathbb{E}\Bigl[ 
    \mathbbm{1}_{\Omega_N}
    f\bigl(Y_N^N\bigr) 
 \Bigr]
  =\mathbb{E}\Bigl[f(X_T)\Bigr]
\end{equation}
for every smooth function
$ f \colon 
\mathbb{R} \rightarrow \mathbb{R} $ whose derivatives
have at most polynomial growth,
see Lemma~\ref{lemm_ew}.
By substituting 
this into an inequality 
analogous to~\eqref{eq:error_error}
and by using the
exponential decay of 
the probability of $(\Omega_N)^c$
in $ N \in \mathbb{N} $,
one can establish convergence
of the Monte Carlo Euler method.
Note that a domination as strong as~\eqref{eq:domination_intro}
holds for more general non-increasing drift functions 
if the diffusion function is identically equal to $1$.
For more general drift and diffusion functions, however, both $\Omega_N$
and the dominating process are more complicated in that they depend on
the Euler approximation.
Nevertheless, the dominating process can be shown to have uniformly bounded
absolute moments; see Section~\ref{sec:Proof_thm_1} for the details.

Our main result, Theorem~\ref{thm:main_result} 
below, establishes convergence of the
Monte Carlo Euler method for SDEs
with globally one-sided Lipschitz continuous
drift functions and with globally
Lipschitz continuous diffusion functions.
Moreover, the coefficients of the SDE 
are assumed to
have continuous fourth derivatives with
at most polynomial growth,
see Section \ref{sec:Main result}
for the exact statement.
The order of convergence turns out to be 
as in the globally Lipschitz case.
In that case,
the stochastic Euler scheme converges
in the sense of
numerically weak convergence 
with order $1$.
The Monte Carlo simulation of 
$\mathbb{E}\bigl[f\bigl(Y_N^{N}\bigr)\bigr]$
with $M$ independent Euler approximations
has convergence order $\frac{1}{2}-$.
For a real number $ r > 0 $, we write
$ r- $ for the convergence order if
the convergence order is better than
$ r - \varepsilon $ for every
arbitrarily small $ \varepsilon \in (0,r) $.
We therefore choose $ M = N^2 $
in order to balance the error
arising from Euler's approximation
and the error arising from the Monte Carlo approximation.
Both error terms are then bounded by a random multiple
of 
$ 
  N^{ (\varepsilon - 1) } 
$
with
$ \varepsilon \in (0,1) $.
Since $ O\!\left( M \cdot N \right) 
= O\!\left( N^3 \right) $ function
evaluations, arithmetical operations
and random variables are needed to
compute the Monte Carlo Euler approximation
\eqref{eq:MCE_intro}, the Monte Carlo Euler
method converges with order
$ \frac{1}{3}- $ with respect to
the computational effort in the case of global
Lipschitz coefficients of the SDE
(see \cite{DG95}).
Theorem~\ref{thm:main_result} shows
that $ \frac{1}{3}-$ is also
the convergence order in the case
of superlinearly growing coefficients
of the SDE.
Simulations support this result,
see Section~\ref{sec:Simulations}.

Let us reconsider the standard splitting~\eqref{eq:error_error}
of the approximation error 
into bias and statistical error.
Theorem 2.1 of~\cite{HJK2011} implies 
that the absolute value
of the bias diverges to infinity
as $N\to\infty$.
This together with our 
Theorem~2.1 below 
yields that also the statistical
error diverges to infinity. 
More formally, we see that
\begin{equation}  \begin{split}
  \underbrace{
  \Bigl| \mathbb{E}\Bigl[f\bigl( X_T \bigr)
              \Bigr]
         -\frac{1}{N^2}\sum_{m=1}^{N^2}
              f \bigl( Y_N^{N,m} \bigr)
  \Bigr|
  }_{
\substack{
\text{approximation error of the}
\\
\text{Monte Carlo 
Euler method}
}
  \; \to 0
}
  \leq
  \underbrace{
  \Bigl|
        \mathbb{E}\Bigl[f\bigl( X_T \bigr)
                  \Bigr]
       -\mathbb{E}\Bigl[f\bigl( Y_N^N \bigr)
              \Bigr]
  \Bigr|
  }_{
    \substack{
      \text{absolute value}
    \\
      \text{of the bias }
    }
    \;\to\infty
  }
  +
  \underbrace{
  \Bigl| \mathbb{E}\Bigl[f\bigl( Y_N^N \bigr)
              \Bigr]
         -\frac{1}{N^2}\sum_{m=1}^{N^2}
              f \bigl( Y_N^{N,m} \bigr)
  \Bigr|
  }_{\text{statistical error }\to\infty}
\end{split}     \end{equation}
$ \mathbb{P} $-a.s.\ as $ N \to \infty $
for every smooth function
$ f \colon \mathbb{R} \rightarrow
\mathbb{R} $ with at most polynomially growing derivatives and with
$ 
  f(x) \geq 
  c
  \left| x \right|^{ c } - 1/c 
$
for all $ x \in \mathbb{R} $
and some $ c \in (0,\infty) $.
This emphasizes that the standard splitting 
of the approximation error of the 
Monte Carlo Euler method into
bias and statistical error is not 
appropriate in case of SDEs with
superlinearly growing coefficients.

%
\section{Main result}%
\label{sec:Main result}
 We establish convergence
 of the Monte Carlo Euler method
 for more general one-dimensional diffusions than
 our introductory example~\eqref{eq:X_intro}.
 More precisely, we pose 
 the following assumptions on the coefficients.
 The drift function is assumed to be globally one-sided Lipschitz continuous
 and the diffusion function is assumed to be globally Lipschitz continuous.
 Additionally, both the drift function and the 
 diffusion function
 are assumed to
 have a continuous fourth derivative
 with at most polynomial growth.
 
We introduce further notation for the formulation 
of our main result.
Fix $ T \in (0,\infty) $ and let
$ \left( \Omega, \mathcal{F}, \mathbb{P} \right) $
be a probability space with a normal filtration 
$ \left( \mathcal{F}_t \right)_{t \in [0, T] } $. Let
$ \xi^{(m)} \colon \Omega \rightarrow \mathbb{R} $, 
$ m \in \mathbb{N}$, be a sequence of independent,
identically distributed 
$ \mathcal{F}_0 $/$\mathcal{B}(\mathbb{R})$-measurable 
mappings with
$ 
\mathbb{E}\!\left[ 
  | \xi^{(1)} |^p 
\right] < \infty $
for every $ p \in [ 1, \infty )$ and let 
$ W^{(m)} \colon [0,T] \times \Omega \rightarrow \mathbb{R} $,
$ m \in \mathbb{N} $,
be a sequence of independent scalar standard 
$ \left( \mathcal{F}_t \right)_{t \in [0, T] } $-Brownian motions
with continuous sample paths.
Furthermore, let $ \mu, \sigma \colon \mathbb{R} \rightarrow \mathbb{R} $
be four times continuously
differentiable functions.
Generalizing~\eqref{eq:X_intro},
let $(X_t)_{t \in [0,T]}$ be a one-dimensional
diffusion with drift function $\mu$ and diffusion function $\sigma$.
More precisely, let
$ X \colon [0, T] \times \Omega \rightarrow \mathbb{R} $ be an 
(up to indistinguishability unique) 
adapted stochastic process with continuous sample paths which satisfies
\begin{equation}\label{x_def}
  \mathbb{P}\!\left[
    X_t
    =
    \xi^{(1)}
    +
    \int_0^t \mu( X_s ) \, ds
    +
    \int_0^t \sigma( X_s ) \, dW_s^{(1)}
  \right]
  = 1
\end{equation}
for every $ t \in [0, T] $.
The functions $\mu$ and $\sigma^2$
are the infinitesimal mean and the infinitesimal 
variance respectively.

Next we introduce independent versions of the Euler approximation.
Define $ \mathcal{F} / \mathcal{B}(\mathbb{R}) $-measurable mappings
$ Y_n^{N,m} \colon \Omega \rightarrow \mathbb{R} $,
$ n \in \{0,1,\ldots,N\} $, $ N \in \mathbb{N} $,
$ m \in \mathbb{N} $,
by $ Y_0^{N,m}(\omega) := \xi^{(m)}(\omega) $
and by
\begin{equation}  \begin{split}  \label{eq:Euler_approx}
  Y_{n+1}^{N,m}(\omega) 
  := 
      Y_{n}^{N,m}(\omega)
      + \frac{T}{N} \cdot \mu\bigl(Y_{n}^{N,m}(\omega)\bigr)
      + \sigma\bigl(Y_{n}^{N,m}(\omega)\bigr) 
      \cdot \left( 
        W^{(m)}_{\frac{(n+1)T}{N}}(\omega) - W^{(m)}_{\frac{nT}{N}}(\omega)
      \right)
\end{split}     \end{equation}
for every $ n \in \{0,1,\ldots,N-1\} $,
$ N \in \mathbb{N} $ and every
$ m \in \mathbb{N} $.
Now we formulate the main result of this article.

%
%
%
%
\begin{theorem}\label{thm:main_result}
Suppose that $ f, \mu, \sigma\colon \mathbb{R} \rightarrow \mathbb{R} $
are four times continuously differentiable functions
with
\begin{equation}\label{eq:polynomial_growth}
  \left| f^{(n)}(x) \right|
  +
  \left| \mu^{(n)}(x) \right|
  +
  \left| \sigma^{(n)}(x) \right|
  \leq 
  L\left(1 + |x|^{\delta}\right) 
  \quad\forall\, x\in\mathbb{R} 
\end{equation}
for every $ n \in \left\{ 0,1,\dots,4 \right\} $,
where 
$ L \in (0, \infty )$ and $ \delta \in (1,\infty) $
are fixed constants.
Moreover, assume that the drift coefficient
is globally one-sided 
Lipschitz continuous
\begin{equation}\label{eq:onesided}
  (x-y) \cdot \left( \mu(x) - \mu(y) \right) \leq 
  L\left(x-y\right)^2
  \quad\forall\, x,y\in\mathbb{R} 
\end{equation}
and that the diffusion coefficient is globally
Lipschitz continuous
\begin{equation}  \label{eq:sigmalip}
  \left| \sigma (x) - \sigma (y) \right| 
  \leq L \left| x - y \right| 
  \quad \forall\, x,y\in\mathbb{R}.
\end{equation}
Then there are $ \mathcal{F} / \mathcal{B}([0,\infty)) $-measurable mappings $ C_{\varepsilon} \colon \Omega \rightarrow [0,\infty) $,
$ \varepsilon \in (0,1) $, and a set 
$ \tilde{\Omega} \in \mathcal{F} $
with
$ \mathbb{P}[\tilde{\Omega}] = 1 $
such that
\begin{equation}      \label{eq:main_result}
  \left|
    \mathbb{E}\bigg[f(X_T)\bigg] -
    \frac{1}{N^2} \left(
      \sum_{m=1}^{N^2} f(Y_N^{N,m}(\omega))
    \right) 
  \right|
  \leq  C_{\varepsilon}(\omega) \cdot
  \frac{  1                     }
       {  N^{(1-\varepsilon )}  }
\end{equation}
holds for every $ \omega \in \tilde{\Omega} $,
$ N \in \mathbb{N} $
and every $ \varepsilon \in (0,1) $.
\end{theorem}
The proof is deferred to Section~\ref{sec:Proof_thm_1}.
For further numerical
approximation results for SDEs with superlinearly growing coefficients,
see e.g.~\cite{HMS2002},
\cite{MSH2002},
\cite{MT2005},
\cite{Szpruch}
and the references in 
the introductory
section of~\cite{HJK2011}.

Note that
the assumptions of Theorem~\ref{thm:main_result}
ensure the existence of an 
adapted stochastic
process $ X\colon [0,T] \times \Omega
\rightarrow \mathbb{R} $
with continuous sample
paths which satisfies
\eqref{x_def} and
\begin{equation}
  \mathbb{E}\left[ 
    \sup_{0 \leq t \leq T}
    \left| X_t \right|^p
  \right] < \infty
\end{equation}
for all $ p \in [1,\infty) $
(see Theorem 2.6.4 
in \cite{GS72}).
Therefore, the expression
$ \mathbb{E}\bigl[ f( X_T ) \bigr ] $
in \eqref{eq:main_result}
in Theorem \ref{thm:main_result}
is well-defined.

Since $ O\!\left( N^3 \right) $
function evaluations, 
arithmetical operations and random variables
are needed to compute
the expression
$
  \frac{1}{N^2} \left(
    \sum_{m=1}^{N^2} f(Y_N^{N,m}(\omega))
  \right) 
$
in \eqref{eq:main_result}
for $ \omega \in \Omega $,
Theorem~\ref{thm:main_result} shows
that the Monte Carlo Euler method
converges under the above assumptions
with order
$\tfrac{1}{3}-$
with respect to the computational effort.
This is the standard convergence order
as in the global Lipschitz case
(see e.g. \cite{DG95}).

\section{Simulations}%
\label{sec:Simulations}
In this section, we simulate the second moment of two stochastic
differential equations.
First we simulate the stochastic Ginzburg-Landau 
equation with multiplicative noise, 
which we choose as there exists 
an explicit solution for this SDE.
Let $(X_t)_{ t \in [0,1] }$ 
be the solution of
\begin{equation}\label{eq:ginzburg2}
  d X_t = \left( 
  \frac{1}{2} X_t - 
  X_t^3 \right) dt 
  + X_t \, dW_t,
  \qquad
  X_0 = 1
\end{equation}
for all $ t \in [0,1]$.
The exact solution at time $1$ 
is known explicitly (see e.g.\ Section 4.4 in \cite{KlPl}) and is given by
\begin{equation} \label{eq:ginzburg_explicit2}
 X_1=\frac{ \exp\left( W_1\right)}
          {\sqrt{1+2 \int_0^1\exp\left(2 W_s\right)ds}}.
\end{equation}
The exact value of the second moment
$\mathbb{E}\bigl[ ( X_T )^2 
\bigr] = \mathbb{E}\bigl[(X_1)^2\bigr]$
is not known.
Instead we use the exact 
solution~\eqref{eq:ginzburg_explicit2}
at time $1$
to approximate the second moment.
For this, we approximate the Lebesgue integral in the denominator
of~\eqref{eq:ginzburg_explicit2} with a Riemann sum 
with $3 \cdot 10^3$
summands.
Moreover, we approximate the second moment at time $1$
by a Monte Carlo simulation with $10^7$ independent approximations of $X_1$.
This results in the approximate value
$\mathbb{E}\bigl[(X_1)^2\bigr]\approx 0.4945$.

Next we approximate the second moment at time $1$ with the
Monte Carlo Euler method.
We will sample one random $ \omega \in \Omega $ 
and calculate the Monte Carlo Euler approximations 
for this $ \omega \in \Omega $ for 
different discretization step sizes.
More precisely,
Table~\ref{t:ginzburg} 
\begin{table}
  \begin{center}
  \begin{tabular}{|c|c|c|c|c|}
  \hline
   $N=2^0$
  &$N=2^1$
  &$N=2^2$
  &$N=2^3$
  &$N=2^4$
  \\ \hline
   $1.1379$
  &$0.9118$
  &$0.4258$
  &$0.2942$
  &$0.4386$
  \\ \hline\hline
  $N=2^5$
  &$N=2^6$
  &$N=2^7$
  &$N=2^8$
  &$N=2^9$
  \\ \hline
   $0.4641$
  &$0.4663$
  &$0.4859$
  &$0.4904$
  &$0.4935$
  \\ \hline
  \end{tabular}
  \end{center}
  \caption{\footnotesize\label{t:ginzburg}
   Monte Carlo Euler approximations
   \eqref{eq:simulation1}
    of $\mathbb{E}\left[ (X_1)^2 \right] $ of the
   SDE \eqref{eq:ginzburg2}.}
\end{table}
shows
the Monte Carlo Euler
approximation
\begin{equation}\label{eq:simulation1}
    \frac{1}{N^2}
      \sum_{m=1}^{N^2} \left(Y_N^{N,m}(\omega)\right)^2
\end{equation}
of the second moment at time $1$
of the SDE \eqref{eq:ginzburg2}
for every $ N \in \left\{ 2^0, 2^1, 2^2, \dots, 2^9
\right\}$
and one random $ \omega \in \Omega $.
In Figure~\ref{f:ginzburg_plot},
\begin{figure}[htp]
  \begin{center}
  \includegraphics[width=8.4cm]{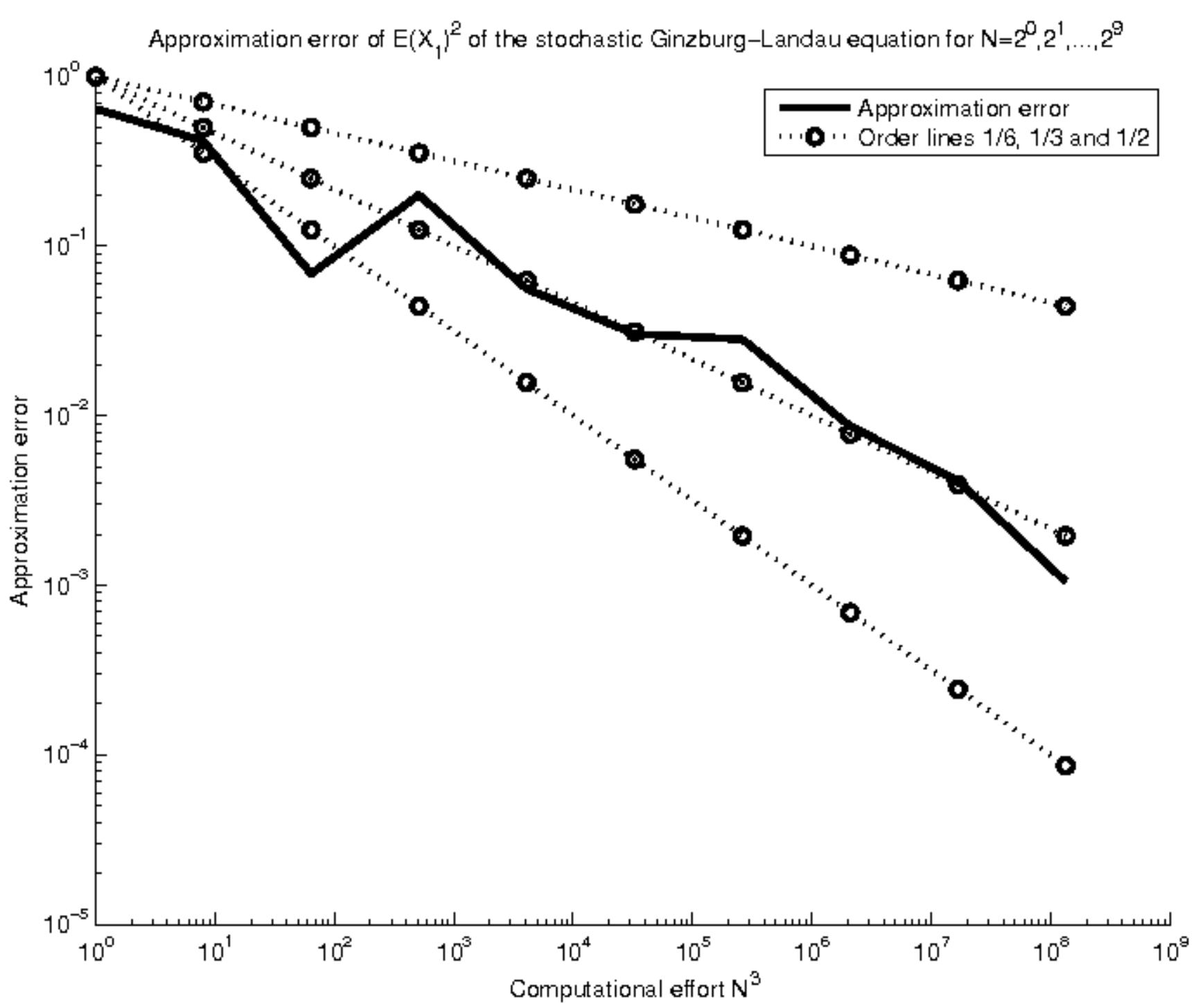}
  \caption{\footnotesize\label{f:ginzburg_plot}
   Approximation error \eqref{eq:simulation2} of the
   Monte Carlo Euler 
   approximations \eqref{eq:simulation1} 
   of $\mathbb{E}\bigl[ (X_1)^2 \bigr]$ of the 
   SDE \eqref{eq:ginzburg2}.}
   \end{center}
\end{figure}
the approximation error
of these Monte Carlo Euler approximations, i.e.,
the quantity
\begin{equation}\label{eq:simulation2}
  \left|
    0.4945-
    \frac{1}{N^2}
      \sum_{m=1}^{N^2} \left(Y_N^{N,m}(\omega)\right)^2
  \right|,
\end{equation}
is plotted against $N^3$
for every 
$N \in \left\{ 2^0, 2^1, 2^2, \dots, 2^9 \right\} $.
Note that
$N^3$ is the computational effort up to a constant.
The three order lines in Figure~\ref{f:ginzburg_plot}
correspond to the convergence orders
$ \frac{1}{6}$, $ \frac{1}{3}$ and $ \frac{1}{2}$.
Hence, Figure~\ref{f:ginzburg_plot}
indicates that the Monte Carlo
Euler method converges in the case of 
the stochastic Ginzburg-Landau 
equation \eqref{eq:ginzburg2}
with its theoretically predicted order
$ \frac{1}{3}- $.

Next we simulate our introductory example.
Let $(X_t)_{t \in [0,T]}$ be the 
solution of the SDE~\eqref{eq:X_intro}
with $T=1$, $\bar{\sigma}=1$ and $x_0=0$.
The SDE~\eqref{eq:X_intro} thus reads as
\begin{equation}\label{eq:introSDE2}
  d X_t = - X_t^3 \, dt + dW_t,
\qquad
  X_0 = 0
\end{equation}
for all $ t \in [0,1] $.
Here there exists no explicit expression for the solution or its second
moments.
As an approximation of the exact value 
$\mathbb{E}\bigl[ (X_1)^2 \bigr]$, we now
take a Monte Carlo Euler approximation 
with a larger $N $.
We choose $N=2^{12}$ and obtain
the value 
$
  0.4529 \approx
  \mathbb{E}\bigl[ (X_1)^2 \bigr] 
$
as an approximation 
of $ \mathbb{E}\bigl[ (X_1)^2 \bigr] $.
Table~\ref{t:intro} shows the value of 
\begin{table}
  \begin{center}
  \begin{tabular}{|c|c|c|c|c|}
  \hline
   $N=2^0$
  &$N=2^1$
  &$N=2^2$
  &$N=2^3$
  &$N=2^4$
  \\ \hline
   $1.4516$
  &$0.5166$
  &$0.4329$
  &$0.5308$
  &$0.4285$
  \\ \hline\hline
  $N=2^5$
  &$N=2^6$
  &$N=2^7$
  &$N=2^8$
  &$N=2^9$
  \\ \hline
   $0.4452$
  &$0.4602$
  &$0.4517$
  &$0.4548$
  &$0.4537$
  \\ \hline
  \end{tabular}
  \end{center}
  \caption{\footnotesize\label{t:intro}
   Monte Carlo Euler approximations 
   \eqref{eq:simulation1b}
   of $\mathbb{E}\bigl[ (X_1)^2 \bigr]$ of the
   SDE~\eqref{eq:introSDE2}.}
\end{table}
the Monte Carlo Euler approximation
\begin{equation}\label{eq:simulation1b}
    \frac{1}{N^2}
      \sum_{m=1}^{N^2} \left(Y_N^{N,m}(\omega)\right)^2
\end{equation}
of the second moment at time $1$
of the SDE \eqref{eq:introSDE2}
for every $ N \in 
\left\{ 2^0, 2^1, 2^2, \dots, 2^9
\right\}$
and one random $ \omega \in \Omega $.
In Figure~\ref{f:intro_plot},
\begin{figure}[htp]
   \begin{center}
  \includegraphics[width=8.7cm]{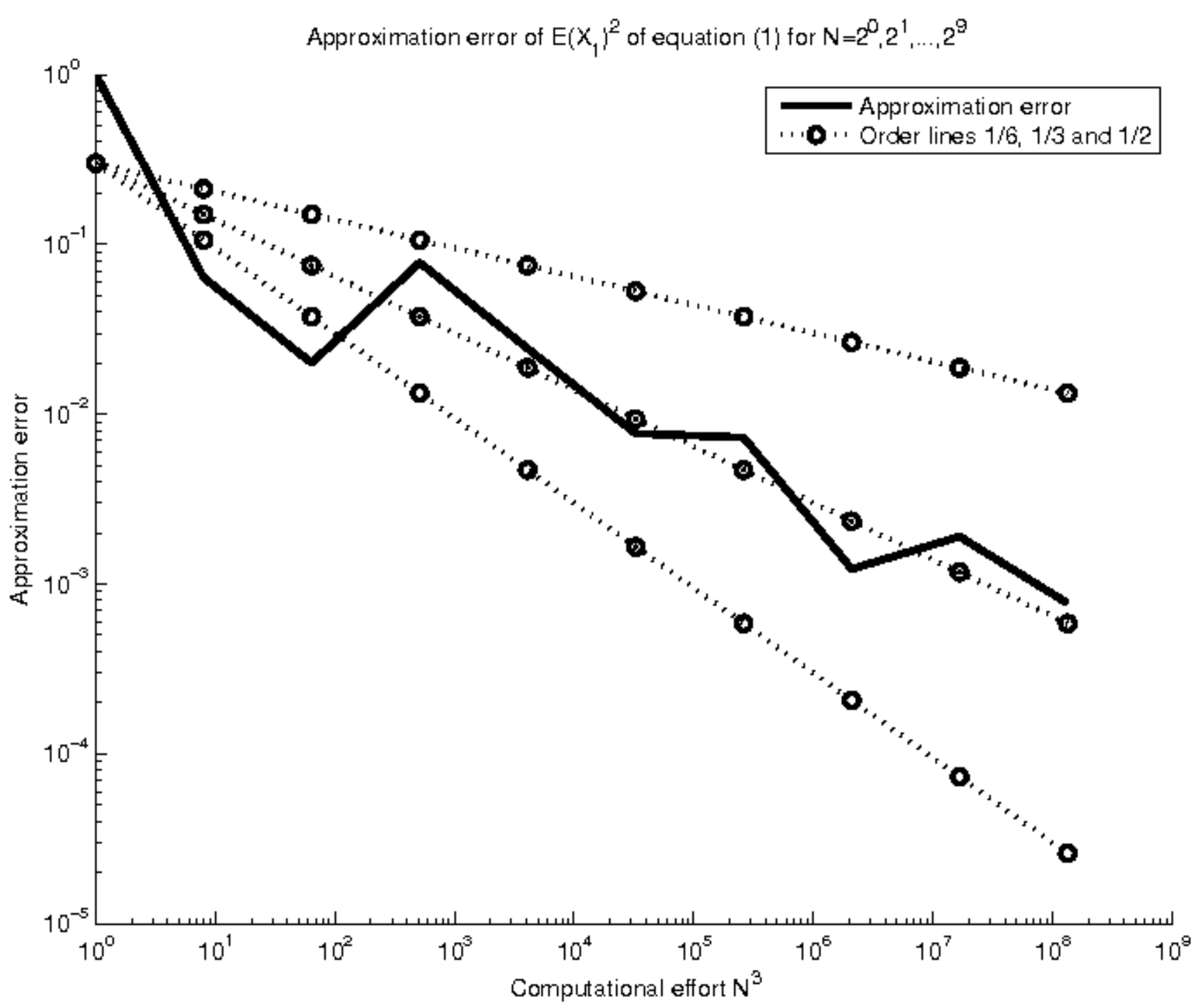}
  \caption{\footnotesize\label{f:intro_plot}
   Approximation error \eqref{eq:difference}
   of the 
   Monte Carlo Euler approximations 
   \eqref{eq:simulation1b}
   of $\mathbb{E}\bigl[ (X_1)^2 \bigr]$ of the
   SDE~\eqref{eq:introSDE2}.}
   \end{center}
\end{figure}
the approximation error
of these Monte Carlo Euler approximations, i.e.,
the quantity
\begin{equation}\label{eq:difference}
  \left|
    0.4529-
    \frac{1}{N^2}
      \sum_{m=1}^{N^2} (Y_N^{N,m}(\omega))^2
  \right|,
\end{equation}
is plotted against $N^3$
for every 
$N \in \left\{ 2^0, 2^1, 2^2, \dots, 2^9 \right\} $.
Note that
$N^3$ is the computational effort up to a constant.
The three order lines in 
Figure~\ref{f:intro_plot} 
correspond to the convergence orders
$ \frac{1}{6}$, $ \frac{1}{3}$ and $ \frac{1}{2}$.
Therefore, Figure~\ref{f:intro_plot}
suggests that
the Monte Carlo
Euler method converges 
in the case of the SDE \eqref{eq:introSDE2}
with its theoretically predicted order
$ \frac{1}{3}- $.
%
%
%
\section{Proof of Theorem \ref{thm:main_result}}%
\label{sec:Proof_thm_1}
First we introduce more notation.
Recall the standard Brownian motion
$W^{(1)}\colon[0,T]\times\Omega\to\mathbb{R}$
and the Euler approximations
$ Y^{N,1}_n\colon \Omega \rightarrow
\mathbb{R} $, $ n \in \left\{ 0,1,\dots,
N\right\}$, $ N \in \mathbb{N}$, 
from Section~\ref{sec:Main result}.
Throughout this section, we use the
stochastic process $ W\colon [0,T] \times
\Omega \rightarrow \mathbb{R} $
and the $ \mathcal{F} $/$ \mathcal{B}(\mathbb{R})$-measurable
mappings $ Y^N_n\colon \Omega
\rightarrow \mathbb{R} $, $ n \in \left\{ 0,1,\dots, N
\right\}$, $ N \in \mathbb{N} $, given by
\begin{equation}\label{y_def}
  W_t(\omega) := W^{(1)}_t(\omega)
\qquad
  \text{and}
\qquad
  Y_n^N(\omega) := Y_n^{N,1}(\omega)
\end{equation} 
for every $ t \in [0,T] $, 
$ n \in \{0,1, \ldots, N \} $,
$ \omega \in \Omega $
and every $ N \in \mathbb{N} $.
\subsection{Outline}
For general drift and diffusion functions, our proof of
Theorem~\ref{thm:main_result} somewhat buries the main
new ideas. To explain these
ideas, we give now a very
rough outline (the precise
estimates and assertions
can be found in Subsections~\ref{sec:notation}-\ref{sec:modifiedweak}
below).
The main step will be to establish uniform boundedness of the
restricted absolute moments
of the Euler approximations
\begin{equation}   \label{eq:mom}
  \sup_{N\in\mathbb{N}}
  \mathbb{E}\Bigl[ 
    \mathbbm{1}_{\Omega_N} 
    \bigl| Y_N^N \bigr| 
  \Bigr]
  <\infty
\end{equation}
where $(\Omega_N)_{N\in\mathbb{N}}$ is a sequence of events whose probabilities
converge to $1$ sufficiently fast.
From here, one can then adapt 
the arguments 
of the global Lipschitz case
to derive the modified numerically weak
convergence~\eqref{eq:modified_weak_convergence_intro}
and to obtain Theorem~\ref{thm:main_result}.

The idea behind~\eqref{eq:mom} is now explained on the example
of negative cubic drift and multiplicative noise.
Formally, we consider $\mu(x)=-x^3$, $\sigma(x)=x$ 
for all $x\in\mathbb{R}$
and $ \xi^{(1)}(\omega) = 1 $ for all
$ \omega \in \Omega $.
The Euler approximation \eqref{y_def}
is then given by $Y_0^N=1$ and
\begin{equation}
  Y_{n+1}^N=Y_n^N-\frac{T}{N}\bigl( Y_n^N \bigr)^3
  +Y_n^N\cdot\Bigl(W_{\frac{(n+1)T}{N}}-W_{\frac{nT}{N}}
                      \Bigr)
\end{equation}
for all $ n\in\{0,1,\ldots,N-1\} $ and
all $ N \in \mathbb{N} $.
Fix $ N \in \mathbb{N} $ and 
assume that the Euler approximation 
$ ( Y^N_k )_{ k \in \left\{0,1,\dots,N\right\} }$
does not change sign until
and including the
$n$-th approximation step
for some
$ n \in \left\{0,1,\dots,N\right\}$ 
which we fix for now,
that is, $Y_k^N\geq0$ for all $k \in 
\left\{ 0,1,\ldots,n\right\}$.
Then, using $1+x\leq \exp(x)$ for 
all $x\in\mathbb{R}$,
we have
\begin{equation}  \begin{split}
  Y_{k}^N &= Y_{k-1}^N - \frac{T}{N}\bigl( Y_{k-1}^N \bigr)^3
           +Y_{k-1}^N\cdot\Bigl(W_{\frac{k T}{N}}-W_{\frac{(k-1)T}{N}}
                      \Bigr)\\
    &\leq Y_{k-1}^N \Bigl( 1 + W_{\frac{k T}{N}}-W_{\frac{(k-1)T}{N}}
                \Bigr)\\
    &\leq Y_{k-1}^N \exp\Bigl(  W_{\frac{k T}{N}}-W_{\frac{(k-1)T}{N}}
                    \Bigr)
\end{split}     \end{equation}
and iterating this inequality shows
\begin{equation}  \begin{split}
  Y_{k}^N
    &\leq Y_{k-2}^N
                    \exp\Bigl(
                         W_{\frac{(k-1)T}{N}}-W_{\frac{(k-2)T}{N}}
                         \Bigr)
                    \exp\Bigl(
                          W_{\frac{k T}{N}}-W_{\frac{(k-1)T}{N}}
                         \Bigr)\\
    &= Y_{k-2}^N \exp\Bigl(
                          W_{\frac{k T}{N}}-W_{\frac{(k-2)T}{N}}
                         \Bigr)\\
    &\leq\ldots\leq 
    Y_0^N \exp\Bigl(  W_{\frac{k T}{N}}
                    \Bigr)
    =: \tilde{D}_{k}^N
\end{split}     \end{equation}
for all $k \in \{ 0,1,\ldots,n\}$.
Thus the Euler approximation 
$ ( Y^N_k )_{ k \in 
\left\{ 0,1, \dots, n
\right\} } $
is bounded above by the dominating process
$ (\tilde{D}_k^N)_{k \in \left\{0,1,\dots,n\right\} } $.
This dominating process has uniformly bounded absolute moments.
So the absolute moments of the Euler approximation can only be unbounded
if the Euler approximation changes its sign.
Now if $Y_k$ happens to be very large
for one $ k \in \left\{ 0,1, \dots, N-1 \right\}$,
then $Y_{k+1}$ is negative with absolute value being very large because
of the negative cubic drift.
Through a sequence of changes in sign,
it could happen that the absolute value of the Euler approximation increases
more and more.
To avoid this, we restrict the Euler approximation
to an event $\Omega_N$
on which the drift alone cannot change the sign of the
Euler approximation.
On $\Omega_N$, the Euler approximation changes sign only due to the
diffusive part.
As the diffusion function is at most
linearly growing, these changes of sign
can be controlled.
In between consecutive changes of sign, the Euler approximation is again
bounded by a dominating process as above.
Through this pathwise comparison with a dominating process,
we will establish the uniform boundedness~\eqref{eq:mom} of the
restricted absolute moments.
For the details, we refer to Lemma~\ref{lemm_1},
which is the key result in our
proof of Theorem~\ref{thm:main_result}.
%
%
%
\subsection{Notation and auxiliary lemmas}
\label{sec:notation}
In order to show 
Theorem \ref{thm:main_result},
the following objects are needed.
First of all, define
$ t_n^N := \frac{nT}{N} $ for every
$ n \in \{0,1,\ldots,N\} $ and every $ N \in \mathbb{N} $
and let the $ \mathcal{B}(\mathbb{R}) 
$/$\mathcal{B}(\mathbb{R}) $-measurable
mapping
$ \tilde{\sigma} \colon 
\mathbb{R} \rightarrow \mathbb{R} $
be given by
\begin{equation}
  \tilde{\sigma}(x) :=
  \begin{cases}
    \frac{ \left( \sigma(x) - \sigma(0) \right) }{ x }
    & : x \neq 0 \\
    0 & : x = 0
  \end{cases}
\end{equation}
for all $ x \in \mathbb{R} $.
Moreover, 
let the $ \mathcal{F} / 
\mathcal{B}(\mathbb{R}) $-measurable
mappings 
$ \alpha_n^{N,m} \colon \Omega \rightarrow \mathbb{R}$
and
$ \beta_n^{N,m} \colon \Omega \rightarrow \mathbb{R}$
be given by
\begin{equation}
\label{eq:def_alpha}
  \alpha_n^{N,m}(\omega)
  := 
  \frac{TL}{N} + \tilde{\sigma}(Y_n^{N,m}(\omega))
  \cdot \left(
     W^{(m)}_{t_{n+1}^N}(\omega) - 
     W^{(m)}_{t_{n}^N}(\omega)
  \right)  
\end{equation}
and
\begin{equation}
\label{eq:def_beta}
  \beta_n^{N,m}(\omega)
  :=
  \frac{T\mu(0)}{N}
  +
  \sigma(0) \cdot
  \left(
    W^{(m)}_{t_{n+1}^N}(\omega) 
    - W^{(m)}_{t_{n}^N}(\omega)
  \right)
\end{equation}
for every 
$ \omega \in \Omega $,
$ n \in \{0,1,\ldots,N-1\} $
and every
$ N, m \in \mathbb{N} $.
For simplicity we also use
$ \alpha_n^N, \beta_n^N \colon \Omega \rightarrow \mathbb{R} $
given by
$
  \alpha_n^N(\omega) := \alpha_n^{N,1}(\omega)
$ and
$
  \beta_n^N(\omega) := \beta_n^{N,1}(\omega)
$
for every $ \omega \in \Omega $, $ n \in \{0,1,\ldots, N-1 \} $
and every $ N \in \mathbb{N} $.
Using these ingredients,
we now define the dominating process.
Let the
$ \mathcal{F} / \mathcal{B}(\mathbb{R}) $-measurable
mapping
$ D_{v,w}^{N,m} \colon \Omega \rightarrow \mathbb{R} $
be given by
\begin{equation}  \begin{split}\label{d_def}
  D_{v,w}^{N,m}(\omega)
  &:=
  e^{\left(
    \sum_{l=v}^{w-1} \alpha_l^{N,m}(\omega)
  \right)}
  \left(
    T\left|\mu(0)\right| + \left| \sigma(0) \right| 
    + \left|\xi^{(m)}(\omega)\right|  + 1
  \right)
\\
 &\qquad +
  \sum_{k=v}^{w-1}
  \sgn( Y_k^{N,m}(\omega) ) \,
  e^{\left(
    \sum_{l=k+1}^{w-1} \alpha_l^{N,m}(\omega)
  \right)}
  \beta_k^{N,m}(\omega)
\end{split}     \end{equation}
for every
$ \omega \in \Omega$,
$ v, w \in \{0,1,\ldots,N\} $ and every
$ N, m \in \mathbb{N} $, where
$ \sgn \colon \mathbb{R} \rightarrow \{-1,1\} $
is given by $ \sgn(x) := 1 $ for every 
$ x \in [0,\infty) $ and $ \sgn(x) := -1 $
for every $ x \in (-\infty, 0) $.
As usual, $\sum_{l=v}^{w-1} \alpha_l^{N,m}(\omega)=0$
for every $v\in\{w,w+1,\ldots,N\}$, $w\in\{0,1,\ldots,N\}$,
$ \omega \in \Omega $
and every
$ N, m \in \mathbb{N} $.
Note that 
$ D_{v,w}^{N,m} \colon \Omega
\rightarrow \mathbb{R} $
only depends on the Brownian motion
$ W^{(m)} \colon [0,T] \times
\Omega \rightarrow \mathbb{R} $
and the initial random variable
$ \xi^{(m)}\colon \Omega \rightarrow 
\mathbb{R} $
for every $ v,w \in \left\{ 0,1, \dots, N \right\}$
and every $ N, m \in \mathbb{N} $.
Therefore,
$D_{v,w}^{N,m}$, $ m \in \mathbb{N} $,
is a sequence of independent random variables
for every $v,w \in \left\{0,1,\dots,N\right\}$
and every $ N \in \mathbb{N} $.
We will show that the Euler approximation 
is dominated by the dominating
process since the last change of sign.
More formally,
let the
$ \mathcal{F} $/$ \mathcal{P}( \{0,1,\ldots,n \} 
) $-measurable mapping 
$ \tau_n^{N} \colon \Omega \rightarrow \{0,1,\ldots,n\} $
be given by
\begin{equation}\label{tau_def}
  \tau_n^{N}(\omega)
  :=
  \max\left(
    \left\{ 0 \right\} 
    \cup
    \left\{
      k \in \left\{ 1,2,\ldots,n \right\}
      \bigg| \,
      \sgn( Y_{k-1}^{N}(\omega) )
      \neq
      \sgn( Y_{k}^{N}(\omega) )
    \right\}
  \right) 
\end{equation}
for every 
$ \omega \in \Omega $,
$ n \in \{0,1,\ldots,N\} $
and every $ N \in \mathbb{N} $.
The random time 
$\tau_n^N \colon \Omega \rightarrow
\left\{ 0,1, \dots, n \right\} $ is the last time of a change of sign of $ Y^N_k $, 
$ k \in \left\{ 0, 1, \dots, n \right\} $,
for every 
$ n \in \left\{ 0, 1, \dots, N \right\} $
and every
$ N \in \mathbb{N} $.
%
%
%
%
%
%
%
%
%
%
%
%
Lemma~\ref{lemm_1} below shows that 
$|Y_n^{N}|$ is bounded by $D_{\tau_n^N,n}^{N,1}$
on a certain event $\Omega_{N,n}$
for every $ n \in \left\{ 0,1, \dots, N \right\} $
and every $ N \in \mathbb{N} $.
Next we define these events 
$\Omega_{N,n} $, $ n \in \left\{ 0,1, \dots, N \right\}$,
$ N \in \mathbb{N} $, such that the drift alone cannot 
cause a change of sign and such that the 
increment of the Brownian motion is not
extraordinary large.
Let the real number
$ r_N \in [0, \infty) $
be given by
\begin{equation}\label{eq:defrN}
  r_N
  :=
  \min\left(
    \frac{N^{\frac{1}{4}}}{L},
    \left[
      \max\left( 0,
      \frac{N}{
      T \left(
        \sup_{s \in [-1,1]} \left| \mu'(s) \right| 
        + 3L
      \right)
      } 
      - 1
      \right)
    \right]^{\frac{1}{\left(\delta-1\right)}}
  \right) 
\end{equation}
for every $ N \in \mathbb{N} $.
Now define sets 
$ \Omega_{N,n},
\Omega^m_N,
 \Omega_N 
\in \mathcal{F} $
by
\begin{equation}\label{omega_defvw}
  \Omega_{N,n} := 
  \left\{
    \omega \in \Omega \ \bigg| \
    \sup_{ v,w \in \{ 0, 1, \ldots, n\}} 
    \left| D_{v,w}^{N,1} ( \omega ) \right|
    \leq r_N,
    \sup_{ k \in \{ 0, 1, \ldots, n-1 \}}
    \left|
      W_{t_{k+1}^N} ( \omega ) -
      W_{t_{k}^N} ( \omega )
    \right| \leq N^{-\frac{1}{4}}
  \right\}
\end{equation}
by
\begin{equation}  \begin{split}
  \Omega^m_{N} := 
  \bigg\{
    \omega \in \Omega \ \bigg| \
    \sup_{ 
      v, w \in \left\{ 0,1, \dots, N \right\} 
    } 
    \left| D_{v,w}^{N,m} ( \omega ) \right|
    \leq r_N,	
    \sup_{ n \in \{ 0, 1, \ldots, N-1 \}}
    \left|
      W^{(m)}_{t_{n+1}^N} ( \omega ) -
      W^{(m)}_{t_{n}^N} ( \omega )
    \right| \leq N^{-\frac{1}{4}}
  \bigg\}
\end{split}     \end{equation}
and by
\begin{equation}\label{omega_def}
  \Omega_N 
  := \Omega_{N,N}
  = \Omega_N^1
\end{equation}
for every
$ n \in \left\{ 0, 1, \dots, N \right\} $
and every
$ N, m \in \mathbb{N}$.
Finally, define
$ \tilde{\Omega} \in \mathcal{F} $
by
\begin{equation}  \begin{split}\label{deftildeomega}
  \tilde{\Omega}
  &:=
  \left(
    \bigcup_{ N \in \mathbb{N} }
    \bigcap^{\infty}_{ M = N }
    \bigcap^{M^2}_{m=1}
    \Omega^m_M 
  \right) 
  \\
  &\quad\qquad
  \bigcap
  \left( 
    \bigcap_{ 
      \varepsilon > 0 
    }
    \left\{ \!
      \omega \in \Omega 
      \bigg|
      \sup_{ N \in \mathbb{N} } \!
      \frac{ 
      \left|      
        \sum^{ N^2 }_{ m=1 } \!
        \left(
          \mathbbm{1}_{ \Omega_N^m }(\omega)
          \!\cdot\!
          f( 
            Y^{N,m}_N(\omega)
          )
          -
          \mathbb{E}\left[ 
            \mathbbm{1}_{ \Omega_N }
            f( 
              Y^{N}_N
            )
          \right]
        \right)
      \right| 
      }{
        N^{ \left( 1 + \varepsilon \right) }
      }
      < \infty \!
    \right\} 
  \right) .
\end{split}     \end{equation}
Note that $ \tilde{\Omega} $ is indeed
in $ \mathcal{F} $.
Moreover, we write
$ \left\| Z \right\|_{L^p}
:= \left( \mathbb{E}\left[ \left| Z \right|^p
\right] \right)^{ \frac{1}{p} }
\in [0,\infty] $
for all $ p \in [1,\infty) $
and all $ \mathcal{F} $/$\mathcal{B}(\mathbb{R})$-measurable
mappings $ Z\colon \Omega \rightarrow \mathbb{R} $.
Our proof of Theorem \ref{thm:main_result}
uses the following lemmas.
%
%
\begin{lemma}[Burkholder-Davis-Gundy
inequality]\label{discretedavis}
Let $N \in \mathbb{N}$ and let 
$ Z_1, \dots, Z_N
: \Omega \rightarrow \mathbb{R} $
be $ \mathcal{F} $/$ \mathcal{B}(\mathbb{R})
$-measurable mappings with $ \mathbb{E}| Z_n |^2 < \infty $ for all
$ n \in \{1, \ldots, N \} $ and with
$ \mathbb{E}\left[ Z_{n+1} | Z_1, \dots, Z_n \right] = 0 $
for all $ n \in \{1, \ldots, N-1 \} $. 
Then we obtain 
\begin{equation}  \label{eq:bdg}
    \left\| Z_1 + \ldots + Z_N 
    \right\|_{L^p}
  \leq
  K_p \cdot
  \left(
    \left\| Z_1 \right\|_{L^p}^2
    +
    \ldots
    +
    \left\| Z_N \right\|_{L^p}^2
  \right)^{ \frac{1}{2} }
\end{equation}
for every $ p \in [2, \infty ) $, where 
$ K_p $, $p \in [2, \infty) $, are universal constants.
\end{lemma}
%
%
%
%
%
%
%
%
The following lemma is the key result
in our proof of Theorem~\ref{thm:main_result}.
\begin{lemma}[Dominator Lemma]\label{lemm_1}
Let $ Y_n^N \colon \Omega \rightarrow \mathbb{R} $, 
$ D_{n,m}^{N,1} 
: \Omega \rightarrow \mathbb{R} $, $ \tau_n^N \colon \Omega \rightarrow \mathbb{R} $ and $ \Omega_{N,n} 
\in \mathcal{F} $ for $ n, m \in \left\{ 0, 1, \dots, N \right\}$
and $ N \in \mathbb{N} $ be given by \eqref{y_def}, \eqref{d_def}, \eqref{tau_def} and 
\eqref{omega_defvw}. Then we have     
\begin{equation}\label{lemmprof1}
  \left| Y_n^N(\omega) \right| 
  \leq
  D_{\tau_n^N(\omega),n}^{N,1}(\omega)
\end{equation}
for every $ \omega \in \Omega_{N,n} $, 
$ n \in \left\{ 0,1,\ldots,N \right\} $ 
and every
$ N \in \mathbb{N} $.
\end{lemma}
\noindent
The domination~\eqref{lemmprof1} might not 
look helpful at first view
since $D_{\tau_n^N,n}^{N,1}$ depends on the 
Euler approximation and since
$\tau_{n}^N$ is in general not a stopping
time for $ n \in \left\{ 1, \dots, N \right\}$
and $ N \in \mathbb{N} $.
However, the dependence of the dominating process 
on the Euler approximation
can be controlled as $\tilde{\sigma}$ 
is bounded and the dependence of
$ D^{ N , 1 }_{ \tau_n^N, n } $ on 
$\tau_{n}^N$ 
for all $ n \in \left\{ 0,1,\dots,
N \right\} $ and all $ N \in \mathbb{N}$
is no problem as $D_{v,w}^{N,1}$
can be controlled uniformly in
$v, w \in \left\{ 0,1,\dots,N\right\}$
and $ N \in \mathbb{N}$.
This is subject of the
following lemma.
%
%
%
\begin{lemma}[Uniformly bounded 
absolute moments of
the dominator]\label{lem_sup}
Let 
$ D^N_{n,m} 
\colon \Omega \rightarrow [0, \infty) $
for 
$n, m \in \left\{0,1,\dots,N\right\}$
and $ N \in \mathbb{N} $ 
be given by \eqref{d_def}. 
Then we have
\begin{equation}\label{dnmoment}
  \sup_{N \in \mathbb{N}} \,
  \mathbb{E}\!
  \left[
    \sup_{ v,w \in 
    \left\{0,1,\dots,N\right\} }
    \left| D^{N,1}_{v,w} \right|^p
  \right]
  < 
  \infty
\end{equation}
for all $ p \in [1,\infty) $.
\end{lemma}
From Lemma~\ref{lemm_1}
and from Lemma~\ref{lem_sup}, we immediately
conclude that the restricted absolute 
moments of the Euler approximation
are uniformly bound\-ed.
\begin{corollary}[Bounded moments of the Euler approximation] \label{c:mom_Euler}
  Let the Euler approximation 
  $Y_n^N\colon\Omega\to\mathbb{R}$
  for $n \in \left\{0,1,\dots,N\right\}$ 
  and $ N \in \mathbb{N} $ 
  be given by \eqref{y_def}. Then we have
  \begin{equation}
    \sup_{N \in \mathbb{N}}
    \sup_{ n \in \{0,1,\ldots,N\} }
    \mathbb{E}
    \bigg[
      \mathbbm{1}_{ { \Omega }_{N,n} } 
        \left| Y_n^N\right|^p
    \bigg]
    < 
    \infty
  \end{equation}
  for every $p \in [1,\infty) $.
\end{corollary}
%
%
%
%
%
%
%
%
%
%
Next we estimate the probability of
$ \Omega_N $ for every large $ N \in \mathbb{N}$.

%
%
\begin{lemma}[Full probability]\label{lemm_omega}
Let $ \Omega_N \in \mathcal{F} $ 
for $ N \in \mathbb{N} $
and $ \tilde{\Omega} \in \mathcal{F} $ 
be given by 
\eqref{omega_def}
and 
\eqref{deftildeomega}.
Then we have that
\begin{equation}
  \sup_{ N \in \mathbb{N} }
  \left(
    N^{4} \cdot
    \mathbb{P}
    \Big[ (\Omega_N)^c \Big]
  \right)
  < \infty
  \qquad
  \text{ and }
  \qquad
  \mathbb{P}
  \Big[ \tilde{\Omega} \Big]
  = 1 
\end{equation}
holds.
\end{lemma}
%
%
%
%
Using the above lemmas, we establish the following
modification of numerical weak convergence.
\begin{lemma}[Modified weak convergence]\label{lemm_ew}
Let $ X \colon [0,T] \times \Omega \rightarrow \mathbb{R} $, 
$ \Omega_N \in \mathcal{F} $
and
$ Y_n^N \colon \Omega \rightarrow \mathbb{R} $ 
for $ n \in \{0,1,\ldots,N\} $ and $ N \in \mathbb{N} $
be given by
\eqref{x_def},
\eqref{omega_def}
and
\eqref{y_def}.
Then we obtain that
\begin{equation}
  \sup_{ N \in \mathbb{N} }
  \left(
    N
    \cdot
    \left|
      \mathbb{E}\bigg[ \mathbbm{1}_{\Omega_N} 
        \cdot f(X_T) \bigg]
      - \mathbb{E}\bigg[ \mathbbm{1}_{\Omega_N} 
        \cdot f(Y_N^N) \bigg]
    \right|
  \right) < \infty
\end{equation}
holds.
\end{lemma}
While the proof
of Lemma~\ref{discretedavis} 
can be found in 
Theorem~6.3.10 in \cite{Stroock},
the proofs of
Lemma \ref{lemm_1}, Lemma \ref{lem_sup},
Lemma \ref{lemm_omega} and Lemma \ref{lemm_ew} 
are given in 
Sections~\ref{sec:proofthm1}-\ref{sec:modifiedweak}
below.
\subsection{Proof of
Theorem~\ref{thm:main_result}}
\label{sec:proofthm1}
%
%
%
%
%
%
\begin{proof}[Proof of Theorem \ref{thm:main_result}]
Consider the
$ \mathcal{F} / \mathcal{B} \left( [0, \infty) \right) $-mea\-su\-ra\-ble mapping
$ Z \colon \Omega \rightarrow [0, \infty) $ given by
\begin{equation}\label{eq_nr1}
  Z(\omega) :=
  \sup_{N\in\mathbb{N}} 
  \left( \frac{
    \mathbbm{1}_{\tilde{\Omega}}(\omega)}{N} \left(
    \sum_{m=1}^{N^2} \left| f (Y_N^{N,m} (\omega) ) 
    \right| \right)
    \left(1 - \mathbbm{1}_{
      \left( \cap_{M=N}^{\infty} 
      \cap_{m=1}^{M^2} \Omega_{M}^{m}
      \right)
    }   (\omega) \right)
  \right) 
\end{equation}
for every $ \omega \in \Omega $.
Finiteness in~\eqref{eq_nr1} follows from
\begin{equation}
  \lim_{N \rightarrow \infty}
  \left(
      \mathbbm{1}_{ 
        \left( 
        \cap_{M=N}^{\infty}
        \cap^{M^2}_{m=1} 
        \Omega_{M}^m
        \right)
      }(\omega) 
  \right)
  =
      \mathbbm{1}_{ 
        \left(
        \cup_{N \in \mathbb{N}} 
        \cap_{M=N}^{\infty}
        \cap^{M^2}_{m=1} 
        \Omega_{M}^m
        \right)
      }(\omega) 
  \geq
    \mathbbm{1}_{
      \tilde{\Omega}
    }(\omega)	
  = 1
\end{equation}
for every $ \omega \in \tilde{\Omega} $.
Moreover, define $ \mathcal{F} $/$ \mathcal{B}([0,\infty))
$-measurable mappings $ R_{\varepsilon} \colon \Omega \rightarrow [0, \infty) $,
$ \varepsilon \in (0,1) $, 
by
\begin{equation}\label{eq_nr2}
  R_{\varepsilon} (\omega) :=
  \sup_{N\in\mathbb{N}} \left( 
  \mathbbm{1}_{\tilde{\Omega}} (\omega)
  \frac{ \left|
    \sum_{m=1}^{N^2} \left(
      \mathbbm{1}_{\Omega_N^m}(\omega) 
      \cdot f( Y_N^{N,m}(\omega) ) -
      \mathbb{E} \left[ \mathbbm{1}_{\Omega_N} \cdot
      f (Y_N^N) \right] \right) \right|
  }{N^{(1+\varepsilon)}}
  \right)
\end{equation}
for every $ \omega \in \Omega $
and every $ \varepsilon \in (0,1) $.
By definition of $ \tilde{\Omega} $, the
mappings 
$ R_{\varepsilon} \colon \Omega
\rightarrow [0,\infty) 
$, 
$ \varepsilon \in (0,1) $,
are also finite.
Additionally, let the real number $ C \in [0,\infty) $ 
be given by
\begin{equation}  \begin{split}\label{eq_nr3}
  C :=
  \sup_{ N \in \mathbb{N} }
  \left(
    N
    \cdot
    \left|
      \mathbb{E}\bigg[ \mathbbm{1}_{\Omega_N} 
        \cdot f(X_T) \bigg]
      - \mathbb{E}\bigg[ \mathbbm{1}_{\Omega_N} 
        \cdot f(Y_N^N) \bigg]
    \right|
  \right)
  +
  \sup_{ N \in \mathbb{N} }
  \left(
    N^2
    \cdot
    \mathbb{P}\bigg[
      \left( \Omega_N \right)^c
    \bigg]
  \right) ,
\end{split}     \end{equation}
which is finite due to Lemma \ref{lemm_omega} and Lemma \ref{lemm_ew}.
Moreover, we have
\begin{align*}
  &\left| \mathbb{E}\bigg[ f (X_T) \bigg] -
    \frac{1}{N^2} \left(
      \sum_{m=1}^{N^2} f \left(Y_N^{N,m} (\omega) \right)
    \right)
  \right| 
  \\ & \leq \left|
    \mathbb{E}\bigg[ f (X_T) \bigg] -
    \mathbb{E}\bigg[ \mathbbm{1}_{\Omega_N} \cdot f (X_T) \bigg]
  \right|
  + \left|
    \mathbb{E}\bigg[ \mathbbm{1}_{\Omega_N} \cdot f (X_T) \bigg] -
    \mathbb{E}\bigg[ \mathbbm{1}_{\Omega_N} \cdot f (Y_N^N) \bigg]
  \right|
  \\ & \quad + \left|
    \mathbb{E}\bigg[\mathbbm{1}_{\Omega_N} \cdot f (Y_N^N)\bigg] -
    \frac{1}{N^2} \left(
      \sum_{m=1}^{N^2} f \left(Y_N^{N,m} (\omega) \right) \cdot \mathbbm{1}_{\Omega_N^m} (\omega)
    \right)
  \right|
  \\ & \quad + \left|
    \frac{1}{N^2} \left(
      \sum_{m=1}^{N^2} f \left(Y_N^{N,m} (\omega) \right) \cdot \mathbbm{1}_{\Omega_N^m} (\omega)
    \right) -
    \frac{1}{N^2} \left(
      \sum_{m=1}^{N^2} f \left(Y_N^{N,m} (\omega) \right)
    \right)
  \right|
\end{align*}
for every $ \omega \in \tilde{\Omega} $ and every
$ N \in \mathbb{N} $.
Using~\eqref{eq_nr3}, we then obtain
\begin{align*}
  &\left| \mathbb{E}\bigg[ f (X_T) \bigg] -
    \frac{1}{N^2} \left(
      \sum_{m=1}^{N^2} f \left(Y_N^{N,m} (\omega) \right)
    \right)
  \right| 
  \\ & \leq \left|
    \mathbb{E}\bigg[ \bigg(1 - \mathbbm{1}_{\Omega_N} \bigg) \cdot f (X_T) \bigg]
  \right| + C \cdot \frac{1}{N} 
  \\ & \quad + \frac{1}{N^2} \left|
    \sum_{m=1}^{N^2} \left(
      f \left( Y_N^{N,m} \left(\omega\right) \right)
    \cdot \mathbbm{1}_{\Omega_N^m} \left(\omega\right) - 
    \mathbb{E}\bigg[ 
      \mathbbm{1}_{\Omega_N} \cdot f \left( Y_N^{N} \right)
    \bigg] \right)
  \right|
  \\ & \quad + \frac{1}{N^2} \left(
    \sum_{m=1}^{N^2} \left| f \left( Y_N^{N,m} (\omega) \right) \right| \cdot \left( 1-\mathbbm{1}_{\Omega_N^m} (\omega) \right)
  \right)
\end{align*}
and
\begin{equation}  \begin{split}
  &\left| \mathbb{E}\bigg[ f (X_T) \bigg] -
    \frac{1}{N^2} \left(
      \sum_{m=1}^{N^2} f \left(Y_N^{N,m} (\omega) \right)
    \right)
  \right| 
  \\ & \leq \left|
    \mathbb{E}\bigg[ \mathbbm{1}_{(\Omega_N)^c} \cdot f (X_T) \bigg]
  \right| + C \cdot \frac{1}{N} 
  \\ & \quad + \frac{N^{\left( \varepsilon - 1 \right) }}{N^{\left( 1 + \varepsilon \right) }} \left|
    \sum_{m=1}^{N^2} \left( \mathbbm{1}_{\Omega_N^m} \left(\omega\right) \cdot 
      f \left( Y_N^{N,m} \left(\omega\right) \right) - 
    \mathbb{E}\bigg[ 
      \mathbbm{1}_{\Omega_N} \cdot f \left( Y_N^{N} \right)
    \bigg] \right)
  \right|
  \\ & \quad + \frac{1}{N^2} \left(
    \sum_{m=1}^{N^2} \left| f \left( Y_N^{N,m} (\omega) \right) \right|
  \right) \left(
    \sup_{m \in \left\{ 1,2,\ldots,N^2 \right\} } 
  \left( 1-\mathbbm{1}_{\Omega_N^m} (\omega) \right)
  \right)
\end{split}     \end{equation}
for every $ \omega \in \tilde{\Omega} $,
$ N \in \mathbb{N} $ and every $ \varepsilon \in (0,1) $.
Hence, we get
\begin{align*}
  &\left| \mathbb{E}\bigg[ f (X_T) \bigg] -
    \frac{1}{N^2} \left(
      \sum_{m=1}^{N^2} f \left(Y_N^{N,m} (\omega) \right)
    \right)
  \right|
  \\ & \leq
    \mathbb{E}\bigg[ \mathbbm{1}_{(\Omega_N)^c} \cdot \left| f(X_T) \right| \bigg]  + C \cdot \frac{1}{N}
  \\ & \quad + \left( \frac{\left|
    \sum_{m=1}^{N^2} \left( \mathbbm{1}_{\Omega_N^m} \left(\omega\right) \cdot 
      f \left( Y_N^{N,m} \left(\omega\right) \right) - 
    \mathbb{E}\left[ 
      \mathbbm{1}_{\Omega_N} \cdot f \left( Y_N^{N} \right)
    \right] \right)
    \right|}{N^{\left( 1 + \varepsilon \right)}}
  \right) \cdot N^{\left( \varepsilon - 1 \right) }
  \\ & \quad + \frac{1}{N^2} \left(
    \sum_{m=1}^{N^2} \left| f \left( Y_N^{N,m} (\omega) \right) \right|
  \right) \left(
    \sup_{m \in \left\{ 1,2,\ldots,N^2 \right\} } 
  \left( 1-\mathbbm{1}_{\Omega_N^m} (\omega) \right)
  \right)
\end{align*}
and 
\begin{equation}  \begin{split}
  \lefteqn{
  \left| \mathbb{E}\bigg[ f (X_T) \bigg] -
    \frac{1}{N^2} \left(
      \sum_{m=1}^{N^2} f \left(Y_N^{N,m} (\omega) \right)
    \right)
  \right|
  }
  \\ & \leq \left(
    \mathbb{E}\bigg[ \mathbbm{1}_{(\Omega_N)^c} \bigg]
  \right)^{\frac{1}{2}} \cdot \left(
    \mathbb{E}\bigg[ \big| f(X_T) \big|^2 \bigg]
  \right)^{\frac{1}{2}} + C \cdot \frac{1}{N} + 
  R_{\varepsilon}(\omega)
   \cdot N^{\left( \varepsilon - 1 \right) }
  \\ & \quad + \frac{1}{N^2} \left(
    \sum_{m=1}^{N^2} \left| f\left( 
      Y_N^{N,m}(\omega) \right) \right|
  \right) \left(
    1 - 
    \inf_{m \in \left\{ 1,2,\ldots,N^2 \right\} } 
    \mathbbm{1}_{\Omega_N^m} (\omega)
  \right)
\end{split}     \end{equation}
for every $ \omega \in \tilde{\Omega} $,
$ N \in \mathbb{N} $ and every $ \varepsilon \in (0,1) $
by using H\"older's inequality and 
the definition~\eqref{eq_nr2}
of 
$R_\varepsilon\colon\Omega\to[0,\infty)$,
$ \varepsilon \in (0,1) $.
Therefore, the polynomial growth assumption~\eqref{eq:polynomial_growth}
implies
\begin{align*}
  &\left| \mathbb{E}\bigg[ f (X_T) \bigg] -
    \frac{1}{N^2} \left(
      \sum_{m=1}^{N^2} f \left(Y_N^{N,m} (\omega) \right)
    \right)
  \right|
  \\ & \leq \left(
    \mathbb{P}\bigg[ (\Omega_N)^c \bigg]
  \right)^{\frac{1}{2}} \cdot \left(
    \mathbb{E}\bigg[ L^2 \left( 1+ |X_T|^{\delta} \right)^2 \bigg]
  \right)^{\frac{1}{2}} + C \cdot \frac{1}{N} + 
  R_{\varepsilon}(\omega)
   \cdot N^{\left( \varepsilon - 1 \right) }
  \\ & \quad + \frac{1}{N^2} \left(
    \sum_{m=1}^{N^2} \left| f \left( Y_N^{N,m} (\omega) \right) \right|
  \right) \left(
    1 - \mathbbm{1}_{
      \left( \cap_{m=1}^{N^2} \Omega_N^m 
      \right) 
    }(\omega)
  \right)
  \\ & \leq \left(
    \mathbb{P}\bigg[ (\Omega_N)^c \bigg]
  \right)^{\frac{1}{2}} \cdot L \cdot \left(
    \mathbb{E}\bigg[ 2 \left( 1+ |X_T|^{2\delta} \right) \bigg]
  \right)^{\frac{1}{2}} 
  + C \cdot \frac{1}{N} 
  + 
  R_{\varepsilon}(\omega)
   \cdot N^{\left( \varepsilon - 1 \right) }
  \\ & \quad + \left( \frac{1}{N} \left(
    \sum_{m=1}^{N^2} \left| f \left( Y_N^{N,m} (\omega) \right) \right| \right)
    \left(
    1 - \mathbbm{1}_{\left( \cap_{m=1}^{N^2} \Omega_N^m \right)} (\omega)
    \right)
  \right) \cdot N^{-1}
\end{align*}
for every $ \omega \in \tilde{\Omega} $,
$ N \in \mathbb{N} $ and every $ \varepsilon \in (0,1) $.
Furthermore, we have
\begin{align*}
  &\left| \mathbb{E}\bigg[ f (X_T) \bigg] -
    \frac{1}{N^2} \left(
      \sum_{m=1}^{N^2} f \left(Y_N^{N,m} (\omega) \right)
    \right)
  \right|
  \\ & \leq \left(
    \mathbb{P}\bigg[ (\Omega_N)^c \bigg]
    \right)^{\frac{1}{2}} \cdot 2L \cdot \left(
    1 + \mathbb{E}\bigg[ |X_T|^{2\delta} \bigg]
    \right)^{\frac{1}{2}} + C \cdot \frac{1}{N} + 
    R_{\varepsilon}(\omega)
    \cdot N^{\left( \varepsilon - 1 \right) }
  \\ & \quad + \left( \frac{1}{N} \left(
    \sum_{m=1}^{N^2} \left| f \left( Y_N^{N,m} (\omega) \right) \right| \right)
    \left(
    1 - \mathbbm{1}_{\left( \cap_{M=N}^{\infty} \cap_{m=1}^{M^2} \Omega_M^m \right)} (\omega)
    \right)
  \right) \cdot N^{-1}
\end{align*}
for every $ \omega \in \tilde{\Omega} $,
$ N \in \mathbb{N} $ and 
every $ \varepsilon \in (0,1) $.
Using the definition~\eqref{eq_nr1} of 
$Z\colon\Omega\to[0,\infty)$
then yields
\begin{align*}
  &\left| \mathbb{E}\bigg[ f (X_T) \bigg] -
    \frac{1}{N^2} \left(
      \sum_{m=1}^{N^2} f \left(Y_N^{N,m} (\omega) \right)
    \right)
  \right|
  \\ & \leq 2L \cdot \left(
    \mathbb{P}\bigg[ (\Omega_N)^c \bigg]
  \right)^{\frac{1}{2}} \cdot \left(
    1 + \mathbb{E}\bigg[ |X_T|^{2\delta} \bigg]
  \right)^{\frac{1}{2}} + C \cdot \frac{1}{N} + 
  R_{\varepsilon}(\omega)
   \cdot N^{\left( \varepsilon - 1 \right) }
  + Z (\omega) \cdot N^{-1}
  \\ & \leq 2L \cdot \left(
    \mathbb{P}\bigg[ (\Omega_N)^c \bigg]
  \right)^{\frac{1}{2}} \cdot \left(
    1 + \mathbb{E}\bigg[ |X_T|^{2\delta} \bigg]
  \right)^{\frac{1}{2}} + C \cdot N^{\left( \varepsilon - 1 \right) } + 
  R_{\varepsilon}(\omega)
   \cdot N^{\left( \varepsilon - 1 \right) }
  + Z (\omega) \cdot N^{(\varepsilon - 1)}
\end{align*}
and finally
\begin{align*}
  &\left| \mathbb{E}\bigg[ f (X_T) \bigg] -
    \frac{1}{N^2} \left(
      \sum_{m=1}^{N^2} f \left(Y_N^{N,m} (\omega) \right)
    \right)
  \right|
  \\ 
  &\leq 
  2 L \cdot \left( \mathbb{P} \bigg[
    \left( \Omega_N \right)^c 
  \bigg] \right)^{\frac{1}{2}} \cdot
  \left( 1 + \mathbb{E} \bigg[
    \left| X_T \right|^{2\delta}
  \bigg] \right)
  + \bigg( C + R_{\varepsilon}(\omega) + Z(\omega) \bigg) \cdot N^{(\varepsilon-1)}
\\
  &\leq 
  2 L \sqrt{C} N^{-1}
  \left(
    1 + 
    \mathbb{E}\left[ \left| X_T \right|^{2 \delta} \right]
  \right)
  + \left(
    C + R_{\varepsilon}(\omega) + Z(\omega)
  \right)
  \cdot N^{(\varepsilon - 1)}
\\
  &\leq
  \left(
    2 L \sqrt{C}
    \left(
      1 + 
      \mathbb{E}\left[ \left| X_T \right|^{2 \delta} \right]
    \right)
    + 
    C + R_{\varepsilon}(\omega) + Z(\omega)
  \right)
  \cdot N^{(\varepsilon - 1)}
\end{align*}
for every $ \omega \in \tilde{\Omega} $,
$ N \in \mathbb{N} $ and 
every $ \varepsilon \in (0,1) $ due to the
definition~\eqref{eq_nr3} of $C$.
The right-hand side is finite
according to Theorem 2.6.4 
in \cite{GS72}.
This completes the proof
of Theorem \ref{thm:main_result}.
\end{proof}
%
%
\subsection{Proof of Lemma \ref{lemm_1}}
\begin{proof}[Proof of Lemma \ref{lemm_1}]
Roughly speaking, $r_N$ is chosen such that the
drift term alone cannot change the sign of the
$N$-th Euler approximation
as long as the $N$-th Euler approximation is bounded by 
$r_N$ where $ N \in \mathbb{N}$.
Now we formalize this and observe that
\begin{align*}
  &x \cdot \left(
    x + \frac{T}{N} \cdot \mu(x) - \frac{T\mu(0)}{N}
  \right)
=
  x^2 + \frac{T}{N} \cdot x \cdot \mu(x) 
  - x \cdot \frac{T\mu(0)}{N}
\\&\geq
  x^2 - \left| \frac{T}{N} \cdot x \cdot \mu(x) 
  - x \cdot \frac{T\mu(0)}{N} \right|
=
  x^2 - \frac{T}{N} \cdot \left|x\right| \cdot \left| \mu(x) - \mu(0) \right|
\\
  &=
  x^2 - \frac{T\left| x \right|}{N} \left(
    \mathbbm{1}_{[-1,1]}(x) \cdot \left| \mu(x) - \mu(0) \right|
    + 
      \mathbbm{1}_{ \mathbb{R}\setminus [-1,1] }(x)
      \cdot
      \left| \mu(x) - \mu(0) \right|
  \right)
\\
  &\geq
  x^2 - \frac{T\left|x\right|}{N} 
  \left(
    \mathbbm{1}_{[-1,1]}(x) \cdot \left| x \right| \cdot \left(
      \sup_{s \in [-1,1]} \left| \mu'(s) \right|
    \right)
    + 
      \mathbbm{1}_{ \mathbb{R}\setminus [-1,1] }(x)
      \cdot
      \left| \mu(x) - \mu(0) \right|
  \right)
\end{align*}
holds for every $ x \in \mathbb{R} $.
Moreover, using that $ \mu \colon \mathbb{R}
\rightarrow \mathbb{R} $ has at most polynomial growth
(see~\eqref{eq:polynomial_growth})
shows
\begin{align*}
  &x \cdot \left(
    x + \frac{T}{N} \cdot \mu(x) - \frac{T\mu(0)}{N}
  \right)
\\
  &\geq
  x^2 - \frac{T\left|x\right|}{N}
  \left(
    \mathbbm{1}_{[-1,1]}(x) \cdot \left| x \right| \cdot \left(
      \sup_{s \in [-1,1]} \left| \mu'(s) \right|
    \right)
    + 
    \mathbbm{1}_{ \mathbb{R}\setminus [-1,1] }(x)
    \cdot
    \left( L\left( 1 + \left| x \right|^{\delta} \right) 
    + L \right)
  \right)
\\
  &\geq
  x^2 - \frac{T\left|x\right|}{N} 
  \left(
    \mathbbm{1}_{[-1,1]}(x) \cdot \left| x \right| \cdot \left(
      \sup_{s \in [-1,1]} \left| \mu'(s) \right|
    \right)
    + 
    \mathbbm{1}_{ \mathbb{R}\setminus [-1,1] }(x)
    \cdot
    L \cdot \left( 2 + \left| x \right|^{\delta} \right)
  \right)
\end{align*}
and hence
\begin{align*}
  &x \cdot \left(
    x + \frac{T}{N} \cdot \mu(x) - \frac{T\mu(0)}{N}
  \right)
\\
  &\geq
  x^2 - \frac{T\left|x\right|}{N} \left(
    \mathbbm{1}_{[-1,1]}(x) \cdot \left| x \right| \cdot \left(
      \sup_{s \in [-1,1]} \left| \mu'(s) \right|
    \right)
    + 
    \mathbbm{1}_{ \mathbb{R}\setminus [-1,1] }(x)
    \cdot
    3L \cdot \left| x \right|^{\delta}
  \right)
\\
  &\geq
  x^2 - \frac{T\left|x\right|}{N} \left(
    \sup_{s \in [-1,1]} \left| \mu'(s) \right| + 3L
  \right)\left(
    \mathbbm{1}_{[-1,1]}(x) \cdot \left| x \right|
    + 
    \mathbbm{1}_{ \mathbb{R}\setminus [-1,1] }(x)
    \cdot
    \left| x \right|^{\delta}
  \right)
\\
  &\geq
  x^2 - \frac{T\left|x\right|}{N} \left(
    \sup_{s \in [-1,1]} \left| \mu'(s) \right| + 3L
  \right)\left(
    \left| x \right|
    + \left| x \right|^{\delta}
  \right)
\end{align*}
for every $ x \in \mathbb{R} $.
This implies that
\begin{equation}  \begin{split}
  x \cdot \left(
    x + \frac{T}{N} \cdot \mu(x) - \frac{T\mu(0)}{N}
  \right)
  &\geq
  x^2 - \frac{T}{N} \left(
    \sup_{s \in [-1,1]} \left| \mu'(s) \right| + 3L
  \right)\left(
    x^2
    + \left| x \right|^{\left(1+\delta\right)}
  \right)
\\
  &=
  x^2 \left(
    1 - \frac{T}{N} \left(
      \sup_{s \in [-1,1]} \left| \mu'(s) \right| + 3L
    \right)\left(
      1
      + \left| x \right|^{\left(\delta-1\right)}
    \right)
  \right)
\\
  &\geq
  x^2 \left(
    1 - \frac{T}{N} \left(
      \sup_{s \in [-1,1]} \left| \mu'(s) \right| + 3L
    \right)\left\{
      1
      + \left( r_N \right)^{\left(\delta-1\right)}
    \right\}
  \right)
\end{split}     \end{equation}
holds 
for every $x\in[-r_N,r_N]$.
Therefore, we finally obtain
\begin{equation}  \begin{split}
  \lefteqn{
  x \cdot \left(
    x + \frac{T}{N} \cdot \mu(x) - \frac{T\mu(0)}{N}
  \right)
  }
  \\
  &\geq
  x^2 \Bigg(
    1 - \frac{T}{N} \left(
      \sup_{s \in [-1,1]} \left| \mu'(s) \right| + 3L
    \right)
  \left\{
      1
      +
      \max\left( 0,
      \frac{N}{
      T \left(
        \sup_{s \in [-1,1]} \left| \mu'(s) \right| 
        + 3L
      \right)
      } 
      - 1
      \right)
    \right\}
  \Bigg)
\\
  &=
  x^2 \left(
    1 - \frac{T}{N} \left(
      \sup_{s \in [-1,1]} \left| \mu'(s) \right| + 3L
    \right)
      \frac{N}{
      T \left(
        \sup_{s \in [-1,1]} \left| \mu'(s) \right| 
        + 3L
      \right)
      } 
  \right)
\\
  &=
  x^2 \left(
    1 - 1
  \right)
  = 0.
\end{split}     \end{equation}
for every $x\in[-r_N,r_N]$.
Hence, we have shown that
\begin{equation}\label{eq:posflowRN}
    x + \frac{T}{N} \cdot \mu(x) - \frac{T\mu(0)}{N}
  \geq 0
\end{equation}
holds for all $ x \in [0, r_N] $ and that
\begin{equation}\label{eq:negflowRN}
    x + \frac{T}{N} \cdot \mu(x) - \frac{T\mu(0)}{N}
  \leq 0
\end{equation}
holds for all $ x \in [-r_N,0) $.
The estimates~\eqref{eq:posflowRN} and~\eqref{eq:negflowRN} give
us control on the effect of the drift function in one direction.
In addition, we need to bound the growth in the other direction.
The one-sided Lipschitz continuity
of $ \mu $ (see~\eqref{eq:onesided})
implies
$$
  x \cdot \left( \mu(x) - \mu(0) \right)
  \leq L \cdot x^2
$$
and hence
$$
  x \cdot \left( \mu(x) - L \cdot x - \mu(0) \right)
  \leq 0
$$
for every $ x \in \mathbb{R} $.
Therefore, we obtain
\begin{equation}\label{eq:muonesided1}
  \mu(x) - L\cdot x - \mu(0) \leq 0
\end{equation}
for every $ x \in [0,\infty) $
and 
\begin{equation}\label{eq:muonesided2}
  \mu(x) - L\cdot x - \mu(0) \geq 0
\end{equation}
for every $ x \in (-\infty,0) $.
With these inequalities at hand, we now establish~\eqref{lemmprof1}
by induction on $n\in\{0,1,\ldots,N\}$ where $N\in\mathbb{N}$ is
fixed.
First of all, we have $ \tau_0^N(\omega) = 0 $
and therefore
\begin{align*}
  \left| Y_0^N(\omega) \right|
  &=
  \left| \xi(\omega) \right|
  \leq
    T\left|\mu(0)\right| + \left| \sigma(0) \right| 
    + \left|\xi(\omega)\right|  + 1
\\
  &=
  e^{
    \left(
      \sum_{l=0}^{-1} \alpha_l^N(\omega)
    \right)
  } \left(
    T\left|\mu(0)\right| + \left| \sigma(0) \right| 
    + \left|\xi(\omega)\right|  + 1
  \right)
\\
  &=
  e^{
    \left(
      \sum_{l=\tau_0^N(\omega)}^{-1} \alpha_l^N(\omega)
    \right)
  } \left(
    T\left|\mu(0)\right| + \left| \sigma(0) \right| 
    + \left|\xi(\omega)\right|  + 1
  \right)
\\
  &\quad
  + \sum_{k=\tau_0^N(\omega)}^{-1}
  \sgn( Y_k^N(\omega) ) \,
  e^{
    \left(
      \sum_{l=k+1}^{-1} \alpha_l^N(\omega)
    \right)
  } \beta_k^N(\omega)
  = D_{\tau_0^N(\omega),0}^{N,1}(\omega)
\end{align*}
for every $ \omega \in \Omega_{N,0} $, 
which shows \eqref{lemmprof1} 
in the base case $ n=0 $.
Suppose now
that \eqref{lemmprof1} holds for one fixed
$ n \in \{0,1,2,\ldots,N-1\} $. 
Moreover, we fix an arbitrary
$ \omega \in \Omega_{N,n+1} \subset \Omega_{ N, n } $
and we now show \eqref{lemmprof1}
for $ \omega $ and $ n+1 $.
For this
we distinguish between the following 
four different cases.\\
1.) First of all
suppose that $ Y_n^N(\omega) \geq 0 $ and that $ Y_{n+1}^N(\omega) \geq 0 $ holds.
Then 
$ \tau_n^N(\omega) = \tau_{n+1}^N(\omega) $
and 
\begin{align*}
  0
  &\leq
  Y_{n+1}^N(\omega)
\\
  &=
  Y_n^N(\omega) + \frac{T}{N} \cdot \mu
  \left( 
    Y_n^N(\omega) 
  \right)
  + \sigma \left( Y_n^N(\omega) \right) \cdot
  \left( 
    W_{t_{n+1}^N}(\omega) - W_{t_{n}^N}(\omega)
  \right)
\\
  &=
  Y_n^N(\omega) + \frac{T}{N} \cdot
  \left(
    \mu
    \left( 
      Y_n^N(\omega) 
    \right)
    - L \cdot Y_n^N(\omega) -\mu(0)
  \right)
  + \frac{TL}{N} \cdot Y_n^N(\omega)
\\
  &\quad 
  + \frac{T\mu(0)}{N}
  + \left(
    \sigma \left( Y_n^N(\omega) \right) - \sigma(0)
  \right) \cdot \left(
    W_{t_{n+1}^N}(\omega) - W_{t_{n}^N}(\omega)
  \right)
  + \sigma(0) \cdot \left(
    W_{t_{n+1}^N}(\omega) - W_{t_{n}^N}(\omega)
  \right)
\end{align*}
and hence
\begin{multline*}
  0
  \leq
  Y_{n+1}^N(\omega)
 \leq
  Y_n^N(\omega)
  + \frac{TL}{N} \cdot Y_n^N(\omega) 
  + \frac{T\mu(0)}{N}
\\
  + \left(
    \sigma\!\left( \!Y_n^N(\omega)\! \right) - \sigma(0)
  \right) \cdot \left(
    W_{t_{n+1}^N}(\omega) - W_{t_{n}^N}(\omega)
  \right)
  + \sigma(0) \cdot \left(
    W_{t_{n+1}^N}(\omega) - W_{t_{n}^N}(\omega)
  \right)
\end{multline*}
due to \eqref{eq:muonesided1}.
Therefore, using $1+x\leq e^x$ for all
$ x \in \mathbb{R} $ yields
\begin{align*}
  0
  \leq
  Y_{n+1}^N(\omega)
&\leq
  Y_n^N(\omega) \cdot \left(
    1 + \frac{TL}{N} + \tilde{\sigma} \left( Y_n^N(\omega) \right) \cdot \left(
      W_{t_{n+1}^N}(\omega) - W_{t_{n}^N}(\omega)
    \right)
  \right)
\\
  &\quad
  + \left(
    \frac{T\mu(0)}{N} + \sigma(0) \cdot \left(
      W_{t_{n+1}^N}(\omega) - W_{t_{n}^N}(\omega)
    \right)
  \right)
\\
  &\leq
  Y_n^N(\omega) \cdot 
    e^{ \left(
      \frac{TL}{N} + \tilde{\sigma} \left( Y_n^N(\omega) \right) \cdot \left(
        W_{t_{n+1}^N}(\omega) - W_{t_{n}^N}(\omega)
      \right)
    \right) }
\\
  &\quad
  + \left(
    \frac{T\mu(0)}{N} + \sigma(0) \cdot \left(
      W_{t_{n+1}^N}(\omega) - W_{t_{n}^N}(\omega)
    \right)
  \right)
\\
  &=
  Y_n^N(\omega)
  \cdot
  e^{
    \alpha_n^N(\omega)
  } + \beta_n^N(\omega)
\\&=
  \left| Y_n^N(\omega) \right|
  \cdot e^{\alpha_n^N(\omega)}
  +
  \sgn( Y_n^N(\omega) )
  \cdot 
  \beta_n^N(\omega) .
\end{align*}
The induction hypothesis then implies
\begin{align*}
  0
  &\leq
  Y_{n+1}^N(\omega)
\leq
  D_{\tau^N_n(\omega),n}^{N,1}(\omega)
  \cdot
  e^{
    \alpha_n^N(\omega)
  } 
  + 
  \sgn( Y_n^N(\omega) )
  \cdot 
  \beta_n^N(\omega) 
\\
  &=
  e^{\left(
    \sum_{l=\tau^N_n(\omega)}^{n} \alpha_l^N(\omega)
  \right)}
  \left(
    T\left|\mu(0)\right| + \left| \sigma(0) \right| 
    + \left|\xi(\omega)\right|  + 1
  \right)
\\
  &\quad +
  \sum_{k=\tau^N_n(\omega)}^{n-1}
  \sgn( Y_k^N(\omega) ) \,
  e^{\left(
    \sum_{l=k+1}^{n} \alpha_l^N(\omega)
  \right)}
  \beta_k^N(\omega)
  + \sgn( Y_n^N(\omega) ) \cdot \beta_n^N(\omega) 
\end{align*}
and hence
\begin{equation}  \begin{split}
  \left| Y_{n+1}^N(\omega) \right|
&\leq
  e^{\left(
    \sum_{l=\tau^N_n(\omega)}^{n} \alpha_l^N(\omega)
  \right)} \left(
    T\left|\mu(0)\right| + \left| \sigma(0) \right| 
    + \left|\xi(\omega)\right|  + 1
  \right)
\\&\quad 
  +
  \sum_{k=\tau^N_n(\omega)}^n
  \sgn( Y_k^N(\omega) ) \,
  e^{\left(
    \sum_{l=k+1}^{n} \alpha_l^N(\omega)
  \right)}
  \beta_k^N(\omega)
\\&=
  D_{\tau^N_n(\omega),n+1}^{N,1}(\omega)
  =
  D_{\tau^N_{n+1}(\omega),n+1}^{N,1}(\omega) ,
\end{split}     \end{equation}
which shows that \eqref{lemmprof1} 
holds in this case for $ \omega $ and $ n+1 $.\\
2.) Suppose now that $ Y_n^N(\omega) < 0 $
and that $ Y_{n+1}^N(\omega) < 0 $
holds. Then we also have
$ \tau_n^N(\omega) = \tau_{n+1}^N(\omega) $
and
\begin{align*}
  0
  &>
  Y_{n+1}^N(\omega)
\\
  &=
  Y_n^N(\omega) + \frac{T}{N} \cdot \mu
  \left( 
    Y_n^N(\omega) 
  \right)
  + \sigma \left( Y_n^N(\omega) \right) \cdot
  \left( 
    W_{t_{n+1}^N}(\omega) - W_{t_{n}^N}(\omega)
  \right)
\\
  &=
  Y_n^N(\omega) + \frac{T}{N} \cdot
  \left(
    \mu
    \left( 
      Y_n^N(\omega) 
    \right)
    - L \cdot Y_n^N(\omega) -\mu(0)
  \right)
  + \frac{TL}{N} \cdot Y_n^N(\omega)
\\
  &\quad 
  + \frac{T\mu(0)}{N}
  + \left(
    \sigma \left( Y_n^N(\omega) \right) - \sigma(0)
  \right) \cdot \left(
    W_{t_{n+1}^N}(\omega) - W_{t_{n}^N}(\omega)
  \right)
  + \sigma(0) \cdot \left(
    W_{t_{n+1}^N}(\omega) - W_{t_{n}^N}(\omega)
  \right)
\end{align*}
and hence
\begin{multline*}
  0
  >
  Y_{n+1}^N(\omega)
 \geq
  Y_n^N(\omega)
  + \frac{TL}{N} \cdot Y_n^N(\omega) 
  + \frac{T\mu(0)}{N}
\\
  + \left(
    \sigma\!\left( \!Y_n^N(\omega)\! \right) - \sigma(0)
  \right) \cdot \left(
    W_{t_{n+1}^N}(\omega) - W_{t_{n}^N}(\omega)
  \right)
  + \sigma(0) \cdot \left(
    W_{t_{n+1}^N}(\omega) - W_{t_{n}^N}(\omega)
  \right)
\end{multline*}
due to \eqref{eq:muonesided2}.
Therefore, we obtain
\begin{align*}
  0 &> Y^N_{n+1}(\omega)
\geq
  Y_n^N(\omega) \cdot \left(
    1 + \frac{TL}{N} + 
    \tilde{\sigma}\left( Y_n^N(\omega) \right) \cdot \left(
      W_{t_{n+1}^N}(\omega) - W_{t_{n}^N}(\omega)
    \right)
  \right)
\\
  &\quad
  + \left(
    \frac{T\mu(0)}{N} + \sigma(0) \cdot \left(
      W_{t_{n+1}^N}(\omega) - W_{t_{n}^N}(\omega)
    \right)
  \right)
\\
  &\geq
  Y_n^N(\omega) \cdot 
    e^{ \left(
      \frac{TL}{N} + 
      \tilde{\sigma}
      \left( Y_n^N(\omega) \right) \cdot \left(
        W_{t_{n+1}^N}(\omega) - W_{t_{n}^N}(\omega)
      \right)
    \right) }
\\
  &\quad
  + \left(
    \frac{T\mu(0)}{N} + \sigma(0) \cdot \left(
      W_{t_{n+1}^N}(\omega) - W_{t_{n}^N}(\omega)
    \right)
  \right)
\\
  &=
  Y_n^N(\omega)
  \cdot
  e^{
    \alpha_n^N(\omega)
  } + \beta_n^N(\omega) 
\\&=
  -\left(
    \left| Y_n^N(\omega) \right|
    \cdot e^{\alpha_n^N(\omega)}
    + \sgn( Y_n^N(\omega) )
    \cdot \beta_n^N(\omega) 
  \right) .
\end{align*}
Hence, the induction hypothesis yields
\begin{align*}
&\left| Y_{n+1}^N(\omega) \right|
  \leq
  D_{\tau^N_n(\omega),n}^{N,1}(\omega)
  \cdot
  e^{
    \alpha_n^N(\omega)
  } + \sgn( Y_n^N(\omega) ) \cdot \beta_n^N(\omega) 
\\
  &=
  e^{\left(
    \sum_{l=\tau^N_n(\omega)}^{n} \alpha_l^N(\omega)
  \right)}
  \left(
    T\left|\mu(0)\right| + \left| \sigma(0) \right| 
    + \left|\xi(\omega)\right|  + 1
  \right)
\\
  &\quad +
  \sum_{k=\tau^N_n(\omega)}^{n-1}
  \sgn( Y_k^N(\omega) ) \,
  e^{\left(
    \sum_{l=k+1}^{n} \alpha_l^N(\omega)
  \right)}
  \beta_k^N(\omega)
  + \sgn( Y_n^N(\omega) ) \cdot \beta_n^N(\omega) 
\\
  &=
  e^{\left(
    \sum_{l=\tau^N_n(\omega)}^{n} \alpha_l^N(\omega)
  \right)} \left(
    T\left|\mu(0)\right| + \left| \sigma(0) \right| 
    + \left|\xi(\omega)\right|  + 1
  \right)
\\
  &\quad +
  \sum_{k=\tau^N_n(\omega)}^{n}
  \sgn( Y_k^N(\omega) ) \,
  e^{\left(
    \sum_{l=k+1}^{n} \alpha_l^N(\omega)
  \right)}
  \beta_k^N(\omega)
\\
  &=
  D_{\tau^N_n(\omega),n+1}^{N,1}(\omega)
  =
  D_{\tau^N_{n+1}(\omega),n+1}^{N,1}(\omega) ,
\end{align*}
which shows that \eqref{lemmprof1} 
also holds in this case for $ \omega $ and $ n+1 $.\\
3.) In the third case assume that
$ Y_n^N(\omega) \geq 0 $ and that
$ Y_{n+1}^N(\omega) < 0 $ holds.
Then we obtain 
$ \tau_{n+1}^N(\omega) = n + 1 $.
Additionally, note that
\begin{equation}
  \left| Y_n^N(\omega) \right|
  \leq 
  \sup_{k,l \in \{0,1,\ldots,n\} }
  \left| D_{k,l}^{N,1}(\omega) \right|
  \leq 
  r_N
\end{equation}
holds due to the induction hypothesis and
since 
$ \omega \in \Omega_{N,n+1} \subset \Omega_{N,n} $.
Hence, \eqref{eq:posflowRN} yields
\begin{equation}
  Y_n^N(\omega)
  +\frac{T}{N} \mu\left( Y_n^N(\omega) \right)
  -\frac{T\mu(0)}{N}
  \geq 0
\end{equation}
and therefore
\begin{equation}  \begin{split}
  0
  &\geq
  Y_{n+1}^N(\omega)
\\
  &=
  Y_n^N(\omega)
  +\frac{T}{N} \mu\left( Y_n^N(\omega) \right)
  +\sigma\left(Y_n^N(\omega)\right) \cdot \left(
    W_{t_{n+1}^N}(\omega) - W_{t_{n}^N}(\omega)
  \right)
\\
  &=
  \left(
    Y_n^N(\omega)
    +\frac{T}{N} \mu\left( Y_n^N(\omega) \right)
    -\frac{T\mu(0)}{N}
  \right)
  +\frac{T\mu(0)}{N}
  +\sigma\left(Y_n^N(\omega)\right) \cdot \left(
    W_{t_{n+1}^N}(\omega) - W_{t_{n}^N}(\omega)
  \right)
\\
  &\geq
  \frac{T\mu(0)}{N}
  +\sigma\left(Y_n^N(\omega)\right) \cdot \left(
    W_{t_{n+1}^N}(\omega) - W_{t_{n}^N}(\omega)
  \right)
\end{split}     \end{equation}
and
\begin{align*}
  0
  &\geq
  Y_{n+1}^N(\omega)
  \geq
  -\left|\frac{T\mu(0)}{N}\right|
  -\left|\sigma\left(Y_n^N(\omega)\right) \right|
  \cdot \left|
    W_{t_{n+1}^N}(\omega) - W_{t_{n}^N}(\omega)
  \right|
\\  
  &\geq
  -T\left|\mu(0)\right|
  -\left(
    L \cdot \left|Y_n^N(\omega)\right|
    + \left| \sigma(0) \right|
  \right) \cdot
  \left|
    W_{t_{n+1}^N}(\omega) - W_{t_{n}^N}(\omega)
  \right|.
\end{align*}
This implies
\begin{align*}
  0
  &\geq
  Y_{n+1}^N(\omega)
  \geq
  -T\left|\mu(0)\right|
  \!-\!\left(
    L \cdot D_{\tau_n^N(\omega),n}^{N,1} (\omega)
    + \left| \sigma(0) \right|
  \right) 
  \!\cdot\!
  \left|
    W_{t_{n+1}^N}(\omega) - W_{t_{n}^N}(\omega)
  \right|
\\  
  &\geq
  -T\left|\mu(0)\right|
  -\left(
    L \cdot r_N
    + \left| \sigma(0) \right|
  \right)
  \cdot
  \left|
    W_{t_{n+1}^N}(\omega) - W_{t_{n}^N}(\omega)
  \right|
\\  
  &\geq
  -T\left|\mu(0)\right|
  -\left(
    L \cdot \frac{N^{\frac{1}{4}}}{L}
    + \left| \sigma(0) \right|
  \right)
  \cdot
  \left|
    W_{t_{n+1}^N}(\omega) - W_{t_{n}^N}(\omega)
  \right|
\end{align*}
and
\begin{equation}
  0
  \geq
  Y_{n+1}^N(\omega)
  \geq
  -T\left|\mu(0)\right|
  -\left(
    \left| \sigma(0) \right|
    + N^{\frac{1}{4}}
  \right)
  \cdot 
  \frac{1}{N^{\frac{1}{4}}}
\\  
  =
  -T\left|\mu(0)\right|
  -\frac{
    \left| \sigma(0) \right|
  }{
    N^{\frac{1}{4}}
  }
  - 1
  \geq
  -T\left|\mu(0)\right|
  -\left| \sigma(0) \right|
  - 1
\end{equation}
and finally
\begin{align*}
&\left| Y_{n+1}^N(\omega) \right|
  \leq 
  T\left|\mu(0)\right|
  + \left| \sigma(0) \right|
  + 1
  + \left| \xi(\omega) \right|
\\
  &=
  e^{
    \left(\sum_{l=n+1}^{n} \alpha_l^N(\omega)\right)
  } \left(
    T \left|\mu(0)\right|
    + \left| \sigma(0) \right|
    + 1
    + \left| \xi(\omega) \right|
  \right)
  + \sum_{k=n+1}^{n} 
  \sgn( Y_k^N(\omega) ) \,
  e^{
    \left(\sum_{l=k+1}^{n} \alpha_l^N(\omega)\right)
  } 
  \beta_k^N(\omega)
\\
  &=
  e^{
    \left(\sum_{l=\tau_{n+1}^N(\omega)}^{n} \alpha_l^N(\omega)\right)
  } \left(
    T \left|\mu(0)\right|
    + \left| \sigma(0) \right|
    + 1
    + \left| \xi(\omega) \right|
  \right)
  + \sum_{k=\tau_{n+1}^N(\omega)}^{n} 
  \sgn( Y_k^N(\omega) ) \,
  e^{
    \left(\sum_{l=k+1}^n \alpha_l^N(\omega)\right)
  } 
  \beta_k^N(\omega)
\\
&
  = 
  D_{\tau_{n+1}^N(\omega),n+1}^{N,1}(\omega) ,
\end{align*}
which shows that \eqref{lemmprof1} 
also holds in this case for $ \omega $ and $ n + 1 $.\\
4.) In the last case assume that
$ Y_n^N(\omega) < 0 $ and that
$ Y_{n+1}^N(\omega) \geq 0 $ holds.
Then we also obtain 
$ \tau_{n+1}^N(\omega) = n + 1 $.
Note that
\begin{equation}
  \left| Y_n^N(\omega) \right|
  \leq 
  \sup_{k,l \in \{0,1,\ldots,n\} } 
  \left| D_{k,l}^{N,1}(\omega) \right|
  \leq 
  r_N
\end{equation}
holds due to the induction hypothesis and
since 
$ \omega \in \Omega_{N,n+1} \subset \Omega_{N,n} $.
Therefore, \eqref{eq:negflowRN} implies
\begin{equation}
  Y_n^N(\omega)
  +\frac{T}{N} \mu\!\left( Y_n^N(\omega) \right)
  -\frac{T\mu(0)}{N}
  \leq 0
\end{equation}
and hence
\begin{equation}  \begin{split}
  0
  &\leq
  Y_{n+1}^N(\omega)
\\
  &=
  Y_n^N(\omega)
  +\frac{T}{N} \mu\left( Y_n^N(\omega) \right)
  +\sigma\left(Y_n^N(\omega)\right) \cdot \left(
    W_{t_{n+1}^N}(\omega) - W_{t_{n}^N}(\omega)
  \right)
\\
  &=
  \left(
    Y_n^N(\omega)
    +\frac{T}{N} \mu\left( Y_n^N(\omega) \right)
    -\frac{T\mu(0)}{N}
  \right)
  +\frac{T\mu(0)}{N}
  +\sigma\left(Y_n^N(\omega)\right) \cdot \left(
    W_{t_{n+1}^N}(\omega) - W_{t_{n}^N}(\omega)
  \right)
\\
  &\leq
  \frac{T\mu(0)}{N}
  +\sigma\left(Y_n^N(\omega)\right) \cdot \left(
    W_{t_{n+1}^N}(\omega) - W_{t_{n}^N}(\omega)
  \right)
\end{split}     \end{equation}
and
\begin{equation}  \begin{split}
  0
  \leq
  Y_{n+1}^N(\omega)
  &\leq
  \left|\frac{T\mu(0)}{N}\right|
  + \left|\sigma\left(Y_n^N(\omega)\right) \right|
  \cdot \left|
    W_{t_{n+1}^N}(\omega) - W_{t_{n}^N}(\omega)
  \right|
\\  
  &\leq
  T\left|\mu(0)\right|
  + \left(
    L \cdot \left|Y_n^N(\omega)\right|
    + \left| \sigma(0) \right|
  \right)
  \cdot
  \left|
    W_{t_{n+1}^N}(\omega) - W_{t_{n}^N}(\omega)
  \right|
\\  
  &\leq
  T\left|\mu(0)\right|
  +\left(
    L \cdot 
    D_{\tau_n^N(\omega),n}^{N,1}(\omega)
    + \left| \sigma(0) \right|
  \right)
  \cdot
  \left|
    W_{t_{n+1}^N}(\omega) - W_{t_{n}^N}(\omega)
  \right| .
\end{split}     \end{equation}
This shows
\begin{equation}  \begin{split}
\left| Y_{n+1}^N(\omega) \right|
  &\leq
  T\left|\mu(0)\right|
  +\left(
    L \cdot r_N
    + \left| \sigma(0) \right|
  \right)
  \cdot
  \left|
    W_{t_{n+1}^N}(\omega) - W_{t_{n}^N}(\omega)
  \right|
\\  
  &\leq
  T\left|\mu(0)\right|
  +\left(
    L \cdot \frac{N^{\frac{1}{4}}}{L}
    + \left| \sigma(0) \right|
  \right)
  \cdot
  \left|
    W_{t_{n+1}^N}(\omega) - W_{t_{n}^N}(\omega)
  \right|
\\  
  &\leq
  T\left|\mu(0)\right|
  +\left(
    \left| \sigma(0) \right|
    + N^{\frac{1}{4}}
  \right)
  \cdot 
  \frac{1}{N^{\frac{1}{4}}}
  =
  T\left|\mu(0)\right|
  +\frac{
    \left| \sigma(0) \right|
  }{
    N^{\frac{1}{4}}
  }
  + 1
\end{split}     \end{equation}
and finally
\begin{align*}
&\left|
    Y_{n+1}^N(\omega)
  \right|
  \leq 
  T\left|\mu(0)\right|
  + \left| \sigma(0) \right|
  + 1
  + \left| \xi(\omega) \right|
\\
  &=
  e^{
    \left(\sum_{l=n+1}^{n} \alpha_l^N(\omega)\right)
  } \left(
    T \left|\mu(0)\right|
    + \left| \sigma(0) \right|
    + 1
    + \left| \xi(\omega) \right|
  \right)
  + \sum_{k=n+1}^{n}
  \sgn( Y_k^N(\omega) ) \,
  e^{
    \left(\sum_{l=k+1}^{n} \alpha_l^N(\omega)\right)
  } 
  \beta_k^N(\omega)
\\
  &=
  e^{
    \left(\sum_{l=\tau_{n+1}^N(\omega)}^{n} 
    \alpha_l^N(\omega)\right)
  } \left(
    T \left|\mu(0)\right|
    + \left| \sigma(0) \right|
    + 1
    + \left| \xi(\omega) \right|
  \right)
\\
  &\quad
  + \sum_{k=\tau_{n+1}^N(\omega)}^{n} 
  \sgn( Y_k^N(\omega) ) \,
  e^{
    \left(\sum_{l=k+1}^{n} 
    \alpha_l^N(\omega)\right)
  } 
  \beta_k^N(\omega)
  = 
  D_{\tau_{n+1}^N(\omega),n+1}^{N,1}(\omega) ,
\end{align*}
which shows that \eqref{lemmprof1}
holds in this case for $\omega $ and $n+1$
and which finally yields \eqref{lemmprof1}
for every 
$ \omega \in \Omega_{N,n} $, 
$ n \in \left\{ 0, 1, 2, \dots, N \right\} $ 
and every $ N \in \mathbb{N} $
by induction.
\end{proof}
\subsection{Proof of Lemma \ref{lem_sup}}
In order to bound the 
moments of the 
dominating process,
we need to estimate the absolute
moments of a normally 
distributed random
variable.
%
%
\begin{lemma}  \label{l:mom_normal}
  Let 
  $ Y \colon \Omega \rightarrow \mathbb{R} $ 
  be a normally distributed
  $ \mathcal{F} $/$ \mathcal{B}( \mathbb{R} ) $-measurable
  mapping. Then we obtain that
  \begin{equation}\label{lpnormal}
    \left\| Y \right\|_{ L^p }
    \leq
    p
    \left\| Y \right\|_{ L^2 }
  \end{equation}
  holds for every $ p \in [1,\infty) $.
\end{lemma}
\noindent
\begin{proof}[Proof
of Lemma~\ref{l:mom_normal}]
  First of all, H\"older's inequality implies 
  \begin{equation}
    \left\| 
      Y
    \right\|_{ L^p }
    \leq
    \left\|
      Y
    \right\|_{ L^2 }
    \leq
    p
    \left\|
      Y
    \right\|_{ L^2 }
  \end{equation}
for every $ p \in [1,2) $,
  which shows \eqref{lpnormal} in the case $ p \in [1,2) $.
  Denote now the mean of 
  $ Y \colon \Omega \rightarrow \mathbb{R}$ 
  by $ c:=\mathbb{E}\left[ Y \right] \in \mathbb{R}$ 
  and the 
  standard deviation by 
  $ \sigma := \sqrt{ \mathbb{E}\left[ 
    \left( Y - c \right)^2
  \right] } 
  \in [0,\infty) $.
  If $\hat{Y}\colon\Omega\to\mathbb{R}$ is 
  a standard
  normally distributed
  $\mathcal{F}$/$\mathcal{B}(\mathbb{R})$-measurable
  mapping, then
  \begin{equation}  \begin{split}
    \Bigl\|Y\Bigr\|_{ L^p }
    =
    \left\| \sigma \hat{Y} + c \right\|_{ L^p }
  & \leq
    \left\| \sigma \hat{Y} + c \right\|_{ 
      L^{ \left\lceil p \right\rceil } 
    }
    =
    \left(
    \mathbb{E}\bigg[
    \left| \sigma \hat{Y} + c \right|^{ \left\lceil p \right\rceil }
    \bigg]
    \right)^{ \frac{1}{ \left\lceil p \right\rceil } }\\
  &\leq
    \left(
    \mathbb{E}\left[
    \sum_{k=0}^{ \left\lceil p \right\rceil }
    \left( 
      \begin{array}{c}
        \left\lceil p \right\rceil \\ k
      \end{array}
    \right)
    \left| \sigma \right|^k 
    \left| \hat{Y} \right|^k 
    \left|c \right|^{ ( \left\lceil p \right\rceil - k ) } 
    \right]
    \right)^{ \frac{1}{ \left\lceil p \right\rceil } } \\
  &=
    \left(
    \sum_{k=0}^{ \left\lceil p \right\rceil }
    \left( 
      \begin{array}{c}
        \left\lceil p \right\rceil \\ k
      \end{array}
    \right)
    \left| \sigma \right|^k 
    \mathbb{E}
    \left| \hat{Y} \right|^k 
    \left|c \right|^{ ( \left\lceil p \right\rceil - k ) } 
    \right)^{ \frac{1}{ \left\lceil p \right\rceil } }
  \end{split}     \end{equation}
  holds for every $ p \in [2,\infty)$.
  Using
  $\mathbb{E}| \hat{Y} |^k\leq (k-1)^{k/2}$
  for all $ k \in \left\{ 2,3, \dots \right\}$,
  $\mathbb{E}|\hat{Y}|\leq\sqrt{\mathbb{E}(\hat{Y})^2}=1$
  and
  $(|\sigma|+|c|)^2\leq 2(\sigma^2+c^2)$
  yields
  \begin{equation}  \begin{split}
    \left\| Y \right\|_{ L^p }
    &
  \leq
    \left(
    \sum_{k=0}^{ \left\lceil p \right\rceil }
    \left( 
      \begin{array}{c}
        \left\lceil p \right\rceil \\ k
      \end{array}
    \right)
    \left| \sigma \right|^k 
    \left(
      \left\lceil p \right\rceil - 1
    \right)^{ \frac{\left\lceil p \right\rceil}{2} }
    \left|c \right|^{ ( \left\lceil p \right\rceil - k ) } 
    \right)^{ \frac{1}{ \left\lceil p \right\rceil } }
  \\
    &
  =
    \sqrt{
      \left( \left\lceil p \right\rceil - 1 
      \right)
    }
    \left(
      |\sigma| 
      + |c| 
    \right)
  \leq
    \sqrt{
      p
    }
    \left(
      |\sigma| 
      + |c| 
    \right)
  \leq
    p
    \sqrt{
    \left( \sigma^2 + c^2
    \right)
    }
  = 
    p
    \left\| {Y} \right\|_{L^2}
  \end{split}     \end{equation}
  for every $ p \in [2,\infty)$,
  which finally shows~\eqref{lpnormal}.
\end{proof}
\begin{lemma}\label{l:exp_mom_normal2}
  Let 
  $ Y \colon \Omega \rightarrow \mathbb{R} $ 
  be a standard normally distributed
  $ \mathcal{F} $/$ \mathcal{B}( \mathbb{R} ) $-measurable
  mapping. Then we obtain that
  \begin{equation}\label{elpnormal}
    \left\|
      e^{ c Y } - 1
    \right\|_{ L^p }
    \leq
    \left| c \right|
    e^{ (c^2 + 1) p }
  \end{equation}
  holds for every $ c \in \mathbb{R}$
  and every $ p \in [1,\infty) $.
\end{lemma}
\begin{proof}[Proof of
Lemma~\ref{l:exp_mom_normal2}]
We establish \eqref{elpnormal}
in the case $ c \in (0,\infty) $
since the case $ c = 0 $ is trivial
and since the case $ c \in (-\infty,0)$
immediately follows from the
case $ c \in (0,\infty) $.
In order to show \eqref{elpnormal}
in the case $ c \in (0,\infty)$,
note that
\begin{align*}
  \mathbb{E}
  \left[
  \left|
    e^{ c Y } - 1
  \right|^p
  \right]
&=
  \mathbb{E}\left[
    \mathbbm{1}_{ 
      \left\{
        Y \geq 0 
      \right\}
    }
    \, e^{ c p Y }
    \left(
      1 - e^{ - c Y }
    \right)^p
  \right]
  +
  \mathbb{E}\left[
    \mathbbm{1}_{ 
      \left\{
        Y < 0 
      \right\}
    }
    \left(
      1 - e^{ c Y }
    \right)^p
  \right]
\\&\leq
  c^p \cdot
  \mathbb{E}\left[
    \mathbbm{1}_{ 
      \left\{
        Y \geq 0 
      \right\}
    }
    \, e^{ c p Y }
    \left| Y \right|^p
  \right]
  +
  c^p \cdot
  \mathbb{E}\left[
    \mathbbm{1}_{ 
      \left\{
        Y < 0 
      \right\}
    }
    \left| Y \right|^p
  \right]
\\&=
  c^p \cdot
  \mathbb{E}\left[
    \left(
      \mathbbm{1}_{ 
        \left\{
          Y \geq 0 
        \right\}
      }
      \, e^{ c p Y }
      +
      \mathbbm{1}_{ 
        \left\{
          Y < 0 
        \right\}
      }
    \right)
    \left| Y \right|^p
  \right]
\\&\leq
  c^p 
  \left(
    \mathbb{E}\!\left[
      e^{ 2 c p Y }
    \right]
    + \mathbb{P}\!\left[Y<0\right]
  \right)^{ \frac{1}{2} }
  \sqrt{
    \mathbb{E}\left| 
      Y
    \right|^{ 2 p }
  }
=
  c^p 
  \left(
      e^{ 2 c^2 p^2 }
    + \frac{1}{2}
  \right)^{ \frac{1}{2} }
  \sqrt{
    \mathbb{E}\left| 
      Y
    \right|^{ 2 p }
  }
\end{align*}
and 
the estimate
$\|Y\|_{L^{2p}}\leq \sqrt{2p-1}$ 
therefore shows that
\begin{align*}
  \left\|
    e^{ c Y } - 1
  \right\|_{ L^p }
&\leq
  c 
  \left(
      e^{ 2 c^2 p^2 }
    + \frac{1}{2}
  \right)^{ \frac{1}{2 p} }
  \left\| 
    Y
  \right\|_{ L^{ 2 p } }
\leq
  c \,
  2^{ \frac{1}{2 p} }
  e^{ c^2 p }
  \left\| 
    Y
  \right\|_{ L^{ 2 p } }
\leq
  c \cdot
  e^{ (c^2 + 1) p }
\end{align*}
for all $ c \in (0,\infty) $
and all $ p \in [1,\infty) $.
This completes the proof
of Lemma~\ref{l:exp_mom_normal2}.
\end{proof}
\begin{lemma}
\label{l:exp_alpha}
Let 
$ \alpha^N_n : \Omega
\rightarrow \mathbb{R} $
for 
$ n \in \left\{0,1,\dots,N-1\right\}$
and $ N \in \mathbb{N}$
be given by \eqref{eq:def_alpha}.
Then we obtain
\begin{equation}
  \left\| 
    e^{ -\alpha^N_n }
    - 1
  \right\|_{ L^p }
\leq
  L N^{-\frac{1}{2}}
  \left(
    \sqrt{T}
  e^{ \left( 
      L^2 T
      + 1
    \right) p 
  }
  +
    T
  \right)
\end{equation}
for all
$ n \in \left\{ 0, 1,\dots, N \right\}$,
$ N \in \mathbb{N}$ and all
$ p \in [1,\infty) $.
\end{lemma}
\begin{proof}[Proof
of Lemma~\ref{l:exp_alpha}]
The triangle inequality and
Lemma~\ref{l:exp_mom_normal2}
impliy
\begin{align*}
  \left\| 
    e^{ -\alpha^N_n }
    - 1
  \right\|_{ L^p }
&\leq
  \left\| 
    e^{ -\alpha^N_n }
    - e^{ -\frac{ T L }{ N } }
  \right\|_{ L^p }
  +
  \left\| 
    e^{ -\frac{ T L }{ N } }
    - 1
  \right\|_{ L^p }
\\&=
  e^{ -\frac{ T L }{ N } }
  \left\| 
    e^{ 
      -\tilde{ \sigma }( Y^N_n )
      \big( 
        W_{ t^N_{n+1} } - W_{ t^N_n }
      \big)
    }
    - 1
  \right\|_{ L^p }
  +
  \left(
    1 -
    e^{ - \frac{ T L }{ N } }
  \right)
\\&\leq
  L \sqrt{
    \frac{T}{N}
  }
  e^{ \left( 
      L^2 
      \frac{T}{N} + 1
    \right) p 
  }
  +
  \frac{ T L }{ N }
\\&\leq
  L N^{-\frac{1}{2}}
  \left(
    \sqrt{T}
  e^{ \left( 
      L^2 T
      + 1
    \right) p 
  }
  +
    T
  \right)
\end{align*}
for all 
$ n \in 
\left\{ 0, 1, \dots, N-1 \right\} $,
$ N \in \mathbb{N} $
and all $ p \in [1,\infty) $.
This completes the
proof of 
Lemma~\ref{l:exp_alpha}.
\end{proof}
\begin{proof}[Proof of Lemma \ref{lem_sup}]
Let $ C \in (0, \infty) $ 
be a real number satisfying
\begin{equation}
\label{eq:realC}
  5 + L + \delta + T
  +
  \left| \mu( 0 ) \right|
  +
  \left|
    \sigma( 0 )
  \right|
  \leq
  C,
  \qquad
  \left|
    \mu( x )
  \right|
  +
  \left|
    \mu'( x )
  \right|
  \leq
  C\left( 1 + |x|^C \right)
\end{equation}
for all $ x \in \mathbb{R} $.
Such a real number 
$ C < \infty $ indeed
exists since the derivative of $\mu$ is assumed
to grow at most polynomially
according to~\eqref{eq:polynomial_growth}.
Since the exponential function is convex,
we obtain that
\begin{equation}
  \exp
    \left( 
      \sum_{l=0}^{n-1} 
      z \cdot \tilde{\sigma}( Y_l^N ) 
      \cdot \left( W_{t_{l+1}^N} - W_{t_l^N} \right) 
    \right) 
\end{equation}
is a positive submartingale in $ n \in \{ 0, 1, \ldots, N \} $
for every $ z \in \{ -1, 1 \} $ and every $ N \in \mathbb{N} $.
Therefore, Doob's inequality (see e.g. Theorem 11.2 (ii) in \cite{K08}) shows
\begin{equation}  \begin{split}
  &\mathbb{E}\left[
    \left|
      \sup_{n \in \{0,1,\ldots,N\}}
      e^{ 
        \left( 
          \sum_{l=0}^{n-1} 
          z 
          \cdot
          \tilde{\sigma}( Y_l^N ) 
          \cdot 
          \left( W_{t_{l+1}^N} - W_{t_l^N} \right)
        \right)
      }
    \right|^p
  \right]
\\
  &\leq
  \left( \frac{p}{p-1} \right)^p
  \mathbb{E}\left[
    \left|
      e^{ 
        \left( 
          \sum_{l=0}^{N-1} 
          z 
          \cdot
          \tilde{\sigma}( Y_l^N ) 
          \cdot 
          \left( W_{t_{l+1}^N} - W_{t_l^N} \right)
        \right)
      }
    \right|^p
  \right]
\\
  &=
  \left( \frac{p}{p-1} \right)^p
  \mathbb{E}\Bigg[ 
    e^{
      p z 
      \left( 
        \sum_{l=0}^{N-2} 
        \tilde{\sigma}( Y_l^N ) 
        \cdot 
        \left( W_{t_{l+1}^N} - W_{t_l^N} \right)
      \right)
    }
    \mathbb{E}\left[
      e^{
        p z 
        \left(  
          \tilde{\sigma}( Y_{N-1}^N ) 
          \cdot 
          \left( W_{t_{N}^N} - W_{t_{N-1}^N} \right)
        \right)
      }
      \bigg| \mathcal{F}_{t_{N-1}^N}
    \right]
  \Bigg]
\\
  &=
  \left( \frac{p}{p-1} \right)^p
  \mathbb{E}\left[ 
    e^{
      p z 
      \left( 
        \sum_{l=0}^{N-2} 
        \tilde{\sigma}( Y_l^N ) 
        \cdot 
        \left( W_{t_{l+1}^N} - W_{t_l^N} \right)
      \right)
    }
    e^{
      \frac{1}{2} 
      \left( 
        p z \cdot 
        \tilde{\sigma}( Y_{N-1}^N ) 
      \right)^2
      \frac{T}{N}
    }
  \right]
\end{split}     \end{equation}
due to
the moment generating function
$\mathbb{E}[\exp(cY)]=\exp(c^2/2)$, $c\in\mathbb{R}$, of the standard normally distributed
random variable $Y:=\sqrt{\tfrac{N}{T}}\left(W_{t_N^N}-W_{t_{N-1}}^N\right)$
and hence, using $|\tilde{\sigma}(x)|\leq L$
for every $x\in\mathbb{R}$,
\begin{equation}  \begin{split}
  &\left\|
    \sup_{ n \in \{ 0, 1, \ldots, N \} }
    e^{ 
      \left( 
        \sum_{l=0}^{n-1} 
        z 
        \cdot
        \tilde{\sigma}( Y_l^N ) 
        \cdot 
        \left( W_{t_{l+1}^N} - W_{t_l^N} \right)
      \right)
    }
  \right\|_{L^p}
\\
  &\leq
  \left\{
  \left( \frac{p}{p-1} \right)^p
  \mathbb{E}\left[ 
    e^{
      p z 
      \left( 
        \sum_{l=0}^{N-2} 
        \tilde{\sigma}( Y_l^N ) 
        \cdot 
        \left( W_{t_{l+1}^N} - W_{t_l^N} \right)
      \right)
    }
  \right]
  e^{
    \frac{1}{2} 
    p^2 L^2
    \frac{T}{N}
  }
  \right\}^{\frac{1}{p}}
\\&\leq
  \ldots
\leq
  \left\{
  \left( \frac{p}{p-1} \right)^p
  \prod_{l=0}^{N-1}
  e^{
    \frac{1}{2} 
    p^2 L^2 \frac{T}{N}
  }
  \right\}^{\frac{1}{p}}
  =
  \left( \frac{p}{p-1} \right)
  e^{
    \frac{1}{2} 
    p L^2 T
  }
\end{split}     \end{equation}
for every $ p \in (1, \infty) $, $ z \in \{-1,1\} $ 
and every $ N \in \mathbb{N} $.
This implies
\begin{equation}  \begin{split}\label{lpref}
&
  \left\|
    \sup_{n \in \{0,1,\ldots,N\}}
    e^{
      z\left( 
        \sum_{l=0}^{n-1} \alpha_l^N 
      \right)
  }
  \right\|_{L^p}
\\&\leq
  e^{TL}
  \left\|
    \sup_{n \in \{0,1,\ldots,N\}}
    e^{
      \left( 
        \sum_{l=0}^{n-1} z \cdot  
        \tilde{\sigma}( Y_l^N ) \cdot
        \left(
          W_{t_{l+1}^N} - W_{t_{l}^N}
        \right)
      \right)
    }
  \right\|_{L^p}
\\
  &\leq
  e^{ TL }
  \left(
    \frac{ p }{ p - 1 }
  \right)
  e^{ \frac{1}{2} p L^2 T }
  \leq
  2
  e^{ \frac{p}{2} (L^2 + L) T }
\leq
  e^{ \frac{p}{2} (1 + C^2 T + C^2) }  
\leq
  e^{ \frac{p}{2} C^3 } 
\leq
  e^{ \frac{p}{10} C^4 } 
\end{split}     \end{equation}
for every $ p \in [2, \infty) $, $ z \in \{-1,1\} $
and every $ N \in \mathbb{N} $.
Moreover, Lemma~\ref{l:exp_alpha}
shows
\begin{equation}  \begin{split}
\label{eq:alphaest}
  \left\|
    e^{-\alpha^N_n }
    - 1
  \right\|_{ L^p }
&\leq
  L N^{-\frac{1}{2}}
  \left( 
    \sqrt{T} 
    e^{ \left( L^2 T + 1 \right) p }
    + T
  \right)
\\
&\leq
  C^2 N^{-\frac{1}{2}}
  \left( 
    \frac{1}{2} 
    e^{ \left( C^3 + 1 \right) p }
    + 1
  \right)    
\leq
  e^{ 
    \left( C + C^3 + 1 
    \right) p
  } 
  N^{-\frac{1}{2}}
\leq
  \frac{
    e^{ 
      \frac{ p }{ 4 }
      C^4
    }
  }{ 
    \sqrt{N} 
  }
\end{split}     \end{equation}
for all 
$ n \in \left\{ 0,1, \dots, N-1 
\right\} $,
$ N \in \mathbb{N} $
and all $ p \in [2,\infty) $.
With these estimates at hand,
we now bound the $p$-th 
absolute moment
\begin{equation}
 \mathbb{E}
  \left[
    \sup_{ v,w \in 
      \left\{0,1,\dots,N\right\} 
    }
    \left| D^{N,1}_{v,w} \right|^p
  \right]
\end{equation}
of the dominating process
for every $ N \in \mathbb{N} $
and every $ p \in [2,\infty) $.
By definition~\eqref{d_def} and by the triangle inequality
we have
\begin{align*}
  \left\| 
    \sup_{ v,w \in 
       \left\{0,1,\dots,N\right\} 
    }
    \left| D^{N,1}_{v,w} \right|
  \right\|_{L^p}
&\leq
  \left\| 
    \left(
      \sup_{v,w \in \{0,1,\ldots,N\}}
      e^{ 
        \left( 
          \sum_{l=v}^{w-1} 
          \alpha_l^N 
        \right) 
      }
    \right)
    \left(
      C^2 + | \xi |
    \right)
  \right\|_{L^p}
\\
&\quad +
  \left\| 
    \sup_{v,w \in \{0,1,\ldots,N\}}
    \left|
      \sum_{k=v}^{w-1}
      \sgn( Y_k^N ) \,
        e^{ 
          \left( 
            \sum_{l=k+1}^{w-1} 
            \alpha_l^N 
          \right) 
        }
      \beta_k^N
    \right|
  \right\|_{L^p}
\end{align*}
and, using H\"older's inequality,
\begin{align*}
  \left\| 
    \sup_{ v,w \in 
       \left\{0,1,\dots,N\right\} 
    }
    \left| D^{N,1}_{v,w} \right|
  \right\|_{L^p}
&\leq
  \left\| 
    \sup_{0 \leq v \leq w \leq N}
    e^{ 
      \left( 
        \sum_{l=v}^{w-1} \alpha_l^N 
      \right) 
    }
  \right\|_{L^{2p}}
  \left(
    C^2 + \| \xi \|_{ L^{2p} }
  \right)
\\&\quad 
  +
  \left\| 
    \sup_{0 \leq v \leq w \leq N}
    \left|
      \sum_{k=v}^{w-1}
      \sgn( Y_k^N ) \,
      e^{ 
        \left( 
          \sum_{l=k+1}^{w-1} \alpha_l^N 
        \right) 
      }
      \beta_k^N
    \right|
  \right\|_{L^p}
\end{align*}
for all $ N \in \mathbb{N} $
and all $ p \in [2, \infty) $.
Therefore, we obtain
\begin{equation}  \begin{split}
  \left\| 
    \sup_{ v,w \in 
       \left\{0,1,\dots,N\right\} 
    }
    \left| D^{N,1}_{v,w} \right|
  \right\|_{L^p}
&\leq
  \left\| 
    \sup_{0 \leq v \leq w \leq N}
    e^{ \left( \sum_{l=v}^{w-1} \alpha_l^N \right) }
  \right\|_{L^{2p}}
  \left(
    C^2 + \| \xi \|_{L^{2p}}
  \right)
\\
&\quad+
  \left\| 
    \sup_{0 \leq v \leq w \leq N}
    \left(
      e^{ 
        \left( \sum_{l=0}^{w-1}   
          \alpha_l^N 
        \right) 
      }
      \left|
        \sum_{k=v}^{w-1}
        \sgn( Y_k^N ) \,
        e^{ -\left( \sum_{l=0}^{k} \alpha_l^N \right) }
        \beta_k^N
      \right|
    \right)
  \right\|_{L^p}
\end{split}     \end{equation}
and
\begin{equation}  \begin{split}
&
  \left\| 
    \sup_{ v,w \in 
       \left\{0,1,\dots,N\right\} 
    }
    \left| D^{N,1}_{v,w} \right|
  \right\|_{L^p}
\\&\leq
  \left\| 
    \sup_{w \in \{0,1,\ldots, N\}}
    e^{ 
      \left( \sum_{l=0}^{w-1} \alpha_l^N  
      \right) 
    }
  \right\|_{L^{4p}}
  \left\| 
    \sup_{v \in \{0,1,\ldots, N\}}
    e^{ -\left( 
      \sum_{l=0}^{v-1} \alpha_l^N 
    \right) }
  \right\|_{L^{4p}}
  \left(
    C^2 + \| \xi \|_{L^{2p}}
  \right)
\\&\quad
  +
  \left\| 
    \sup_{w \in \{0,1,\ldots, N\} }
    e^{ \left( 
      \sum_{l=0}^{w-1} \alpha_l^N 
    \right) }
  \right\|_{L^{2p}}
  \left\| 
    \sup_{0 \leq v \leq w \leq N}
      \left|
        \sum_{k=v}^{w-1}
        \sgn( Y_k^N ) \,
        e^{ -\left( 
          \sum_{l=0}^{k} \alpha_l^N 
        \right) }
        \beta_k^N
      \right|
  \right\|_{L^{2p}}
\end{split}     \end{equation}
for all $ N \in \mathbb{N} $
and all $ p \in [2, \infty) $.
Inequality~\eqref{lpref} therefore yields
\begin{equation}  \begin{split}
\lefteqn{
  \left\| 
    \sup_{ v,w \in 
       \left\{0,1,\dots,N\right\} 
    }
    \left| D^{N,1}_{v,w} \right|
  \right\|_{L^p}
}
\\
&\leq
  e^{ \frac{4p}{10} C^4 }
  e^{ \frac{4p}{10} C^4 }
  \left(
    C^2 + \| \xi \|_{L^{2p}}
  \right)
  +
  e^{ \frac{2p}{10} C^4 }
  \left\| 
    \sup_{0 \leq v \leq w \leq N}
      \left|
        \sum_{k=v}^{w-1}
        \sgn( Y_k^N ) \,
        e^{ -\left( 
          \sum_{l=0}^{k} \alpha_l^N 
        \right) }
        \beta_k^N
      \right|
  \right\|_{L^{2p}}
\\&
\leq
  e^{ p C^4 }
  \left(
    C^2 + \| \xi \|_{L^{2p}}
  \right)
  +
  2 e^{ \frac{p}{5} C^4 }
  \left\| 
    \sup_{ w \in \{0,1,\ldots,N\} }
      \left|
        \sum_{k=0}^{w-1}
        \sgn( Y_k^N ) \,
        e^{ -\left( 
          \sum_{l=0}^{k} \alpha_l^N 
        \right) }
        \beta_k^N
      \right|
  \right\|_{L^{2p}}
\end{split}     \end{equation}
for all $ N \in \mathbb{N} $
and all $ p \in [2, \infty) $.
By definition of $ \beta^N_n $,
$ n \in \left\{ 0,1,\dots, N-1 \right\}$,
$ N \in \mathbb{N} $,
(see \eqref{eq:def_beta})
we then obtain
\begin{equation*}  \begin{split}
  \left\| 
    \sup_{ v,w \in 
       \left\{0,1,\dots,N\right\} 
    }
    \left| D^{N,1}_{v,w} \right|
  \right\|_{L^p}
&\leq
  e^{ p C^4 }
  (
    C^2 + \| \xi \|_{L^{2p}}
  )
  +
  2 e^{ \frac{p}{5} C^4 }
  \left(
        \sum_{k=0}^{N-1}
  \left\| 
        \sgn( Y_k^N ) \,
        e^{ -\left( 
          \sum_{l=0}^{k} \alpha_l^N 
        \right) }
          \frac{T}{N}
          \mu(0)
    \right\|_{L^{2p}}
  \right)
\\
&\quad
  +
  2 e^{ \frac{p}{5} C^4 }
  \left\| 
    \sup_{ w \in \{0,1,\ldots,N\} }
      \left|
        \sum_{k=0}^{w-1}
        \sgn( Y_k^N ) \,
        e^{ -\left( 
          \sum_{l=0}^{k} \alpha_l^N 
        \right) }
        \sigma(0)
        \left(
          W_{ t^N_{k+1} }
          -
          W_{ t^N_k }
        \right)
      \right|
  \right\|_{L^{2p}}
\end{split}     \end{equation*}
and therefore
\begin{equation*}  \begin{split}
  \left\| 
    \sup_{ v,w \in 
       \left\{0,1,\dots,N\right\} 
    }
    \left| D^{N,1}_{v,w} \right|
  \right\|_{L^p}
&\leq
  e^{ p C^4 }
  (
    C^2 + \| \xi \|_{L^{2p}}
  )
  +
  2 C^2 N^{-1} e^{ \frac{p}{5} C^4 }
  \left(
        \sum_{k=0}^{N-1}
  \left\| 
        e^{ -\left( 
          \sum_{l=0}^{k} \alpha_l^N 
        \right) }
    \right\|_{L^{2p}}
  \right)
\\
&\quad
  +
  2 C e^{ \frac{p}{5} C^4 }
  \left\| 
    \sup_{ w \in \{0,1,\ldots,N\} }
      \left|
        \sum_{k=0}^{w-1}
        \sgn( Y_k^N ) \,
        e^{ -\left( 
          \sum_{l=0}^{k} \alpha_l^N 
        \right) }
        \left(
          W_{ t^N_{k+1} }
          -
          W_{ t^N_k }
        \right)
      \right|
  \right\|_{L^{2p}}
\end{split}     \end{equation*}
for all $ N \in \mathbb{N} $
and all $ p \in [2,\infty) $.
The triangle inequality
and again estimate~\eqref{lpref}
hence yield
\begin{multline*}
  \left\| 
    \sup_{ v,w \in 
       \left\{0,1,\dots,N\right\} 
    }
    \left| D^{N,1}_{v,w} \right|
  \right\|_{L^p}
\leq
  e^{ p C^4 }
  (
    C^2 + \| \xi \|_{L^{2p}}
  )
  +
  2 C^2 e^{ \frac{2 p}{5} C^4 }
\\
  +
  2 C e^{ \frac{p}{5} C^4 }
  \left\| 
    \sup_{ w \in \{0,1,\ldots,N\} }
      \left|
        \sum_{k=0}^{w-1}
        \sgn( Y_k^N ) \,
        e^{ -\left( 
          \sum_{l=0}^{k-1} \alpha_l^N 
        \right) }
        \left( 
          e^{ -\alpha^N_k } - 1
        \right)
        \left(
          W_{ t^N_{k+1} }
          -
          W_{ t^N_k }
        \right)
      \right|
  \right\|_{L^{2p}}
\\
  +
  2 C e^{ \frac{p}{5} C^4 }
  \left\| 
    \sup_{ w \in \{0,1,\ldots,N\} }
      \left|
        \sum_{k=0}^{w-1}
        \sgn( Y_k^N ) \,
        e^{ -\left( 
          \sum_{l=0}^{k-1} \alpha_l^N 
        \right) }
        \left(
          W_{ t^N_{k+1} }
          -
          W_{ t^N_k }
        \right)
      \right|
  \right\|_{L^{2p}}
\end{multline*}
and Doob's inequality (see, e.g., Theorem~11.2~(ii) 
in \cite{K08}) and
Da\-vis-Burk\-hol\-der-Gun\-dy's
inequality~\eqref{eq:bdg} then show
\begin{multline*}
  \left\| 
    \sup_{ v,w \in 
       \left\{0,1,\dots,N\right\} 
    }
    \left| D^{N,1}_{v,w} \right|
  \right\|_{L^p}
\leq
  e^{ 2 p C^4 }
  (
    C^2 + \| \xi \|_{L^{2p}}
  )
\\
  +
  e^{ \frac{p}{4} C^4 }
  \left(
        \sum_{k=0}^{N-1}
  \left\| 
        e^{ -\left( 
          \sum_{l=0}^{k-1} \alpha_l^N 
        \right) }
        \left( 
          e^{ -\alpha^N_k } - 1
        \right)
        \left(
          W_{ t^N_{k+1} }
          -
          W_{ t^N_k }
        \right)
  \right\|_{L^{2p}}
  \right)
\\
  + 
  e^{ \frac{p}{4} C^4 } K_{ 2 p }
  \left(
        \sum_{k=0}^{N-1}
  \left\| 
        e^{ -\left( 
          \sum_{l=0}^{k-1} \alpha_l^N 
        \right) }
        \left(
          W_{ t^N_{k+1} }
          -
          W_{ t^N_k }
        \right)
  \right\|_{L^{2p}}^2
      \right)^{ \! \frac{1}{2} }
\end{multline*}
for all $ N \in \mathbb{N} $
and all $ p \in [2,\infty) $.
H\"{o}lder's inequality thus
gives
\begin{multline*}
  \left\| 
    \sup_{ v,w \in 
       \left\{0,1,\dots,N\right\} 
    }
    \left| D^{N,1}_{v,w} \right|
  \right\|_{L^p}
\leq
  e^{ 2 p C^4 }
  (
    C^2 + \| \xi \|_{L^{2p}}
  )
\\
  +
  e^{ \frac{p}{4} C^4 }
  \left(
        \sum_{k=0}^{N-1}
      \left\| 
        e^{ -\left( 
          \sum_{l=0}^{k-1} \alpha_l^N 
        \right) }
      \right\|_{ L^{ 6 p } }
        \left\|
          e^{ -\alpha^N_k } - 1
        \right\|_{ L^{ 6 p } }
        \left\|
          W_{ \frac{T}{N} }
        \right\|_{ L^{ 6 p } }
  \right)
\\
  +
  e^{ \frac{p}{4} C^4 } K_{ 2 p }
      \left\|
          W_{ \frac{T}{N} }
      \right\|_{L^{4p}}
  \left(
        \sum_{k=0}^{N-1}
      \left\| 
        e^{ -\left( 
          \sum_{l=0}^{k-1} \alpha_l^N 
        \right) }
      \right\|_{ L^{4 p} }^2
      \right)^{ \! \frac{1}{2} }
\end{multline*}
and 
inequality~\eqref{lpref},
inequality~\eqref{eq:alphaest}  
and
Lemma~\ref{l:mom_normal}
finally yield
\begin{equation}  \begin{split}
  \left\| 
    \sup_{ v,w \in 
       \left\{0,1,\dots,N\right\} 
    }
    \left| D^{N,1}_{v,w} \right|
  \right\|_{L^p}
&\leq
  e^{ 2 p C^4 }
  (
    C^2 + \| \xi \|_{L^{2p}}
  )
  +
  e^{ \frac{p}{4} C^4 }
  \left(
        \sum_{k=0}^{N-1}
      e^{\frac{6p}{10}C^4}
      \cdot 
      e^{\frac{6p}{4}C^4}
      \frac{1}{\sqrt{N}}
      \cdot 
      6p
      \sqrt{\frac{T}{N}}
  \right)
 \\
 &\qquad
  +
  e^{ \frac{p}{4} C^4 } K_{ 2 p }
  \cdot 4p \sqrt{\frac{T}{N}}
  \left(
        \sum_{k=0}^{N-1}
      e^{\frac{8p}{10}C^4}
      \right)^{ \! \frac{1}{2} }
\\
&\leq
  e^{ 2 p C^4 }
  (
    C^2 + \| \xi \|_{L^{2p}}
  )
  +
  e^{ 3p C^4 }
      6p
      \sqrt{T}
  +
  e^{ p C^4 } K_{ 2 p }
   4p \sqrt{T}
\end{split}     \end{equation}
for all $ N \in \mathbb{N} $
and all $ p \in [2,\infty) $.
This shows the assertion 
in the case
$ p \in [2,\infty)$.
The case $ p \in [1,2) $ then follows from
Jensen's inequality and
this completes the proof of 
Lemma~\ref{lem_sup}.
\end{proof}
\subsection{Proof of Lemma \ref{lemm_omega}}
\begin{proof}[Proof of Lemma \ref{lemm_omega}]
First of all, we have
\begin{align*}
  \mathbb{P}\left[
    \sup_{n \in \{0,1,\ldots,N-1\}}
    \left| W_{t_{n+1}^N} - W_{t_n^N} \right|
    > N^{-\frac{1}{4}} 
  \right]
  &=
  \mathbb{P}\left[
    \cup_{n=0}^{N-1} \left\{
      \omega \in \Omega \; \bigg|
      \left| W_{t_{n+1}^N}(\omega) - 
        W_{t_n^N}(\omega) \right|
      > N^{-\frac{1}{4}}
    \right\}
  \right]
\\
  &\leq 
  \sum_{n=0}^{N-1}
  \mathbb{P}\left[
    \left| W_{t_{n+1}^N} - W_{t_n^N} \right|
    > N^{-\frac{1}{4}} 
  \right]
\end{align*}
and
\begin{equation}  \begin{split}
\lefteqn{
  \mathbb{P}\left[
    \sup_{n \in \{0,1,\ldots,N-1\}}
    \left| W_{t_{n+1}^N} - W_{t_n^N} \right|
    > N^{-\frac{1}{4}} 
  \right]
}
\\
&\leq
  N \cdot
  \mathbb{P}\left[
    \left| W_{t_{1}^N} - W_{t_0^N} \right|
    > N^{-\frac{1}{4}} 
  \right]
\\
  &=
  2 N \cdot
  \mathbb{P}\left[
     W_{\frac{T}{N}} > N^{-\frac{1}{4}} 
  \right]
  =
  2 N \cdot
  \mathbb{P}\left[
     \sqrt{\frac{N}{T}} W_{\frac{T}{N}} > N^{-\frac{1}{4}} \sqrt{\frac{N}{T}} 
  \right]
\end{split}     \end{equation}
and
\begin{equation}  \begin{split}\label{tildeomegaNest}
\lefteqn{
  \mathbb{P}\left[
    \sup_{n \in \{0,1,\ldots,N-1\}}
    \left| W_{t_{n+1}^N} - W_{t_n^N} \right|
    > N^{-\frac{1}{4}} 
  \right]
}
\\
  &\leq
  2 N
  \int_{\frac{N^{\frac{1}{4}}}{\sqrt{T}}}^{\infty}
  \frac{1}{\sqrt{2 \pi}} e^{-\frac{x^2}{2}} \, dx
  \leq
  2 N
  \int_{\frac{N^{\frac{1}{4}}}{\sqrt{T}}}^{\infty}
  \frac{1}{\sqrt{2 \pi}} 
  \frac{x \sqrt{T}}{N^{\frac{1}{4}}} e^{-\frac{x^2}{2}} \, dx
\\&=
  2 N^{\frac{3}{4}} \sqrt{T} \!\!
  \int_{\frac{N^{\frac{1}{4}}}{\sqrt{T}}}^{\infty}
  \frac{1}{\sqrt{2 \pi}} \, x \, e^{-\frac{x^2}{2}} \, dx
  =
  2 N^{\frac{3}{4}} \sqrt{T}
  \frac{1}{\sqrt{2 \pi}} \! 
  \left[ \!
    - e^{-\frac{x^2}{2}}
  \right]_{x = \frac{N^{\frac{1}{4}}}{\sqrt{T}}}^{x = \infty}
\\&=
  2 N^{\frac{3}{4}} \sqrt{T}
  \frac{1}{\sqrt{2 \pi}} 
  e^{-\frac{\sqrt{N}}{2 T}}
  =
  \sqrt{\frac{2}{\pi}}
  N^{\frac{3}{4}} \sqrt{T}
  e^{-\frac{\sqrt{N}}{2 T}}
  \leq 
  N^{\frac{3}{4}} \sqrt{T}
  e^{-\frac{\sqrt{N}}{2 T}}
\end{split}     \end{equation}
for every $ N \in \mathbb{N} $.
Additionally, note that
\begin{equation}
  ( \Omega_N )^c
  =
  \left\{
    \omega \in \Omega \; \bigg|
    \sup_{v,w \in \{0,1,\ldots,N\}}
    \left| D_{v,w}^{N,1}(\omega) \right|
    > r_N 
  \right\}
  \cup
  \left\{
    \sup_{n \in \{0,1,\ldots,N-1\}}
    \left| W_{t_{n+1}^N} - W_{t_n^N} \right|
    > N^{-\frac{1}{4}} 
  \right\}
\end{equation}
for every $ N \in \mathbb{N} $.
Therefore, inequality~\eqref{tildeomegaNest} implies
\begin{equation}  \begin{split}
  \mathbb{P}\Big[
    ( \Omega_N )^c 
  \Big]
&\leq
  \mathbb{P}\left[
      \sup_{v,w \in \{0,1,\ldots,N\}}
      \left| D_{v,w}^{N,1} \right| 
      > r_N
  \right]
  +
  \mathbb{P}\left[
    \sup_{n \in \{0,1,\ldots,N-1\}}
    \left| W_{t_{n+1}^N} - W_{t_n^N} \right|
    > N^{-\frac{1}{4}} 
  \right]
\\
&\leq
  \mathbb{P}\left[
    \left( 1 +
    \sup_{v,w \in \{0,1,\ldots,N\}}
    \left| D_{v,w}^{N,1} \right| 
     \right)^p
    > (1 + r_N)^p
  \right]
  +
  N^{\frac{3}{4}} \sqrt{T}
  e^{-\frac{\sqrt{N}}{2 T}}
\end{split}     \end{equation}
for every $ N \in \mathbb{N} $ and
every $ p \in [1,\infty) $.
Now we apply Markov's inequality to obtain
\begin{equation}  \begin{split}
  \mathbb{P}\bigg[
    ( \Omega_N )^c 
  \bigg]
  &\leq
  \frac{\mathbb{E}\left[ \left( 1 +
      \sup_{v,w \in \{0,1,\ldots,N\}}
      \left| D_{v,w}^{N,1} \right| 
     \right)^p
  \right]}{(1 + r_N)^p  }
  +
  N^{\frac{3}{4}} \sqrt{T}
  e^{-\frac{\sqrt{N}}{2 T}}
\\
&\leq
  \left(
    \sup_{M \in \mathbb{N}}
    \left\| 1 \!+\!
       \sup_{v,w \in \{0,1,\ldots,M\}}
      \left| D_{v,w}^{M,1} \right| 
    \right\|_{L^p}^p
  \right) \!\!
  \frac{1}{(1 + r_N)^p} 
  +
  N^{\frac{3}{4}} \sqrt{T}
  e^{-\frac{\sqrt{N}}{2 T}}
\end{split}     \end{equation}
and
\begin{equation}  \begin{split}\label{omegaNest}
  \mathbb{P}\bigg[
    ( \Omega_N )^c 
  \bigg]
&\leq
  \left(
    \sup_{M \in \mathbb{N}}
    \left\| 1 +
        \sup_{v,w \in \{0,1,\ldots,M\}}
        \left| D_{v,w}^{M,1} \right|
    \right\|_{L^p}^p
  \right)
  c^p N^{ -p \cdot 
    \min\left( \frac{1}{4}, \frac{1}{(\delta - 1)} \right) }
  +
  N^{\frac{3}{4}} \sqrt{T}
  e^{-\frac{\sqrt{N}}{2 T}}
\\
&\leq
  \left[
    c \left(
    1 + 
    \sup_{M \in \mathbb{N}}
    \left\|
       \sup_{v,w \in \{0,1,\ldots,M\}}
      \left| D_{v,w}^{M,1} \right| 
    \right\|_{L^p}
    \right)
  \right]^p
  N^{ -p \cdot 
  \min\left( \frac{1}{4}, \frac{1}{(\delta - 1)} \right) }
  +
  N^{\frac{3}{4}} \sqrt{T}
  e^{-\frac{\sqrt{N}}{2 T}}
\end{split}     \end{equation}
for every $ N \in \mathbb{N} $ and every $ p \in [1, \infty) $, 
where we used
$$
  \frac{1}{(1 + r_N)}
\leq
  c 
  \cdot 
  N^{ -\min\left( 
    \frac{1}{4}, \frac{1}{(\delta - 1)} \right) 
  }
$$
for every $ N \in \mathbb{N} $ with
$ c \in (0,\infty) $ given by
$
  c
  :=
  \sup_{ N \in \mathbb{N} }
  \left(
      N^{ 
        \min\left( 
          \frac{1}{4}, \frac{1}{(\delta - 1)} 
        \right) }
     /
      \left( 1 + r_N \right) 
  \right)
  < \infty .
$
Moreover, \eqref{omegaNest} with
$p:=4\max( 4, \delta -1 )$ yields
\begin{equation}  \begin{split}
  \mathbb{P}\bigg[
    ( \Omega_N )^c 
  \bigg]
&\leq
  N^{ -4 }
  \left[
    c \left(
    1 + 
    \sup_{M \in \mathbb{N}}
    \left\|
      \sup_{v,w \in \{0,1,\ldots,M\}}
      \left| D_{v,w}^{M,1} \right|
    \right\|_{ L^{ 4 \max( 4, \delta - 1 ) } }
    \right)
  \right]^{ 4 \max( 4, \delta - 1 ) }
  +
  N \sqrt{T}
  e^{-\frac{\sqrt{N}}{2 T}}
\\&\leq
  N^{ -4 }
  \left[
    c \left(
    1 + 
    \sup_{M \in \mathbb{N}}
    \left\|
      \sup_{v,w \in \{0,1,\ldots,M\}}
      \left| D_{v,w}^{M,1} \right|
    \right\|_{ L^{ 4 \max( 4, \delta - 1 ) } }
    \right)
  \right]^{ 4 \max( 4, \delta - 1 ) }
\\
&\qquad
  +
  N^{-4}
  \left[
  \sqrt{T}
  \left( \sup_{ x \in [0,\infty) }
    x^5 e^{ - \frac{ \sqrt{x} }{ 2 T } }
  \right)
  \right]
\end{split}     \end{equation}
for every $ N \in \mathbb{N} $.
The right-hand side is finite by Lemma~\ref{lem_sup}.
This proves
\begin{equation}\label{omega_estimate}
  \tilde{c}
  :=
  \sup_{ N \in \mathbb{N} }
  \left(
    N^4 \cdot 
    \mathbb{P}\bigg[ (\Omega_N)^c \bigg]
  \right)
  \in [0, \infty).
\end{equation}
Next we show that the event $\tilde{\Omega}$ has probability $1$.
We have
\begin{equation}  \begin{split}
  \mathbb{P}\left[
    \bigcap^{\infty}_{ M = N }
    \bigcap^{M^2}_{m=1}
    \Omega^m_M 
  \right]
&=
  1 -
  \mathbb{P}\left[ \,
    \left(
      \bigcap^{\infty}_{ M = N }
      \bigcap^{M^2}_{m=1}
      \Omega^m_M 
    \right)^c \;
  \right]
=
  1 -
  \mathbb{P}\left[ 
      \bigcup^{\infty}_{ M = N }
      \bigcup^{M^2}_{m=1}
    \left(
      \Omega^m_M 
    \right)^c
  \right]
\end{split}     \end{equation}
and hence
\begin{multline*}
  \mathbb{P}\left[
    \bigcap^{\infty}_{ M = N }
    \bigcap^{M^2}_{m=1}
    \Omega^m_M 
  \right]
\geq
  1 -
  \sum^{\infty}_{M=N}
  \sum^{M^2}_{m=1}
  \mathbb{P}\bigg[ 
    \left(
      \Omega^m_M 
    \right)^c
  \bigg]
=
  1 -
  \sum^{\infty}_{M=N}
  \left(
  M^2 \cdot
  \mathbb{P}\bigg[ 
    \left(
      \Omega_M 
    \right)^c
  \bigg]
  \right)
  \geq 
  1 - 
  \left(
    \sum_{M=N}^{\infty} \tilde{c} M^{-2}
  \right)
\end{multline*}
for every $ N \in \mathbb{N} $ due to \eqref{omega_estimate}.
This implies
\begin{align*}
  \mathbb{P}\left[
    \bigcap^{\infty}_{ M = N }
    \bigcap^{M^2}_{m=1}
    \Omega^m_M 
  \right]
  \geq
  1 - \tilde{c}
  \left(
    \int_{N-1}^{\infty}
    \frac{1}{s^2} \, ds
  \right)
  =
  1 - \tilde{c}
  \left[ -\frac{1}{s} \right]_{s=N-1}^{s=\infty}
  =
  1 - \frac{\tilde{c}}{(N-1)}
\end{align*}
for every $ N \in \{2,3,\ldots\} $, which shows
\begin{equation}\label{pfull1}
  \mathbb{P}\left[
    \bigcup_{ N \in \mathbb{N} }
    \bigcap^{\infty}_{ M = N }
    \bigcap^{M^2}_{m=1}
    \Omega^m_M
  \right]
  =
  \lim_{N \rightarrow \infty}
  \mathbb{P}\left[
    \bigcap^{\infty}_{ M = N }
    \bigcap^{M^2}_{m=1}
    \Omega^m_M
  \right]
  = 1 .
\end{equation}
Moreover, we have
\begin{align*}
&
  \mathbb{E}\left[
    \left( 
      \sup_{ N \in \mathbb{N} }
      \frac{
        \left|      
          \sum^{ N^2 }_{ m=1 }
          \left(
            \mathbbm{1}_{ \Omega_N^m }
            \cdot
            f( 
              Y^{N,m}_N
            )
            -
            \mathbb{E}\left[ 
              \mathbbm{1}_{ \Omega_N }
              f( 
                Y^{N}_N
              )
            \right]
          \right)
        \right| 
      }{
        N^{ \left( 1 + \varepsilon 
        \right) }
      }
    \right)^p \;
  \right] 
\\&\leq
  \mathbb{E}\left[
      \sum_{ N=1 }^{ \infty }
      \frac{
        \left|      
          \sum^{ N^2 }_{ m=1 }
          \left(
            \mathbbm{1}_{ \Omega_N^m }
            \cdot
            f( 
              Y^{N,m}_N
            )
            -
            \mathbb{E}\left[ 
              \mathbbm{1}_{ \Omega_N }
              f( 
                Y^{N}_N
              )
            \right]
          \right)
        \right|^p 
      }{
        N^{ \left( 1 + \varepsilon 
        \right) p }
      }
  \right] 
\\&=
      \sum^{ \infty }_{ N=1 }
      \frac{
        \left\|      
          \sum^{ N^2 }_{ m=1 }
          \left(
            \mathbbm{1}_{ \Omega_N^m }
            \cdot
            f( 
              Y^{N,m}_N
            )
            -
            \mathbb{E}\left[ 
              \mathbbm{1}_{ \Omega_N }
              f( 
                Y^{N}_N
              )
            \right]
          \right)
        \right\|_{L^p}^p 
      }{
        N^{ \left( 1 + \varepsilon 
        \right) p }
      }
\\&\leq
      \sum^{ \infty }_{ N=1 }
      \frac{
        \left(
          K_p
          \left( 
            \sum_{m=1}^{N^2}
        \left\|      
            \mathbbm{1}_{ \Omega_N }
            f( 
              Y^N_N
            )
            -
            \mathbb{E}\left[ 
              \mathbbm{1}_{ \Omega_N }
              f( 
                Y^{N}_N
              )
            \right]
        \right\|_{L^p}^2  
          \right)^{ \frac{1}{2} }
        \right)^p
      }{
        N^{ \left( 1 + \varepsilon 
        \right) p }
      }
\end{align*}
due to Lemma \ref{discretedavis} and therefore
\begin{equation}  \begin{split}
&
  \mathbb{E}\left[
    \left( 
      \sup_{ N \in \mathbb{N} }
      \frac{
        \left|      
          \sum^{ N^2 }_{ m=1 }
          \left(
            \mathbbm{1}_{ \Omega_N^m }
            \cdot
            f( 
              Y^{N,m}_N
            )
            -
            \mathbb{E}\left[ 
              \mathbbm{1}_{ \Omega_N }
              f( 
                Y^{N}_N
              )
            \right]
          \right)
        \right| 
      }{
        N^{ \left( 1 + \varepsilon 
        \right) }
      }
    \right)^p \;
  \right] 
\\&\leq
      \sum^{ \infty }_{ N=1 }
      \frac{
        \left(
          K_p \cdot N 
          \cdot
        \left\|      
            \mathbbm{1}_{ \Omega_N }
            f( 
              Y^N_N
            )
            -
            \mathbb{E}\left[ 
              \mathbbm{1}_{ \Omega_N }
              f( 
                Y^{N}_N
              )
            \right]
        \right\|_{L^p} 
        \right)^p
      }{
        N^{ \left( 1 + \varepsilon 
        \right) p }
      }
\\&=
      \sum^{ \infty }_{ N=1 }
      \left(
      \frac{
          \left( K_p \right)^p
          \cdot
        \left\|      
            \mathbbm{1}_{ \Omega_N }
            f( 
              Y^N_N
            )
            -
            \mathbb{E}\left[ 
              \mathbbm{1}_{ \Omega_N }
              f( 
                Y^{N}_N
              )
            \right]
        \right\|_{L^p}^p 
      }{
        N^{ \varepsilon p }
      }
    \right)
\\&\leq
   \left( K_p \right)^p
    \left(
      \sum^{ \infty }_{ N=1 }
        N^{ - \varepsilon p }
    \right)
  \left( 
  \sup_{ N \in \mathbb{N} }
  \mathbb{E}\left[
        \left|      
            \mathbbm{1}_{ \Omega_N }
            f( 
              Y^N_N
            )
            -
            \mathbb{E}\left[ 
              \mathbbm{1}_{ \Omega_N }
              f( 
                Y^{N}_N
              )
            \right]
        \right|^p 
  \right]
  \right)
\end{split}     \end{equation}
for every $ p \in [2,\infty) $
and every $ \varepsilon \in (0,\infty) $.
Hence, we obtain
\begin{align*}
&
  \mathbb{E}\left[
    \left( 
      \sup_{ N \in \mathbb{N} }
      \frac{
        \left|      
          \sum^{ N^2 }_{ m=1 }
          \left(
            \mathbbm{1}_{ \Omega_N^m }
            \cdot
            f( 
              Y^{N,m}_N
            )
            -
            \mathbb{E}\left[ 
              \mathbbm{1}_{ \Omega_N }
              f( 
                Y^{N}_N
              )
            \right]
          \right)
        \right| 
      }{
        N^{ \left( 1 + \varepsilon 
        \right) }
      }
    \right)^p \;
  \right] 
\\&\leq
     \left( 2^p
       \left( K_p \right)^p
     \right)
    \left(
      \sum^{ \infty }_{ N=1 }
        N^{ - \varepsilon p }
    \right)
  \left( 
  \sup_{ N \in \mathbb{N} }
  \mathbb{E}\left[
        \left|      
            \mathbbm{1}_{ \Omega_N }
            f( 
              Y^N_N
            )
        \right|^p 
  \right]
  \right)
\\
  &\leq
  \left( 2^p
    \left( K_p \right)^p
  \right)
  \left(
    \sum^{ \infty }_{ N=1 }
    N^{ - \varepsilon p }
  \right)
  \left( 
    \sup_{ N \in \mathbb{N} }
    \mathbb{E}\left[
      \mathbbm{1}_{ \Omega_N }
      L^p
      \left(
        1 + \left| Y^N_N \right|^{\delta}
      \right)^p 
    \right]
  \right)
\\
  &\leq
  \left( 2 L K_p \right)^p
  \left(
    \sum^{ \infty }_{ N=1 }
    N^{ - \varepsilon p }
  \right)
  \left( 
    \sup_{ N \in \mathbb{N} }
    \mathbb{E}\left[
      \mathbbm{1}_{ \Omega_N }
      \left(
        1 + \left| Y^N_N \right|^{\delta}
      \right)^p 
    \right]
  \right)
\end{align*}
and
\begin{align*}
&
  \mathbb{E}\left[
    \left( 
      \sup_{ N \in \mathbb{N} }
      \frac{
        \left|      
          \sum^{ N^2 }_{ m=1 }
          \left(
            \mathbbm{1}_{ \Omega_N^m }
            \cdot
            f( 
              Y^{N,m}_N
            )
            -
            \mathbb{E}\left[ 
              \mathbbm{1}_{ \Omega_N }
              f( 
                Y^{N}_N
              )
            \right]
          \right)
        \right| 
      }{
        N^{ \left( 1 + \varepsilon 
        \right) }
      }
    \right)^p \;
  \right]
\\
  &\leq
  \left( 4 L K_p \right)^p
  \left(
    \sum^{ \infty }_{ N=1 }
    N^{ - \varepsilon p }
  \right)
  \left(
    1 + 
    \sup_{ N \in \mathbb{N} }
    \mathbb{E}\left[
      \mathbbm{1}_{ \Omega_N }
      \left| Y^N_N \right|^{p \delta}
    \right]
  \right)
  < \infty
\end{align*}
for every $ p \in 
\left( \frac{1}{\varepsilon} , \infty
\right) $ 
and every $ \varepsilon \in (0,1) $
due to Corollary~\ref{c:mom_Euler}.
This implies
\begin{equation}\label{pfull2}
  \mathbb{P}\bigg[
    \bigg\{ 
      \omega \in \Omega 
      \, \bigg|
      \sup_{ N \in \mathbb{N} }
      \frac{ 
      \left|      
        \sum^{ N^2 }_{ m=1 }
        \left(
          \mathbbm{1}_{ \Omega_N^m }(\omega)
          \cdot
          f( 
            Y^{N,m}_N(\omega)
          )
          -
          \mathbb{E}\left[ 
            \mathbbm{1}_{ \Omega_N }
            f( 
              Y^{N}_N
            )
          \right]
        \right)
      \right| 
      }{
        N^{ \left( 1 + \varepsilon \right) }
      }
      < \infty
    \bigg\}
  \bigg] 
  = 1
\end{equation}
for every $ \varepsilon \in (0,\infty) $
(see also Lemma~2.1 in 
Kloeden \& Neuenkirch~\cite{KN07}).
Putting together \eqref{pfull1}
and \eqref{pfull2} shows
$
  \mathbb{P}[ \tilde{\Omega} ]
  = 1
$. This completes the proof of 
Lemma~\ref{lemm_omega}.
\end{proof}
%
%
%
\subsection{Proof of Lemma \ref{lemm_ew}}
\label{sec:modifiedweak}
It is somewhat inconvenient to compare the exact solution, which
is a continuous time process, 
with the Euler approximations,
which are time-discrete
stochastic processes.
Therefore, we consider the following interpolation process of the
Euler approximation.
Let $ \tilde{Y}^N \colon [0,T] \times \Omega \rightarrow \mathbb{R} $,
$ N \in \mathbb{N} $, be given by
\begin{equation}    \label{eq:interpolation}
  \tilde{Y}_t^N(\omega) 
  :=
  Y_n^N(\omega) + 
  \left( t - t_n^N \right) \cdot \mu( Y_n^N(\omega) )
  + \sigma( Y_n^N(\omega) ) 
  \cdot \left( W_t(\omega) - W_{t_n^N}(\omega) \right)
\end{equation}
for every $ t \in [ t_n^N, t_{n+1}^N ] $,
$ n \in \{ 0, 1, \ldots, N-1 \} $,
$ \omega \in \Omega $
and every $ N \in \mathbb{N} $.
Note that $\tilde{Y}_{t_n^N}^N(\omega) = 
Y_n^N(\omega) $ for every $ \omega \in \Omega $,
$n\in\{0,1,\ldots,N\}$ and every $N\in\mathbb{N}$.
Before we prove Lemma~\ref{lemm_ew},
we show that the restricted moments of the interpolation processes
are uniformly bounded.
\begin{lemma}  \label{l:mom_Ytilde}
  Let $ \tilde{Y}^N \colon [0,T] \times
  \Omega \rightarrow \mathbb{R} $
  for $ N \in \mathbb{N}$
  be given by~\eqref{eq:interpolation}.
  Then we have     
  \begin{equation}  \label{eq:mom_Ytilde}
    \sup_{ N\in\mathbb{N} }
    \sup_{ 0 \leq t \leq T }
    \mathbb{E}\left[
      \mathbbm{1}_{ \Omega_{ N, \lfloor
        \frac{ t N }{ T }
      \rfloor } }  
      \bigl| \tilde{Y}_t^N \bigr|^p
    \right]
    <\infty
  \end{equation}
  for every $p\in[1,\infty)$.
\end{lemma}
\begin{proof}[Proof
of Lemma~\ref{l:mom_Ytilde}]
Inserting the definition~\eqref{eq:interpolation} of the interpolation
process and the polynomial growth of $\mu$ and 
$\sigma$ shows
  \begin{align*}
   \left\|
      \mathbbm{1}_{\Omega_{N,n}}
      \left| \tilde{Y}_t^N \right|
    \right\|_{L^p}
  \leq
    \left\|
      \mathbbm{1}_{\Omega_{N,n}}
        \left( 
          1 + \left| Y_n^N \right|^{\delta} 
        \right) 
        \biggl(
          1 + LT + L \left|W_{t} - W_{t_{n}^N}\right|
        \biggr)
    \right\|_{L^p}
  \end{align*}
  for every $ t \in [t_n^N, t_{n+1}^N] $, $ n \in \{0,1,\ldots,N-1\} $
  and every $ N \in \mathbb{N} $.
  Now we apply H\"older's inequality, the triangle inequality and
  Lemma~\ref{l:mom_normal} to arrive at
  \begin{equation}  \begin{split}           \label{eq:last_for_tilde_Y}
   \left\|
      \mathbbm{1}_{\Omega_{N,n}}
      \left| \tilde{Y}_t^N \right| 
    \right\|_{L^p}
  &\leq
    \left(
      1+
      \left\|
        \mathbbm{1}_{\Omega_{N,n}}
        \left| Y_n^N \right|^{\delta} 
      \right\|_{L^{2p}}
    \right)
    \cdot
    \left(
       1 + LT + L \left\|W_{t} - W_{t_{n}^N}\right\|_{L^{2p}}
    \right)
  \\&\leq
    \left(
      1+
      \left(
        \mathbb{E}\left[\mathbbm{1}_{\Omega_{N,n}}
                  \left| Y_n^N \right|^{2\delta p} 
                  \right]
      \right)^{ \frac{1}{2p} }
    \right)
    \cdot
    \left(
       1 + LT + L 2p \sqrt{T}
    \right)
  \end{split}     \end{equation}
  for every $ t \in [t_n^N, t_{n+1}^N] $, $ n \in \{0,1,\ldots,N-1\} $
  and every $ N \in \mathbb{N} $.
  The right-hand side of~\eqref{eq:last_for_tilde_Y} is uniformly bounded
  in $n\in\{0,1,\ldots,N\}$ and $N\in\mathbb{N}$
  according to
  Corollary~\ref{c:mom_Euler}.
  This completes the proof.
\end{proof}

\begin{proof}[Proof of Lemma \ref{lemm_ew}]
Let
$ X^{s,x} \colon [s,T] \times 
\Omega \rightarrow \mathbb{R} $,
$ s \in [0,T] $, $ x \in \mathbb{R} $, 
be a family of adapted stochastic processes
with continuous sample paths given by
\begin{equation}
  X_t^{s, x}
  =
  x
  + 
  \int_{s}^t \mu\left( X_u^{s, x} \right) du
  +
  \int_{s}^t \sigma\left( X_u^{s, x} \right) dW_u
  \qquad \mathbb{P}\text{-a.s.}
\end{equation}
for every $ t \in [s, T] $, 
$ s \in [0, T] $ and every $ x \in \mathbb{R} $. 
Moreover, assume that
the mapping
$
  X^{ s, \cdot }_t( \cdot )
  \colon \mathbb{R} \times
  \Omega \rightarrow \mathbb{R}
$
with
$
  (x, \omega) \mapsto
  X^{ s, x }_t( \omega )
$
for all $ x \in \mathbb{R} $,
$ \omega \in \Omega $
is
$
  ( 
    \mathcal{B}( \mathbb{R} )
    \otimes
    \mathcal{F}_t
  )/(
    \mathcal{B}( \mathbb{R} )
  )
$-measurable
for every $ s, t \in [0,T] $
with $ s \leq t $ and
assume that the mapping
$
  X^{ s, \cdot }_t(\omega)
  \colon
  \mathbb{R} \rightarrow \mathbb{R}
$
with
$
  x \mapsto
  X^{ s, x }_t(\omega)
$
for all $ x \in \mathbb{R} $
is continuous
for every $ s, t \in [0,T] $
with $ s \leq t $
and every $ \omega \in \Omega $.
We will show in~\eqref{e_estimates} below that the
difference between $X_{t_{n+1}}^{t_n^N,Y_n^N}$ and
$Y_{n+1}^N$ is of order $O\bigl(\tfrac{1}{N^2}\bigr)$
in a suitable weak sense for every $n\in\{0,1,\ldots,N-1\}$.
Summing over 
$n\in\{0,1,\ldots,N-1\}$ 
for each $ N \in \mathbb{N} $
will then prove
the assertion.

First we need several preparations.
According to Theorem 2.6.4 
in \cite{GS72}, there are real numbers
$ \kappa_p \in [0,\infty) $, $ p \in \left\{2,4,6,
\dots \right\}$, such that
\begin{equation}
  \mathbb{E}\Bigl[
    \left| X^{ s, x }_t \right|^p
  \Bigr]
\leq
  \kappa_p \left( 1 + \left| x \right|^p \right)
\end{equation}
holds for every $ s, t \in [0,T]$
with $ s \leq t $ 
and every 
$ p \in \left\{ 2, 4, 6, \dots
\right\} $, $ x \in \mathbb{R} $.
This implies
\begin{equation}  \begin{split}  \label{eq:for_mom_XY}
  \sup_{t_n^N\leq t\leq T}
  \mathbb{E}\left[ \mathbbm{1}_{\Omega_{N,n}}
                   \left|X_t^{t_n^N,Y_n^N}\right|^p
            \right]
  &= \sup_{t_n^N\leq t\leq T}
  \mathbb{E}\left[ \mathbbm{1}_{\Omega_{N,n}}
                   \mathbb{E}\left[
                        \left|X_t^{t_n^N,Y_n^N}\right|^p
                        \bigg| \mathcal{F}_{t_n^N}
                   \right]
            \right]
  \\ &\leq \kappa_p \cdot
  \mathbb{E}\Bigl[ \mathbbm{1}_{\Omega_{N,n}}
                   \left(
                        1 + \left|Y_n^N\right|^p
                   \right)
            \Bigr]
  \\ &\leq \kappa_p \left(1+
  \mathbb{E}\left[ \mathbbm{1}_{\Omega_{N,n}}
                   \left| Y_n^N \right|^p
            \right]
              \right)
\end{split}     \end{equation}
for every $p\in \left\{2,4,6,\dots\right\}$, 
$n\in\{0,1,\ldots,N\}$
and every $N\in\mathbb{N}$.
Corollary~\ref{c:mom_Euler}
hence
implies
\begin{equation}  \label{eq:mom_XY}
  \sup_{ N\in\mathbb{N} } 
  \sup_{ n\in\{0,1,\ldots,N\} } 
  \sup_{ t_n^N \leq t\leq T }
  \mathbb{E}\left[ \mathbbm{1}_{\Omega_{N,n}}
                   \left|X_t^{t_n^N,Y_n^N}\right|^p
            \right]
  <\infty
\end{equation}
for every $p\in [1,\infty)$.

Now define $ u: [0, T] \times \mathbb{R} \rightarrow \mathbb{R} $
by $ u(t,x) = \mathbb{E} \bigl[ f\bigl( X_T^{t,x} \bigr) \bigr] $
for every $ t \in [0, T] $ and every
$ x \in \mathbb{R} $.
Moreover, let the $n$-th partial derivative 
of $u$ with respect to the second
argument be the function
$ u_n \colon [0,T] \times \mathbb{R} 
\rightarrow \mathbb{R} $
defined through
\begin{equation}
  u_n(t,x) = 
  \left( \frac{\partial^n}{\partial x^n} u 
  \right)(t,x)
\end{equation}
for every $ n \in \left\{0,1,\dots, 4 \right\}$,
$ t \in [0,T] $ and every $ x \in \mathbb{R} $.
Additionally, we use
the functions
$ 
  u, \tilde{u} \colon 
  [0,T] \times \mathbb{R}
  \rightarrow \mathbb{R}
$
given by
\begin{equation}
  \tilde{u}(t,x)
  =
  u_1(t,x) \cdot \mu(x)
  + \frac{1}{2} u_2(t,x) \cdot \left( \sigma(x) \right)^2
\end{equation}
and
\begin{equation}
  \tilde{\tilde{u}}(t,x)
  = 
  \left( \frac{\partial}{\partial x} \tilde{u} \right)(t,x)
  \cdot \mu(x)
  + \frac{1}{2} \left( \frac{\partial^2}{\partial x^2} \tilde{u} \right)(t,x) \cdot \left( \sigma(x) \right)^2
\end{equation}
for every $ t \in [0,T] $ and every $ x \in \mathbb{R} $.
Moreover, let $ R \in [ 5 \delta, \infty ) $ 
be a real number which satisfies
\begin{align}       \label{eq:pol_growth_u}
  &\left| u(t,x) \right|
  \leq 
  R \left( 1 + \left|x\right|^R \right)
  ,
  &\left| u_1(t,x) \right|
  \leq 
  R \left( 1 + \left|x\right|^R \right) ,
  \nonumber
\\
  &\left| u_2(t,x) \right|
  \leq 
  R \left( 1 + \left|x\right|^R \right)
  ,
  &\left| u_3(t,x) \right|
  \leq 
  R \left( 1 + \left|x\right|^R \right) ,
  \nonumber
\\
  &\left| u_4(t,x) \right|
  \leq 
  R \left( 1 + \left|x\right|^R \right)
  ,
  &\left| \tilde{u}(t,x) \right|
  \leq 
  R \left( 1 + \left|x\right|^R \right) ,
\\
  &\left| \tilde{\tilde{u}}(t,x) \right|
  \leq 
  R \left( 1 + \left|x\right|^R \right)
  ,
  &\left| 
    \left( \frac{\partial}{\partial x} \tilde{u}
  \right)(t,x) \right|
  \leq 
  R \left( 1 + \left|x\right|^R \right)
  \nonumber
\end{align}
for every $ t \in [0, T] $ and every $ x \in \mathbb{R} $
(see also Corollary 2.8.1 and Theorem 2.8.1 in~\cite{GS72}). 
The existence of such a real
number can be shown by
exploiting \eqref{eq:polynomial_growth}, \eqref{eq:onesided}
and \eqref{eq:sigmalip}.
In our estimates, we will need the real number
$ C \in [0,\infty) $ defined by
\begin{multline}\label{c_definition}
  C
  :=
  \sup_{ N \in \mathbb{N} }
  \left(
    N^4 \cdot \mathbb{P}\left[ (\Omega_N)^c \right]
  \right)
  +
  \sup_{ N \in \mathbb{N} }
  \sup_{ 0 \leq u \leq T }
  \left\|
    \mathbbm{1}_{\Omega_{N,\lfloor \frac{u N}{T} \rfloor}}
    \left( 2 + \left| \tilde{Y}_u^N \right|^R \right)
  \right\|_{L^6}
\\
  + L + R + T
  +
  \sup_{ N \in \mathbb{N} }
  \sup_{ n \in \{0,1,\ldots, N\} }
  \sup_{ t_n^N \leq t \leq T }
  \left\|
    \mathbbm{1}_{\Omega_{N,n}}
    \left( 2 + 
      \left| X_t^{t_n^N, Y_n^N} \right|^R
    \right)
  \right\|_{ L^4 }
  < \infty .
\end{multline}
Indeed, $ C \in [0, \infty) $ is finite due to Lemma~\ref{lemm_omega},
Lemma~\ref{l:mom_Ytilde} and due to~\eqref{eq:mom_XY}.
Moreover, since
\begin{equation}
  \tilde{Y}_t^N 
  =
  Y_n^N +
  \int_{t_n^N}^t \mu( Y_n^N ) \, ds +
  \int_{t_n^N}^t \sigma( Y_n^N ) \, dW_s
  \qquad \mathbb{P}\text{-a.s.}
\end{equation}
holds for every $ t \in [ t_n^N, t_{n+1}^N ] $,
$ n \in \{ 0, 1, \ldots, N-1 \} $
and every $ N \in \mathbb{N} $, It{\^o}'s formula yields
\begin{equation}  \begin{split}
  u\!\left( 
    t_{n+1}^N, Y_{n+1}^N 
  \right)
  &=
  u\!\left( 
    t_{n+1}^N, \tilde{Y}_{t_{n+1}^N}^N 
  \right)
\\
  &=
  u\!\left( 
    t_{n+1}^N, Y_{n}^N
  \right)
  +
  \int_{t_{n}^N}^{t_{n+1}^N}
  u_1\!\left(
    t_{n+1}^N, \tilde{Y}_{s}^N
  \right)
  \mu( Y_n^N ) \, ds
\\
  &\quad +
  \int_{t_{n}^N}^{t_{n+1}^N}
  u_1\!\left(
    t_{n+1}^N, \tilde{Y}_{s}^N
  \right)
  \sigma( Y_n^N ) \, dW_s
\\
  &\quad + \frac{1}{2}
  \int_{t_{n}^N}^{t_{n+1}^N}
  u_2\!\left(
    t_{n+1}^N, \tilde{Y}_{s}^N
  \right) 
  \left(\sigma( Y_n^N )\right)^2 ds
  \qquad \mathbb{P}\text{-a.s.}
\end{split}     \end{equation}
for every $ n \in \{ 0, 1, \ldots, N-1 \} $ and every
$ N \in \mathbb{N} $.
Again It{\^o}'s formula yields
\begin{align}\label{all_integrals}
\nonumber
  u\!\left( 
    t_{n+1}^N, Y_{n+1}^N 
  \right)
  &=
  u\!\left( 
    t_{n+1}^N, Y_{n}^N 
  \right) 
  +
  \frac{T}{N}
  \tilde{u}\!\left( 
    t_{n+1}^N, Y_{n}^N 
  \right)
\\\nonumber
  &\quad +
  \int_{t_{n}^N}^{t_{n+1}^N} \!\!\!
  \int_{t_{n}^N}^{s}
  u_2\!\left(
    t_{n+1}^N, \tilde{Y}_{r}^N
  \right)
  \left(\mu( Y_n^N )\right)^2 dr \, ds
\\\nonumber
  &\quad +
  \int_{t_{n}^N}^{t_{n+1}^N} \!\!\!
  \int_{t_{n}^N}^{s}
  u_2\!\left(
    t_{n+1}^N, \tilde{Y}_{r}^N
  \right)
  \sigma( Y_n^N ) \, \mu( Y_n^N ) \, dW_r \, ds
\\
  &\quad +
  \int_{t_{n}^N}^{t_{n+1}^N} \!\!\!
  \int_{t_{n}^N}^{s}
  u_3\!\left(
    t_{n+1}^N, \tilde{Y}_{r}^N
  \right)
  \left( \sigma( Y_n^N ) \right)^2 \mu( Y_n^N ) \, dr \, ds
\\\nonumber
  & \quad +
  \int_{t_{n}^N}^{t_{n+1}^N}
  u_1\!\left(
    t_{n+1}^N, \tilde{Y}_{s}^N
  \right)
  \sigma( Y_n^N ) \, dW_s
\\\nonumber
  &\quad + \frac{1}{2}
  \int_{t_{n}^N}^{t_{n+1}^N} \!\!\!
  \int_{t_{n}^N}^{s}
  u_3\!\left(
    t_{n+1}^N, \tilde{Y}_{r}^N
  \right)
  \left( \sigma( Y_n^N ) \right)^3  dW_r \, ds
\\\nonumber
  &\quad + \frac{1}{4}
  \int_{t_{n}^N}^{t_{n+1}^N} \!\!\!
  \int_{t_{n}^N}^{s}
  u_4\!\left(
    t_{n+1}^N, \tilde{Y}_{r}^N
  \right)
  \left( \sigma( Y_n^N ) \right)^4  dr \, ds
  \qquad \mathbb{P}\text{-a.s.}
\end{align}
for every $ n \in \{ 0, 1, \ldots, N-1 \} $ and every
$ N \in \mathbb{N} $.

Now we estimate all non-stochastic integrals on the
right-hand side of \eqref{all_integrals} restricted
to the events 
$\Omega_{N,n+1}$
for $ n \in \left\{0,1,\dots,N-1\right\}$
and $N \in \mathbb{N}$.
For the first integral on the right
hand side of \eqref{all_integrals}, we obtain
\begin{equation}  \begin{split}
&
  \left\|
    \mathbbm{1}_{\Omega_{N,n+1}}
    \int_{t_n^N}^{t_{n+1}^N} \!\!\!
    \int_{t_n^N}^s
    u_2\!\left( t_{n+1}^N , \tilde{Y}_r^N \right)
    \left( \mu( Y_n^N ) \right)^2
    dr \, ds
  \right\|_{L^1}
\\&\leq 
  \int_{t_n^N}^{t_{n+1}^N} \!\!\!
  \int_{t_n^N}^s
  \left\|
    \mathbbm{1}_{\Omega_{N,n+1}}
    u_2\!\left( t_{n+1}^N , \tilde{Y}_r^N \right)
    \left( \mu( Y_n^N ) \right)^2
  \right\|_{L^1}
  dr \, ds
\\&\leq 
  \int_{t_n^N}^{t_{n+1}^N} \!\!\!
  \int_{t_n^N}^s
  2 L^2
  \left\|
    \mathbbm{1}_{\Omega_{N,n+1}}
    u_2\!\left( t_{n+1}^N , \tilde{Y}_r^N \right)
    \left( 1 + \left| Y_n^N \right|^{2 \delta} \right)
  \right\|_{L^1}
  dr \, ds
\end{split}     \end{equation}
and, using the polynomial growth estimate~\eqref{eq:pol_growth_u}
of $u_2$,
\begin{equation}  \begin{split}
&
  \left\|
    \mathbbm{1}_{\Omega_{N,n+1}}
    \int_{t_n^N}^{t_{n+1}^N} \!\!\!
    \int_{t_n^N}^s
    u_2\!\left( t_{n+1}^N , \tilde{Y}_r^N \right)
    \left( \mu( Y_n^N ) \right)^2
    dr \, ds
  \right\|_{L^1}
\\&\leq 
  \int_{t_n^N}^{t_{n+1}^N} \!\!\!
  \int_{t_n^N}^s
  2 L^2 R
  \left\|
    \mathbbm{1}_{\Omega_{N,n+1}}
    \left( 1 + \left| \tilde{Y}_r^N \right|^R \right)
    \left( 1 + \left| Y_n^N \right|^{2 \delta} \right)
  \right\|_{L^1}
  dr \, ds
\\&\leq 
  \int_{t_n^N}^{t_{n+1}^N} \!\!\!
  \int_{t_n^N}^s
  2 L^2 R
  \left\|
    \mathbbm{1}_{\Omega_{N,n+1}}
    \left( 1 + \left| \tilde{Y}_r^N \right|^R \right)
    \left( 2 + \left| Y_n^N \right|^R \right)
  \right\|_{L^1}
  dr \, ds
\end{split}     \end{equation}
and, applying H\"older's 
inequality
and the definition~\eqref{c_definition} 
of $C \in [0,\infty)$,
\begin{align}\label{int_1}
&\nonumber
  \left\|
    \mathbbm{1}_{\Omega_{N,n+1}}
    \int_{t_n^N}^{t_{n+1}^N} \!\!\!
    \int_{t_n^N}^s
    u_2\!\left( t_{n+1}^N , \tilde{Y}_r^N \right)
    \left( \mu( Y_n^N ) \right)^2
    dr \, ds
  \right\|_{L^1}
\\&\leq\nonumber 
  2 L^2 R
  \int_{t_n^N}^{t_{n+1}^N} \!\!\!
  \int_{t_n^N}^s
  \left\|
    \mathbbm{1}_{\Omega_{N,n+1}}
    \left( 2 + \left| \tilde{Y}_r^N \right|^R \right)
  \right\|_{L^2}
  \left\|  
    \mathbbm{1}_{\Omega_{N,n+1}}
    \left( 2 + \left| Y_n^N \right|^R \right)
  \right\|_{L^2}
  dr \, ds
\\&\leq
  2 L^2 R
  \left(
    \sup_{ 0 \leq u \leq T }
    \left\|
      \mathbbm{1}_{\Omega_{N,\lfloor \frac{u N}{T} \rfloor}}
      \left( 2 + \left| \tilde{Y}_u^N \right|^R \right)
    \right\|_{L^2}^2
  \right)
  \frac{1}{2} 
  \left( \frac{T}{N} \right)^2
\\&\leq\nonumber
  L^2 R T^2 C^2 N^{-2}
  \leq
  C^7 N^{-2}
\end{align}
for every $ n \in \{ 0, 1, \ldots, N-1 \} $ and every
$ N \in \mathbb{N} $.
In addition, we have
\begin{align*}
&
  \left\|
    \mathbbm{1}_{\Omega_{N,n+1}}
    \int_{t_n^N}^{t_{n+1}^N} \!\!\!
    \int_{t_n^N}^s
    u_3\!\left( t_{n+1}^N , \tilde{Y}_r^N \right)
    \left( \sigma( Y_n^N ) \right)^2 \mu( Y_n^N )
    \, dr \, ds
  \right\|_{L^1}
\\&\leq
  \int_{t_n^N}^{t_{n+1}^N} \!\!\!
  \int_{t_n^N}^s
  \left\|
    \mathbbm{1}_{\Omega_{N,n+1}}
    u_3\!\left( t_{n+1}^N , \tilde{Y}_r^N \right)
    \left( \sigma( Y_n^N ) \right)^2
    \mu( Y_n^N )
  \right\|_{L^1}
  dr \, ds
\\&\leq
  \int_{t_n^N}^{t_{n+1}^N} \!\!\!
  \int_{t_n^N}^s
  2 L^3
  \left\|
    \mathbbm{1}_{\Omega_{N,n+1}}
    u_3\!\left( t_{n+1}^N , \tilde{Y}_r^N \right)
    \left( 1 + \left| Y_n^N \right|^2 \right)
    \left( 1 + \left| Y_n^N \right|^{\delta} \right)
  \right\|_{L^1}
  dr \, ds
\end{align*}
and
\begin{align*}
&
  \left\|
    \mathbbm{1}_{\Omega_{N,n+1}}
    \int_{t_n^N}^{t_{n+1}^N} \!\!\!
    \int_{t_n^N}^s
    u_3\!\left( t_{n+1}^N , \tilde{Y}_r^N \right)
    \left( \sigma( Y_n^N ) \right)^2 \mu( Y_n^N )
    \, dr \, ds
  \right\|_{L^1}
\\&\leq 
  \int_{t_n^N}^{t_{n+1}^N} \!\!\!
  \int_{t_n^N}^s \!
  2 L^3 R
  \left\|
    \mathbbm{1}_{\Omega_{N,n+1}}
    \left( 1 + \left| \tilde{Y}_r^N \right|^R \right)
    \left( 1 + \left| Y_n^N \right|^2 \right)
    \left( 1 + \left| Y_n^N \right|^{\delta} \right)
  \right\|_{L^1} \!\!\!
  dr \, ds
\\&\leq 
  \int_{t_n^N}^{t_{n+1}^N} \!\!\!
  \int_{t_n^N}^s \!
  2 L^3 R
  \left\|
    \mathbbm{1}_{\Omega_{N,n+1}}
    \left( 1 + \left| \tilde{Y}_r^N \right|^R \right)
    \left( 2 + \left| Y_n^N \right|^R \right)
    \left( 2 + \left| Y_n^N \right|^R \right)
  \right\|_{L^1} \!\!\!
  dr \, ds
\end{align*}
and
\begin{align}\label{int_2}
\nonumber
&
  \left\|
    \mathbbm{1}_{\Omega_{N,n+1}}
    \int_{t_n^N}^{t_{n+1}^N} \!\!\!
    \int_{t_n^N}^s
    u_3\!\left( t_{n+1}^N , \tilde{Y}_r^N \right)
    \left( \sigma( Y_n^N ) \right)^2 \mu( Y_n^N )
    \, dr \, ds
  \right\|_{L^1}
\\&\leq\nonumber 
  \int_{t_n^N}^{t_{n+1}^N} \!\!\!
  \int_{t_n^N}^s
  2 L^3 R
  \left\|
    \mathbbm{1}_{\Omega_{N,n+1}}
    \left( 2 + \left| \tilde{Y}_r^N \right|^R \right)
    \left( 2 + \left| Y_n^N \right|^R \right)^2
  \right\|_{L^1}
  dr \, ds
\\&\leq
  2 L^3 R
  \left(
    \sup_{ 0 \leq u \leq T }
    \left\|
      \mathbbm{1}_{\Omega_{N,\lfloor \frac{u N}{T} \rfloor}}
      \left( 2 + \left| \tilde{Y}_u^N \right|^R \right)
    \right\|_{L^3}^3
  \right)
  \frac{1}{2} 
  \left( \frac{T}{N} \right)^2
\\&\leq\nonumber
  L^3 R T^2 C^3 N^{-2}
  \leq
  C^9 N^{-2}
\end{align}
for every $ n \in \{ 0, 1, \ldots, N-1 \} $ and every
$ N \in \mathbb{N} $.
Next we use the estimates $|\sigma(x)|\leq L(1+|x|)$
and $(1+|x|)^4\leq 8(1+x^4)$ for every $x\in\mathbb{R}$
to obtain
\begin{equation}  \begin{split}
&
  \left\|
    \mathbbm{1}_{\Omega_{N,n+1}}
    \frac{1}{4}
    \int_{t_n^N}^{t_{n+1}^N} \!\!\!
    \int_{t_n^N}^s
    u_4\!\left( t_{n+1}^N , \tilde{Y}_r^N \right)
    \left( \sigma( Y_n^N ) \right)^4
    dr \, ds
  \right\|_{L^1}
\\&\leq
  \frac{1}{4} 
  \int_{t_n^N}^{t_{n+1}^N} \!\!\!
  \int_{t_n^N}^s
  \left\|
    \mathbbm{1}_{\Omega_{N,n+1}}
    u_4\!\left( t_{n+1}^N , \tilde{Y}_r^N \right)
    \left( \sigma( Y_n^N ) \right)^4
  \right\|_{L^1}
  dr \, ds
\\&\leq
  \int_{t_n^N}^{t_{n+1}^N} \!\!\!
  \int_{t_n^N}^s
  2 L^4
  \left\|
    \mathbbm{1}_{\Omega_{N,n+1}}
    u_4\!\left( t_{n+1}^N , \tilde{Y}_r^N \right)
    \left( 1 + \left| Y_n^N \right|^{4} \right)
  \right\|_{L^1}
  dr \, ds
\end{split}     \end{equation}
and
\begin{equation}  \begin{split}
&
  \left\|
    \mathbbm{1}_{\Omega_{N,n+1}}
    \frac{1}{4}
    \int_{t_n^N}^{t_{n+1}^N} \!\!\!
    \int_{t_n^N}^s
    u_4\!\left( t_{n+1}^N , \tilde{Y}_r^N \right)
    \left( \sigma( Y_n^N ) \right)^4
    dr \, ds
  \right\|_{L^1}
\\&\leq 
  2 L^4 R
  \int_{t_n^N}^{t_{n+1}^N} \!\!\!
  \int_{t_n^N}^s
  \left\|
    \mathbbm{1}_{\Omega_{N,n+1}}
    \left( 1 + \left| \tilde{Y}_r^N \right|^R \right)
    \left( 1 + \left| Y_n^N \right|^{4} \right)
  \right\|_{L^1}
  dr \, ds
\\&\leq 
  2 L^4 R
  \int_{t_n^N}^{t_{n+1}^N} \!\!\!
  \int_{t_n^N}^s
  \left\|
    \mathbbm{1}_{\Omega_{N,n+1}}
    \left( 1 + \left| \tilde{Y}_r^N \right|^R \right)
    \left( 2 + \left| Y_n^N \right|^R \right)
  \right\|_{L^1}
  dr \, ds
\end{split}     \end{equation}
and
\begin{align}\label{int_3}
&\nonumber
  \left\|
    \mathbbm{1}_{\Omega_{N,n+1}}
    \frac{1}{4}
    \int_{t_n^N}^{t_{n+1}^N} \!\!\!
    \int_{t_n^N}^s
    u_4\!\left( t_{n+1}^N , \tilde{Y}_r^N \right)
    \left( \sigma( Y_n^N ) \right)^4
    dr \, ds
  \right\|_{L^1}
\\&\leq\nonumber 
  2 L^4 R
  \int_{t_n^N}^{t_{n+1}^N} \!\!\!
  \int_{t_n^N}^s
  \left\|
    \mathbbm{1}_{\Omega_{N,n+1}}
    \left( 2 + \left| \tilde{Y}_r^N \right|^R \right)
  \right\|_{L^2}
  \left\|  
    \mathbbm{1}_{\Omega_{N,n+1}}
    \left( 2 + \left| Y_n^N \right|^R \right)
  \right\|_{L^2}
  dr \, ds
\\&\leq
  2 L^4 R
  \left(
    \sup_{ 0 \leq u \leq T }
    \left\|
      \mathbbm{1}_{\Omega_{N,\lfloor \frac{u N}{T} \rfloor}}
      \left( 2 + \left| \tilde{Y}_u^N \right|^R \right)
    \right\|_{L^2}^2
  \right)
  \frac{1}{2} 
  \left( \frac{T}{N} \right)^2
\\&\leq\nonumber
  L^4 R T^2 C^2 N^{-2}
  \leq
  C^9 N^{-2}
\end{align}
for every $ n \in \{ 0, 1, \ldots, N-1 \} $ and every
$ N \in \mathbb{N} $.
Combining \eqref{all_integrals}, \eqref{int_1}, \eqref{int_2}
and \eqref{int_3} hence yields
\begin{align*}
  &\bigg|
    \mathbb{E}\!\left[
      \mathbbm{1}_{\Omega_{N,n+1}}
      u\!\left( 
        t_{n+1}^N, Y_{n+1}^{N} 
      \right)
    \right]
    -
    \mathbb{E}\!\left[
      \mathbbm{1}_{\Omega_{N,n+1}} \left\{
        u\!\left(
          t_{n+1}^N,
          Y_n^N
        \right) +
        \frac{T}{N}
        \tilde{u}\!\left( 
          t_{n+1}^N, Y_n^N
        \right)
      \right\}
    \right]
  \bigg|
\\
  &\leq 
  3 C^9 N^{-2}
  +
  \left|
  \mathbb{E}\left[ \mathbbm{1}_{\Omega_{N,n+1}} 
  \int_{t_{n}^N}^{t_{n+1}^N} \!\!\!
  \int_{t_{n}^N}^{s}
  u_2\!\left(
    t_{n+1}^N, \tilde{Y}_{r}^N
  \right)
  \sigma( Y_n^N ) \, \mu( Y_n^N ) \, dW_r \, ds
  \right]
  \right|
\\
  &\qquad\qquad\quad +
  \left|
  \mathbb{E}\left[ \mathbbm{1}_{\Omega_{N,n+1}} 
  \int_{t_{n}^N}^{t_{n+1}^N}
  u_1\!\left(
    t_{n+1}^N, \tilde{Y}_{s}^N
  \right)
  \sigma( Y_n^N ) \, dW_s
  \right]
  \right|
\\
  &\qquad\qquad\quad +
  \left|
  \mathbb{E}\left[ \mathbbm{1}_{\Omega_{N,n+1}} 
  \frac{1}{2}
  \int_{t_{n}^N}^{t_{n+1}^N} \!\!\!
  \int_{t_{n}^N}^{s}
  u_3\!\left(
    t_{n+1}^N, \tilde{Y}_{r}^N
  \right)
  \left( \sigma( Y_n^N ) \right)^3  dW_r \, ds
  \right]
  \right|
\end{align*}
for every $ n \in \{ 0, 1, \ldots, N-1 \} $ and every
$ N \in \mathbb{N} $.
Due to $ \Omega_{N,n+1} \subset 
\Omega_{N,n} $ we have
\begin{equation}\label{einser}
  \mathbbm{1}_{ \Omega_{N,n+1} }
  =
  \mathbbm{1}_{\left(
    \Omega_{N,n+1} \cap
    \Omega_{N,n}
  \right)}
  =
  \mathbbm{1}_{\Omega_{N,n+1}}
  \cdot 
  \mathbbm{1}_{\Omega_{N,n}} 
\end{equation}
and therefore
\begin{align*}
  &\bigg|
    \mathbb{E}\!\left[
      \mathbbm{1}_{ \Omega_{N,n+1} }
      u\!\left( 
        t_{n+1}^N, Y_{n+1}^{N} 
      \right)
    \right]
    -
    \mathbb{E}\!\left[
      \mathbbm{1}_{ \Omega_{N,n+1} }
      \!\left\{
        u\!\left(
          t_{n+1}^N,
          Y_n^N
        \right) +
        \frac{T}{N}
        \tilde{u}\!\left( 
          t_{n+1}^N, Y_n^N
        \right)
      \right\}
    \right]
  \bigg|
\\
  &\leq 
  3 C^9 N^{-2}
   +
  \left|
  \mathbb{E}\left[ \mathbbm{1}_{\Omega_{N,n+1}} 
  \int_{t_{n}^N}^{t_{n+1}^N} \!\!\!
  \int_{t_{n}^N}^{s}
  \mathbbm{1}_{ \Omega_{N,n} }
  u_2\!\left(
    t_{n+1}^N, \tilde{Y}_{r}^N
  \right)
  \sigma( Y_n^N ) \, \mu( Y_n^N ) \, dW_r \, ds
  \right]
  \right|
\\
  &\qquad\qquad\quad +
  \left|
  \mathbb{E}\left[ \mathbbm{1}_{\Omega_{N,n+1}}
  \int_{t_{n}^N}^{t_{n+1}^N}
  \mathbbm{1}_{\Omega_{N,n}}
  u_1\!\left(
    t_{n+1}^N, \tilde{Y}_{s}^N
  \right)
  \sigma( Y_n^N ) \, dW_s
  \right]
  \right|
\\
  &\qquad\qquad\quad +
  \left|
  \mathbb{E}\left[ \mathbbm{1}_{\Omega_{N,n+1}}
  \frac{1}{2}
  \int_{t_{n}^N}^{t_{n+1}^N} \!\!\!
  \int_{t_{n}^N}^{s}
  \mathbbm{1}_{\Omega_{N,n}}
  u_3\!\left(
    t_{n+1}^N, \tilde{Y}_{r}^N
  \right)
  \left( \sigma( Y_n^N ) \right)^3  dW_r \, ds
  \right]
  \right|
\end{align*}
and, using that the expectation of every involved stochastic integral
is equal to zero,
\begin{align*}
  &\bigg|
    \mathbb{E}\!\left[
      \mathbbm{1}_{ \Omega_{N,n+1} }
      u\!\left( 
        t_{n+1}^N, Y_{n+1}^{N} 
      \right)
    \right]
    -
    \mathbb{E}\!\left[
      \mathbbm{1}_{ \Omega_{N,n+1} }\!
      \left\{
        u\!\left(
          t_{n+1}^N,
          Y_n^N
        \right) +
        \frac{T}{N}
        \tilde{u}\!\left( 
          t_{n+1}^N, Y_n^N
        \right)
      \right\}
    \right]
  \bigg|
\\
  &\leq 
  3 C^9 N^{-2}
   +
  \int_{t_{n}^N}^{t_{n+1}^N} \left|
  \mathbb{E}\left[ 
    \mathbbm{1}_{\left(\Omega_{N,n+1}\right)^c} 
  \int_{t_{n}^N}^{s}
  \mathbbm{1}_{\Omega_{N,n}}
  u_2\!\left(
    t_{n+1}^N, \tilde{Y}_{r}^N
  \right)
  \sigma( Y_n^N ) \, \mu( Y_n^N ) \, dW_r 
  \right] \right| ds
\\
  &\qquad\qquad\quad +
  \left|
  \mathbb{E}\left[ 
  \mathbbm{1}_{\left(\Omega_{N,n+1}\right)^c} 
  \int_{t_{n}^N}^{t_{n+1}^N}
  \mathbbm{1}_{\Omega_{N,n}}
  u_1\!\left(
    t_{n+1}^N, \tilde{Y}_{s}^N
  \right)
  \sigma( Y_n^N ) \, dW_s
  \right]
  \right|
\\
  &\qquad\qquad\quad +
  \frac{1}{2}
  \int_{t_{n}^N}^{t_{n+1}^N} \left|
  \mathbb{E}\left[ 
  \mathbbm{1}_{\left(\Omega_{N,n+1}\right)^c} 
  \int_{t_{n}^N}^{s}
  \mathbbm{1}_{\Omega_{N,n}}
  u_3\!\left(
    t_{n+1}^N, \tilde{Y}_{r}^N
  \right)
  \left( \sigma( Y_n^N ) \right)^3  dW_r 
  \right] \right| ds
\end{align*}
for every $ n \in \{ 0, 1, \ldots, N-1 \} $ and every
$ N \in \mathbb{N} $.
This implies
\begin{align*}
  &\bigg|
    \mathbb{E}\!\left[
      \mathbbm{1}_{ \Omega_{N,n+1} }
      u\!\left( 
        t_{n+1}^N, Y_{n+1}^{N} 
      \right)
    \right]
    -
    \mathbb{E}\!\left[
      \mathbbm{1}_{ \Omega_{N,n+1} } 
      \! \left\{
        u\!\left(
          t_{n+1}^N,
          Y_k^N
        \right) +
        \frac{T}{N}
        \tilde{u}\left( 
          t_{n+1}^N, Y_n^N
        \right)
      \right\}
    \right]
  \bigg|
\\
  &\leq 
  3 C^9 N^{-2}
  +
  \left(
    \mathbb{P} \bigg[ \left( \Omega_{N,n+1} \right)^c \bigg]
  \right)^{\frac{1}{2}}
  \int_{t_{n}^N}^{t_{n+1}^N}
  \left\|
  \int_{t_{n}^N}^{s}
  \mathbbm{1}_{\Omega_{N,n}}
  u_2\!\left(
    t_{n+1}^N, \tilde{Y}_{r}^N
  \right)
  \sigma( Y_n^N ) \, \mu( Y_n^N ) \, dW_r 
  \right\|_{L^2} ds
\\
  & +
  \left(
    \mathbb{P} \bigg[ \left( \Omega_{N,n+1} \right)^c \bigg]
  \right)^{\frac{1}{2}} 
  \left\|
  \int_{t_{n}^N}^{t_{n+1}^N}
  \mathbbm{1}_{\Omega_{N,n}}
  u_1\!\left(
    t_{n+1}^N, \tilde{Y}_{s}^N
  \right)
  \sigma( Y_n^N ) \, dW_s
  \right\|_{L^2}
\\
  & + 
  \frac{1}{2}
  \left(
    \mathbb{P} \bigg[ \left( \Omega_{N,n+1} \right)^c \bigg]
  \right)^{\frac{1}{2}}
  \int_{t_{n}^N}^{t_{n+1}^N}
  \left\|
  \int_{t_{n}^N}^{s}
  \mathbbm{1}_{\Omega_{N,n}}
  u_3\!\left(
    t_{n+1}^N, \tilde{Y}_{r}^N
  \right)
  \left( \sigma( Y_n^N ) \right)^3  dW_r 
  \right\|_{L^2} ds
\end{align*}
and, using $(\Omega_{N,n+1})^c\subseteq(\Omega_N)^c$
and the It\^o isometry,
\begin{align*}
  &\left|
    \mathbb{E}\!\left[
      \mathbbm{1}_{ \Omega_{N,n+1} }
      u\!\left( 
        t_{n+1}^N, Y_{n+1}^{N} 
      \right)
    \right]
    -
    \mathbb{E}\!\left[
      \mathbbm{1}_{ \Omega_{N,n+1} }
      \! \left\{
        u\!\left(
          t_{n+1}^N,
          Y_n^N
        \right) +
        \frac{T}{N}
        \tilde{u}\left( 
          t_{n+1}^N, Y_n^N
        \right)
      \right\}
    \right]
  \right|
\\&\leq 
  3 C^9 N^{-2}
  +
  \left(
    \mathbb{P} \bigg[ \left( \Omega_{N} \right)^c \bigg]
  \right)^{\frac{1}{2}}
  \int_{t_{n}^N}^{t_{n+1}^N}
  \left(
    \int_{t_{n}^N}^{s}
    \left\|  
      \mathbbm{1}_{\Omega_{N,n}}
      u_2\!\left(
        t_{n+1}^N, \tilde{Y}_{r}^N
      \right)
      \sigma( Y_n^N ) \, \mu( Y_n^N ) 
    \right\|_{L^2}^2 dr
  \right)^{\frac{1}{2}} ds
\\
  & +
  \left(
    \mathbb{P} \bigg[ \left( \Omega_{N} \right)^c \bigg]
  \right)^{\frac{1}{2}} 
  \left(
    \int_{t_{n}^N}^{t_{n+1}^N}
    \left\|  
      \mathbbm{1}_{\Omega_{N,n}}
        u_1\!\left(
          t_{n+1}^N, \tilde{Y}_{s}^N
        \right)
      \sigma( Y_n^N )
    \right\|_{L^2}^2 ds
  \right)^{\frac{1}{2}}
\\
  & + 
  \frac{1}{2}
  \left(
    \mathbb{P} \bigg[ \left( \Omega_{N} \right)^c \bigg]
  \right)^{\frac{1}{2}}
  \int_{t_{n}^N}^{t_{n+1}^N}
  \left(
    \int_{t_{n}^N}^{s}
    \left\|
      \mathbbm{1}_{\Omega_{N,n}}
      u_3\!\left(
        t_{n+1}^N, \tilde{Y}_{r}^N
      \right)
      \left( \sigma( Y_n^N ) \right)^3 
    \right\|_{L^2}^2 dr
  \right)^{\frac{1}{2}} ds
\end{align*}
for every $ n \in \{ 0, 1, \ldots, N-1 \} $ and every
$ N \in \mathbb{N} $.
Hence, \eqref{c_definition} shows
\begin{align*}
  &\left|
    \mathbb{E}\!\left[
      \mathbbm{1}_{ \Omega_{N,n+1} }
      u\!\left( 
        t_{n+1}^N, Y_{n+1}^{N} 
      \right)
    \right]
    -
    \mathbb{E}\!\left[
      \mathbbm{1}_{ \Omega_{N,n+1} }
      \! \left\{
        u\!\left(
          t_{n+1}^N,
          Y_n^N
        \right) +
        \frac{T}{N}
        \tilde{u}\left( 
          t_{n+1}^N, Y_n^N
        \right)
      \right\}
    \right]
  \right|
\\&\leq 
  3 C^9 N^{-2}
  +
  C N^{-2}
\\
  &\cdot
  \int_{t_{n}^N}^{t_{n+1}^N} \!\!
  \left(
    \int_{t_{n}^N}^{s} \!
    \left\|  
      \mathbbm{1}_{\Omega_{N,n}}
      R \left( 1 + \left| \tilde{Y}_{r}^N \right|^R \right)
      L \left( 1 + \left| Y_n^N \right| \right)
      L \left( 1 + \left| Y_n^N \right|^{\delta} \right) 
    \right\|_{L^2}^2 \!\! dr
  \right)^{\frac{1}{2}} \!\! ds
\\
  & +
  C N^{-2} 
  \left(
    \int_{t_{n}^N}^{t_{n+1}^N}
    \left\|  
      \mathbbm{1}_{\Omega_{N,n}}
      R \left( 1 + \left| \tilde{Y}_{s}^N \right|^R \right)
      L \left( 1 + \left| Y_n^N \right| \right)
    \right\|_{L^2}^2 ds
  \right)^{\frac{1}{2}}
\\
  & + 
  \frac{1}{2}
  C N^{-2}
  \int_{t_{n}^N}^{t_{n+1}^N} \!\!
  \left(
    \int_{t_{n}^N}^{s} \!
    \left\|
      \mathbbm{1}_{\Omega_{N,n}}
      R \left( 1 + \left| \tilde{Y}_{r}^N \right|^R \right)
      4 L^3 \left( 1 + \left| Y_n^N \right|^3 \right) 
    \right\|_{L^2}^2 \!\! dr
  \right)^{\frac{1}{2}} \!\! ds
\end{align*}
and
\begin{align*}
  &\left|
    \mathbb{E}\!\left[
      \mathbbm{1}_{ \Omega_{N,n+1} }
      u\!\left( 
        t_{n+1}^N, Y_{n+1}^{N} 
      \right)
    \right]
    -
    \mathbb{E}\!\left[
      \mathbbm{1}_{ \Omega_{N,n+1} }
      \! \left\{
        u\!\left(
          t_{n+1}^N,
          Y_n^N
        \right) +
        \frac{T}{N}
        \tilde{u}\left( 
          t_{n+1}^N, Y_n^N
        \right)
      \right\}
    \right]
  \right|
\\&\leq 
  3 C^9 N^{-2}
\\&
  + R L^2 C N^{-2}
  \int_{t_{n}^N}^{t_{n+1}^N} \!\!
  \left(
    \int_{t_{n}^N}^{s} \!
    \left\|  
      \mathbbm{1}_{\Omega_{N,n}}
      \left( 2 + \left| \tilde{Y}_{r}^N \right|^R \right) \!
      \left( 2 + \left| Y_n^N \right|^R \right)^2
    \right\|_{L^2}^2 dr
  \right)^{\frac{1}{2}} ds
\\
  & +
  R L C N^{-2}
  \left(
    \int_{t_{n}^N}^{t_{n+1}^N} 
    \left\|  
      \mathbbm{1}_{\Omega_{N,n}}
      \left( 2 + \left| \tilde{Y}_{s}^N \right|^R \right)
      \left( 2 + \left| Y_n^N \right|^R \right)
    \right\|_{L^2}^2 ds
  \right)^{\frac{1}{2}}
\\
  & + 
  2
  R L^3 C N^{-2} 
  \int_{t_{n}^N}^{t_{n+1}^N} \!\!
  \left(
    \int_{t_{n}^N}^{s} \!
    \left\|
      \mathbbm{1}_{\Omega_{N,n}}
      \left( 2 + \left| \tilde{Y}_{r}^N \right|^R \right) \!
      \left( 2 + \left| Y_n^N \right|^R \right) 
    \right\|_{L^2}^2 dr
  \right)^{\frac{1}{2}} ds
\end{align*}
for every $ n \in \{ 0, 1, \ldots, N-1 \} $ and every
$ N \in \mathbb{N} $.
Therefore, we have
\begin{align*}
  &\left|
    \mathbb{E}\!\left[
      \mathbbm{1}_{ \Omega_{N,n+1} }
      u\!\left( 
        t_{n+1}^N, Y_{n+1}^{N} 
      \right)
    \right]
    -
    \mathbb{E}\!\left[
      \mathbbm{1}_{ \Omega_{N,n+1} }
      \! \left\{
        u\!\left(
          t_{n+1}^N,
          Y_n^N
        \right) +
        \frac{T}{N}
        \tilde{u}\left( 
          t_{n+1}^N, Y_n^N
        \right)
      \right\}
    \right]
  \right|
\\&\leq 
  3 C^9 N^{-2}
  + R L^2 C N^{-2}
  \int_{t_{n}^N}^{t_{n+1}^N}
  \left(
    \int_{t_{n}^N}^{s}
    \left(
      \sup_{ 0 \leq u \leq T }
      \left\|  
        \mathbbm{1}_{
          \Omega_{N,\lfloor \frac{u N}{T} \rfloor}
        }
        \left( 2 + \left| \tilde{Y}_{u}^N \right|^R \right)
      \right\|_{L^6}^2 
    \right) dr
  \right)^{\frac{1}{2}} ds
\\
  &\qquad\qquad\quad +
  R L C N^{-2}
  \left(
    \int_{t_{n}^N}^{t_{n+1}^N}
    \left(
      \sup_{ 0 \leq u \leq T }
      \left\|  
        \mathbbm{1}_{ 
          \Omega_{N,\lfloor \frac{u N}{T} \rfloor}
        }
        \left( 2 + \left| \tilde{Y}_{u}^N \right|^R \right)
      \right\|_{L^4}^2 
    \right) ds
  \right)^{\frac{1}{2}}
\\
  &\qquad\qquad\quad +
  2 R L^3 C N^{-2}
  \int_{t_{n}^N}^{t_{n+1}^N}
  \left(
    \int_{t_{n}^N}^{s}
    \left(
      \sup_{ 0 \leq u \leq T }
      \left\|
        \mathbbm{1}_{
          \Omega_{N,\lfloor \frac{u N}{T} \rfloor}
        }
        \left( 2 + \left| \tilde{Y}_{u}^N \right|^R \right)
      \right\|_{L^4}^2 
    \right) dr
  \right)^{\frac{1}{2}} ds
\\
&\leq
  3 C^9 N^{-2}
  +
  R L^2 C N^{-2}
  \int_{t_{n}^N}^{t_{n+1}^N}
  \sqrt{T} C
  \, ds
  +
  R L C N^{-2}
  \sqrt{T} C
  + 
  2 R L^3 C N^{-2}
  \int_{t_{n}^N}^{t_{n+1}^N}
  \sqrt{T} C \, ds
\end{align*}
for every $ n \in \{ 0, 1, \ldots, N-1 \} $ and every
$ N \in \mathbb{N} $.
This finally shows that
\begin{align}\label{numeric_expansion}
\nonumber
  &\left|
    \mathbb{E}\!\left[
      \mathbbm{1}_{ \Omega_{N,n+1} }
      u\!\left( 
        t_{n+1}^N, Y_{n+1}^{N} 
      \right)
    \right]
    -
    \mathbb{E}\!\left[
      \mathbbm{1}_{ \Omega_{N,n+1} }
      \! \left\{
        u\!\left(
          t_{n+1}^N,
          Y_n^N
        \right) +
        \frac{T}{N}
        \tilde{u}\left( 
          t_{n+1}^N, Y_n^N
        \right)
      \right\}
    \right]
  \right|
\\&\leq 
  3 C^9 N^{-2}
  +
  R L^2 C^2 \sqrt{T} T N^{-2}
  +
  R L C^2 \sqrt{T} N^{-2}
  + 
  2 R L^3 C^2 \sqrt{T} T N^{-2}
\\\nonumber
  &\leq 
  3 C^9 N^{-2}
  +
  C^7 N^{-2}
  +
  C^5 N^{-2}
  + 
  2 C^8 N^{-2}
  \leq 7 C^9 N^{-2}
\end{align}
for every $ n \in \{ 0, 1, \ldots, N-1 \} $ and every
$ N \in \mathbb{N} $.

Next we aim at a similar estimate as~\eqref{numeric_expansion}
with $Y_{n+1}^N$ replaced by 
$X_{ t_{n+1}^N }^{ t_n^N,Y_n^N }$
for $ n \in \left\{ 0,1, \dots, N-1 \right\}$
and $N \in \mathbb{N}$.
It\^o's formula implies
\begin{equation}  \begin{split}
  &u\!\left( 
    t_{n+1}^N, X_{t_{n+1}^N}^{t_{n}^N, Y_n^N} 
  \right) 
\\
  &=
  u\!\left( 
    t_{n+1}^N, Y_n^N
  \right)
  +
  \int_{t_{n}^N}^{t_{n+1}^N}
  u_1\!\left( 
    t_{n+1}^N, X_{s}^{t_{n}^N, Y_n^N}
  \right) 
  \mu\!\left(X_{s}^{t_{n}^N, Y_n^N}\right) ds
\\
  &\quad +
  \int_{t_{n}^N}^{t_{n+1}^N}
  u_1\!\left( 
    t_{n+1}^N, X_{s}^{t_{n}^N, Y_n^N}
  \right) 
  \sigma\!\left(X_{s}^{t_{n}^N, Y_n^N}\right) dW_s
\\
  &\quad + \frac{1}{2}
  \int_{t_{n}^N}^{t_{n+1}^N}
  u_2\!\left( 
    t_{n+1}^N, X_{s}^{t_{n}^N, Y_n^N}
  \right)
  \left( 
    \sigma\!\left(X_{s}^{t_{n}^N, Y_n^N}\right)
  \right)^2 ds
  \qquad \mathbb{P}\text{-a.s.}
\end{split}     \end{equation}
for every $ n \in \{0,1,\ldots,N-1\} $ and every
$ N \in \mathbb{N} $. This shows
\begin{equation*}  \begin{split}
  u\!\left( 
    t_{n+1}^N, X_{t_{n+1}^N}^{t_{n}^N, Y_n^N} 
  \right) 
&=
  u\!\left( 
    t_{n+1}^N, Y_n^N
  \right)
  +
  \int_{t_{n}^N}^{t_{n+1}^N}
  \tilde{u}\!\left( 
    t_{n+1}^N, X_{s}^{t_{n}^N, Y_n^N}
  \right) ds
\\ &\quad+
  \int_{t_{n}^N}^{t_{n+1}^N}
  u_1\!\left( 
    t_{n+1}^N, X_{s}^{t_{n}^N, Y_n^N}
  \right) 
  \sigma\!\left(X_{s}^{t_{n}^N, Y_n^N}\right) dW_s
  \quad \mathbb{P}\text{-a.s.}
\end{split}     \end{equation*}
and
\begin{equation}  \begin{split}
  &u\!\left( 
    t_{n+1}^N, X_{t_{n+1}^N}^{t_{n}^N, Y_n^N} 
  \right) 
\\
  &=
  u\!\left(
    t_{n+1}^N,
    Y_n^N
  \right) +
  \frac{T}{N}
  \tilde{u}\!\left( 
    t_{n+1}^N, Y_n^N
  \right)
  +
  \int_{t_{n}^N}^{t_{n+1}^N}
  \int_{t_{n}^N}^{s}
  \tilde{\tilde{u}}\!\left( 
    t_{n+1}^N, X_{r}^{t_{n}^N, Y_n^N}
  \right) dr \, ds
\\
  &\quad +
  \int_{t_{n}^N}^{t_{n+1}^N}
  \int_{t_{n}^N}^{s}
  \left(\frac{\partial}{\partial x}\tilde{u}\right)\!\left( 
    t_{n+1}^N, X_{r}^{t_{n}^N, Y_n^N}
  \right) 
  \sigma\!\left(
    X_r^{ t_{n}^N, Y_n^N}
  \right) dW_r \, ds  
\\
  &\quad +
  \int_{t_{n}^N}^{t_{n+1}^N}
  u_1\!\left( 
    t_{n+1}^N, X_{s}^{t_{n}^N, Y_n^N}
  \right) 
  \sigma\!\left(X_{s}^{t_{n}^N, Y_n^N}\right) dW_s
  \qquad \mathbb{P}\text{-a.s.}
\end{split}     \end{equation}
for every $ n \in \{0,1,\ldots,N-1\} $ and every
$ N \in \mathbb{N} $ by It{\^o}'s formula.
Hence, we obtain
\begin{align*}
  &\left|
    \mathbb{E}\!\left[
      \mathbbm{1}_{\Omega_{N,n+1}}
      u\!\left( \!
        t_{n+1}^N, X_{t_{n+1}^N}^{t_{n}^N, Y_n^N} \! 
      \right)
    \right]
    -
    \mathbb{E}\!\left[
      \mathbbm{1}_{\Omega_{N,n+1}} \! \left\{ \!
        u\!\left( \!
          t_{n+1}^N,
          Y_n^N \!
        \right) +
        \frac{T}{N}
        \tilde{u}\!\left(\!
          t_{n+1}^N, Y_n^N \!
        \right) \!
      \right\}
    \right] 
  \right|
\\
  &\leq
  \int_{t_{n}^N}^{t_{n+1}^N}
  \int_{t_{n}^N}^{s}
  \mathbb{E}\left[
    \mathbbm{1}_{\Omega_{N,n+1}} \left|
      \tilde{\tilde{u}}\!\left( 
        t_{n+1}^N, X_{r}^{t_{n}^N, Y_n^N}
      \right)
    \right|
  \right] dr \, ds
\\
  &+
  \int_{t_{n}^N}^{t_{n+1}^N}
  \left|
    \mathbb{E}\left[
      \mathbbm{1}_{\Omega_{N,n+1}}
      \int_{t_{n}^N}^{s}
      \left(\frac{\partial}{ \partial x}\tilde{u}\right)\!\left( 
        t_{n+1}^N, X_{r}^{t_{n}^N, Y_n^N}
      \right) 
      \sigma\!\left(
        X_r^{ t_{n}^N, Y_n^N}
      \right) dW_r 
  \right] \right| ds
\\ 
  &+
  \left|
    \mathbb{E}\left[
      \mathbbm{1}_{\Omega_{N,n+1}}
      \int_{t_{n}^N}^{t_{n+1}^N}
      u_1\!\left( 
        t_{n+1}^N, X_{s}^{t_{n}^N, Y_n^N}
      \right) 
      \sigma\!\left(X_{s}^{t_{n}^N, Y_n^N}\right) dW_s
    \right]
  \right|
\end{align*}
for every $ n \in \{0,1,\ldots,N-1\} $ and every
$ N \in \mathbb{N} $.
By \eqref{einser} we obtain
\begin{align*}
  &\left|
    \mathbb{E}\!\left[
      \mathbbm{1}_{\Omega_{N,n+1}} 
      u\!\left( \!
        t_{n+1}^N, X_{t_{n+1}^N}^{t_{n}^N, Y_n^N} \!
      \right)
    \right]
    -
    \mathbb{E}\!\left[
      \mathbbm{1}_{\Omega_{N,n+1}} \! \left\{ \!
        u\!\left( \!
          t_{n+1}^N,
          Y_n^N \!
        \right) +
        \frac{T}{N}
        \tilde{u}\!\left( \!
          t_{n+1}^N, Y_n^N \!
        \right) \!
      \right\}
    \right]
  \right|
\\
  &\leq
  \int_{t_{n}^N}^{t_{n+1}^N}
  \int_{t_{n}^N}^{s}
  \mathbb{E}\left[
    \mathbbm{1}_{\Omega_{N,n+1}}
    \,R\left(
      1
      +
      \left|
        X_r^{t_{n}^N, Y_n^N}
      \right|^R
    \right)
  \right] dr \, ds
\\
  &+
  \int_{t_{n}^N}^{t_{n+1}^N}
  \left|
    \mathbb{E}\left[
      \mathbbm{1}_{\Omega_{N,n+1}}
      \!\int_{t_{n}^N}^{s}
      \mathbbm{1}_{\Omega_{N,n}}
      \left(\frac{\partial}{ \partial x}\tilde{u}\right)\!\!\left( 
        t_{n+1}^N, X_{r}^{t_{n}^N, Y_n^N}
      \right) 
      \sigma\!\left( \!
        X_r^{ t_{n}^N, Y_n^N}
      \right) \! dW_r
    \right]
  \right| ds
\\ 
  &+
  \left|
    \mathbb{E}\left[
      \mathbbm{1}_{\Omega_{N,n+1}}
      \int_{t_{n}^N}^{t_{n+1}^N}
      \mathbbm{1}_{\Omega_{N,n}}
      u_1\!\left( 
        t_{n+1}^N, X_{s}^{t_{n}^N, Y_n^N}
      \right) 
      \sigma\!\left(\!X_{s}^{t_{n}^N, Y_n^N}\right) dW_s
    \right]
  \right|
\end{align*}
and
\begin{align*}
  &\left|
    \mathbb{E}\!\left[
      \mathbbm{1}_{\Omega_{N,n+1}}
      u\!\left( \!
        t_{n+1}^N, X_{t_{n+1}^N}^{t_{n}^N, Y_n^N} \! 
      \right)
    \right]
    -
    \mathbb{E}\!\left[
      \mathbbm{1}_{\Omega_{N,n+1}} \! \left\{ \!
        u\!\left( \!
          t_{n+1}^N,
          Y_n^N \!
        \right) +
        \frac{T}{N}
        \tilde{u}\!\left(\!
          t_{n+1}^N, Y_n^N \!
        \right) \!
      \right\}
    \right] 
  \right|
\\&\leq
  R C \frac{1}{2} \left( 
    \frac{T}{N}
  \right)^2
\\
  &+ \!\!
  \int_{t_{n}^N}^{t_{n+1}^N} \!
  \left|
  \mathbb{E}\left[
    \mathbbm{1}_{\left(\Omega_{N,n+1}\right)^c} \!\!
      \int_{t_{n}^N}^{s} \!
      \mathbbm{1}_{\Omega_{N,n}} \!
      \left(\frac{\partial}{ \partial x}\tilde{u}\right) \! \left( 
        \! t_{n+1}^N, X_{r}^{t_{n}^N, Y_n^N} \!
      \right) 
      \sigma\!\left( 
        \! X_r^{ t_{n}^N, Y_n^N} \!
      \right) dW_r
  \right] \right| ds
\\ 
  &+
  \left|
  \mathbb{E}\left[
    \mathbbm{1}_{\left(\Omega_{N,n+1}\right)^c}
      \int_{t_{n}^N}^{t_{n+1}^N}
      \mathbbm{1}_{\Omega_{N,n}}
      u_1\!\left( 
        t_{n+1}^N, X_{s}^{t_{n}^N, Y_n^N}
      \right) 
      \sigma\!\left(\!X_{s}^{t_{n}^N, Y_n^N}\right) dW_s
  \right] \right|
\end{align*}
and
\begin{align*}
  \nonumber
  &\left|
    \mathbb{E}\!\left[
      \mathbbm{1}_{\Omega_{N,n+1}}
      u\!\left( \!
        t_{n+1}^N, X_{t_{n+1}^N}^{t_{n}^N, Y_n^N} \! 
      \right)
    \right]
    -
    \mathbb{E}\!\left[
      \mathbbm{1}_{\Omega_{N,n+1}} \! \left\{ \!
        u\!\left( \!
          t_{n+1}^N,
          Y_n^N \!
        \right) +
        \frac{T}{N}
        \tilde{u}\!\left(\!
          t_{n+1}^N, Y_n^N \!
        \right) \!
      \right\}
    \right] 
  \right|
\\ \nonumber
  &\leq
  \frac{1}{2} R C T^2 N^{-2}
\\ \nonumber
  &+
  \left(
    \mathbb{P} \left[ \left( \Omega_N \right)^c \right]
  \right)^{\frac{1}{2}} \!\!
  \int_{t_{n}^N}^{t_{n+1}^N} \!\!
    \left\|
      \int_{t_{n}^N}^{s} \!
      \mathbbm{1}_{\Omega_{N,n}} \!
      \left(\frac{\partial}{ \partial x}\tilde{u}\right) \!\left( 
        \! t_{n+1}^N, X_{r}^{t_{n}^N, Y_n^N} \!
      \right) 
      \sigma\!\left(
        \! X_r^{ t_{n}^N, Y_n^N} \!
      \right) dW_r
    \right\|_{L^2} \!\!\!\! ds
\\ \nonumber
  &+
  \left(
    \mathbb{P} \left[ \left( \Omega_N \right)^c \right]
  \right)^{\frac{1}{2}}
  \left\|
    \int_{t_{n}^N}^{t_{n+1}^N}
    \mathbbm{1}_{\Omega_{N,n}}
    u_1\!\left( 
      t_{n+1}^N, X_{s}^{t_{n}^N, Y_n^N}
    \right) 
    \sigma\!\left(\!X_{s}^{t_{n}^N, Y_n^N}\!\right) dW_s
  \right\|_{L^2}
\end{align*}
for every $ n \in \{0,1,\ldots,N-1\} $ and every
$ N \in \mathbb{N} $.
Therefore, we have
\begin{align*}
  \nonumber
  &\left|
    \mathbb{E}\!\left[
      \mathbbm{1}_{\Omega_{N,n+1}}
      u\!\left( \!
        t_{n+1}^N, X_{t_{n+1}^N}^{t_{n}^N, Y_n^N} \! 
      \right)
    \right]
    -
    \mathbb{E}\!\left[
      \mathbbm{1}_{\Omega_{N,n+1}} \! \left\{ \!
        u\!\left( \!
          t_{n+1}^N,
          Y_n^N \!
        \right) +
        \frac{T}{N}
        \tilde{u}\!\left(\!
          t_{n+1}^N, Y_n^N \!
        \right) \!
      \right\}
    \right] 
  \right|
\\&\leq
\nonumber
  \left(
    \mathbb{P} \left[ \left( \Omega_N \right)^c \right]
  \right)^{\frac{1}{2}} \!\!
  \int_{t_{n}^N}^{t_{n+1}^N} \!\!
  \left(
    \int_{t_{n}^N}^{s} \!
    \left\|
      \mathbbm{1}_{\Omega_{N,n}} \!
      \left(\frac{\partial}{ \partial x}\tilde{u}\right) \!\left( 
        \! t_{n+1}^N, X_{r}^{t_{n}^N, Y_n^N} \!
      \right) 
      \sigma\!\left(
        \! X_r^{ t_{n}^N, Y_n^N} \!
      \right) 
     \right\|_{L^2}^2 \!\!\!\! dr
   \right)^{\frac{1}{2}} \!\! ds
\\ \nonumber
  &+
  \left(
    \mathbb{P} \left[ \left( \Omega_N \right)^c \right]
  \right)^{\frac{1}{2}}
  \left(
    \int_{t_{n}^N}^{t_{n+1}^N}
    \left\|
      \mathbbm{1}_{\Omega_{N,n}}
      u_1\!\left( 
        t_{n+1}^N, X_{s}^{t_{n}^N, Y_n^N}
      \right) 
      \sigma\!\left(\!X_{s}^{t_{n}^N, Y_n^N}\!\right)
    \right\|_{L^2}^2 ds
  \right)^{\frac{1}{2}}
  + \frac{ C^4 }{ N^2 }
\end{align*}
and
\begin{align*}
  &\left|
    \mathbb{E}\!\left[
      \mathbbm{1}_{\Omega_{N,n+1}}
      u\!\left( \!
        t_{n+1}^N, X_{t_{n+1}^N}^{t_{n}^N, Y_n^N} \! 
      \right)
    \right]
    -
    \mathbb{E}\!\left[
      \mathbbm{1}_{\Omega_{N,n+1}} \! \left\{ \!
        u\!\left( \!
          t_{n+1}^N,
          Y_n^N \!
        \right) +
        \frac{T}{N}
        \tilde{u}\!\left(\!
          t_{n+1}^N, Y_n^N \!
        \right) \!
      \right\}
    \right] 
  \right|
\\&\leq
  \frac{ C^4 }{ N^2 } +
  \left(
    \mathbb{P} \left[ \left( \Omega_N \right)^c \right]
  \right)^{\frac{1}{2}}
\\ \nonumber
  &\cdot
  \int_{t_{n}^N}^{t_{n+1}^N} \!\!
  \left(
    \int_{t_{n}^N}^{s} \!
    \left\|
      \mathbbm{1}_{\Omega_{N,n}} \!
      R\left( 
        1 + \left|X_{r}^{t_{n}^N, Y_n^N}\right|^R
      \right) 
      L\left(
        1 +  \left|X_r^{ t_{n}^N, Y_n^N}\right|
      \right) 
     \right\|_{L^2}^2 \!\!\!\! dr
   \right)^{\frac{1}{2}} ds
\\ \nonumber
  &+
  \left(
    \mathbb{P} \left[ \left( \Omega_N \right)^c \right]
  \right)^{\frac{1}{2}}
  \left(
    \int_{t_{n}^N}^{t_{n+1}^N} \!\!
    \left\|
      \mathbbm{1}_{\Omega_{N,n}}
      R\left( 
        1 + \left|X_{s}^{t_{n}^N, Y_n^N}\right|^R
      \right) 
      L\left(
        1 +  \left|X_s^{ t_{n}^N, Y_n^N}\right|
      \right) 
    \right\|_{L^2}^2 \!\! ds
  \right)^{\frac{1}{2}}
\end{align*}
and
\begin{align*}
  \nonumber
  &\left|
    \mathbb{E}\!\left[
      \mathbbm{1}_{\Omega_{N,n+1}}
      u\!\left( \!
        t_{n+1}^N, X_{t_{n+1}^N}^{t_{n}^N, Y_n^N} \! 
      \right)
    \right]
    -
    \mathbb{E}\!\left[
      \mathbbm{1}_{\Omega_{N,n+1}} \! \left\{ \!
        u\!\left( \!
          t_{n+1}^N,
          Y_n^N \!
        \right) +
        \frac{T}{N}
        \tilde{u}\!\left(\!
          t_{n+1}^N, Y_n^N \!
        \right) \!
      \right\}
    \right] 
  \right|
\\&\leq
  \frac{ C^4 }{ N^2 }
  +
  \left(
    \mathbb{P} \left[ \left( \Omega_N \right)^c \right]
  \right)^{\frac{1}{2}}
  L R
  \int_{t_{n}^N}^{t_{n+1}^N} \!\!
  \left(
    \int_{t_{n}^N}^{s} \!
    \left\|
      \mathbbm{1}_{\Omega_{N,n}} \!
      \left( 
        2 + \left|X_{r}^{t_{n}^N, Y_n^N}\right|^R
      \right)^2 
     \right\|_{L^2}^2 \!\!\!\! dr
   \right)^{\frac{1}{2}} ds
\\ \nonumber
  &+
  \left(
    \mathbb{P} \left[ \left( \Omega_N \right)^c \right]
  \right)^{\frac{1}{2}}
  L R
  \left(
    \int_{t_{n}^N}^{t_{n+1}^N}
    \left\|
      \mathbbm{1}_{\Omega_{N,n}}
      \left( 
        2 + \left|X_{s}^{t_{n}^N, Y_n^N}\right|^R
      \right)^2 
    \right\|_{L^2}^2 ds
  \right)^{\frac{1}{2}}
\end{align*}
for every $ n \in \{0,1,\ldots,N-1\} $ and every
$ N \in \mathbb{N} $.
Hence, we finally obtain
\begin{align*}
  \nonumber
  &\left|
    \mathbb{E}\!\left[
      \mathbbm{1}_{\Omega_{N,n+1}}
      u\!\left( \!
        t_{n+1}^N, X_{t_{n+1}^N}^{t_{n}^N, Y_n^N} \! 
      \right)
    \right]
    -
    \mathbb{E}\!\left[
      \mathbbm{1}_{\Omega_{N,n+1}} \! \left\{ \!
        u\!\left( \!
          t_{n+1}^N,
          Y_n^N \!
        \right) +
        \frac{T}{N}
        \tilde{u}\!\left(\!
          t_{n+1}^N, Y_n^N \!
        \right) \!
      \right\}
    \right] 
  \right|
\\&\leq
  \frac{ C^4 }{ N^2 }
  +
  \left(
    \mathbb{P} \left[ \left( \Omega_N \right)^c \right]
  \right)^{\frac{1}{2}}
  L R
  \int_{t_{n}^N}^{t_{n+1}^N} \!\!
  \left(
    \int_{t_{n}^N}^{s} \!
    \left\|
      \mathbbm{1}_{\Omega_{N,n}} \!
      \left( 
        2 + \left|X_{r}^{t_{n}^N, Y_n^N}\right|^R
      \right)
     \right\|_{L^4}^4 \!\!\!\! dr
   \right)^{\frac{1}{2}} ds
\\ \nonumber
  &+
  \left(
    \mathbb{P} \left[ \left( \Omega_N \right)^c \right]
  \right)^{\frac{1}{2}}
  L R
  \left(
    \int_{t_{n}^N}^{t_{n+1}^N}
    \left\|
      \mathbbm{1}_{\Omega_{N,n}}
      \left( 
        2 + \left|X_{s}^{t_{n}^N, Y_n^N}\right|^R
      \right)
    \right\|_{L^4}^4 ds
  \right)^{\frac{1}{2}}
\end{align*}
and
\begin{align*}
  \nonumber
  &\left|
    \mathbb{E}\!\left[
      \mathbbm{1}_{\Omega_{N,n+1}}
      u\!\left( \!
        t_{n+1}^N, X_{t_{n+1}^N}^{t_{n}^N, Y_n^N} \! 
      \right)
    \right]
    -
    \mathbb{E}\!\left[
      \mathbbm{1}_{\Omega_{N,n+1}} \! \left\{ \!
        u\!\left( \!
          t_{n+1}^N,
          Y_n^N \!
        \right) +
        \frac{T}{N}
        \tilde{u}\!\left(\!
          t_{n+1}^N, Y_n^N \!
        \right) \!
      \right\}
    \right] 
  \right|
\\&\leq
  C^4 N^{-2}
  +
  C N^{-2}
  L R
  \int_{t_{n}^N}^{t_{n+1}^N} \!\!
  \left(  
    \int_{t_{n}^N}^{s} 
    C^4 \, dr
   \right)^{\frac{1}{2}} ds
  +
  C N^{-2}
  L R
  \left(
    \int_{t_{n}^N}^{t_{n+1}^N}
    C^4 \, ds
  \right)^{\frac{1}{2}}
\\&\leq 
  C^4 N^{-2}
  +
  C N^{-2}
  L R
  \int_{t_{n}^N}^{t_{n+1}^N} \!\!
  \sqrt{T} C^2 ds
  +
  L R C^3 \sqrt{T} N^{-2}
\end{align*}
and hence
\begin{align}\label{exact_exp}
  \nonumber
  &\left|
    \mathbb{E}\!\left[
      \mathbbm{1}_{\Omega_{N,n+1}}
      u\!\left( \!
        t_{n+1}^N, X_{t_{n+1}^N}^{t_{n}^N, Y_n^N} \! 
      \right)
    \right]
    -
    \mathbb{E}\!\left[
      \mathbbm{1}_{\Omega_{N,n+1}} \! \left\{ \!
        u\!\left( \!
          t_{n+1}^N,
          Y_n^N \!
        \right) +
        \frac{T}{N}
        \tilde{u}\!\left(\!
          t_{n+1}^N, Y_n^N \!
        \right) \!
      \right\}
    \right] 
  \right|
\\&\leq 
\nonumber
  C^4 N^{-2}
  +
  L R C^3 T \sqrt{T} N^{-2}
  +
  L R C^3 \sqrt{T} N^{-2}
\\&\leq 
  C^4 N^{-2}
  +
  C^7 N^{-2}
  +
  C^6 N^{-2}
  \leq
  3 C^7 N^{-2}
\end{align}
for every $ n \in \{0,1,\ldots,N-1\} $ and every
$ N \in \mathbb{N} $.
Combining \eqref{numeric_expansion} and \eqref{exact_exp} yields
\begin{equation}\label{e_estimates}
  \left|
  \mathbb{E}\left[
    \mathbbm{1}_{\Omega_{N,n+1}}
    u\!\left( 
      t_{n+1}^N, Y_{n+1}^{N} 
    \right)
  \right]
  -
  \mathbb{E}\left[
    \mathbbm{1}_{\Omega_{N,n+1}}
    u\!\left( 
      t_{n+1}^N, X_{t_{n+1}^N}^{t_{n}^N, Y_n^N} 
    \right)
  \right]
  \right|
  \leq 
  10 C^9 N^{-2} 
\end{equation}
for every $ n \in \{0,1,\ldots,N-1\} $ and every
$ N \in \mathbb{N} $.
Our interpretation of~\eqref{e_estimates} is that
the difference between the Euler approximation
and the exact solution after a time of order $O(\tfrac{1}{N})$
is in a weak sense of order $O(\tfrac{1}{N^2})$.
Now we split up the interval $[0,T]$ into 
$N \in \mathbb{N}$ subintervals
and sum up all differences which arise in the subintervals.
Rewriting the weak difference between the exact solution
and the Euler approximation by a telescope sum
yields
\begin{equation}  \begin{split}\label{fteleskop}
&
  \mathbb{E} 
    \bigg[ \mathbbm{1}_{\Omega _N } 
      \cdot f\left(
        X_T
      \right)
    \bigg] -
  \mathbb{E} 
    \bigg[ \mathbbm{1}_{\Omega _N } 
      \cdot f\left(
        Y^N_N
      \right)
    \bigg]
\\& = 
  \mathbb{E} 
    \left[ \mathbbm{1}_{\Omega _N } 
      \cdot f\left(
        X^{0, Y^N_0}_T
      \right)
    \right] -
  \mathbb{E} 
    \left[ \mathbbm{1}_{\Omega _N } 
      \cdot f\left(
        X^{T, Y^N_N}_T
      \right)
    \right]
\\&= 
  \mathbb{E} 
    \left[ \mathbbm{1}_{\Omega _N } 
      \cdot f\left(
        X^{t^N_0, Y^N_0}_T
      \right)
    \right] -
  \mathbb{E} 
    \left[ \mathbbm{1}_{\Omega _N } 
      \cdot f\left(
        X^{t_{1}^N, Y^N_1}_T
      \right)
    \right]
\\& \quad + 
  \mathbb{E} 
    \left[ \mathbbm{1}_{\Omega _N } 
      \cdot f\left(
        X^{t_{1}^N, Y^N_1}_T
      \right)
    \right] -
  \mathbb{E} 
    \left[ \mathbbm{1}_{\Omega _N } 
      \cdot f\left(
        X^{t^N_N, Y^N_N}_T
      \right)
    \right]
\end{split}     \end{equation}
and hence
\begin{equation}  \begin{split}
\lefteqn{
  \mathbb{E} 
    \bigg[ \mathbbm{1}_{\Omega _N } 
      \cdot f\left(
        X_T
      \right)
    \bigg] -
  \mathbb{E} 
    \bigg[ \mathbbm{1}_{\Omega _N } 
      \cdot f\left(
        Y^N_N
      \right)
    \bigg]
}
\\
&=
  \sum\limits_{n \, = \, 0}^{N - 1}
   \left( 
     \mathbb{E} 
     \left[ \mathbbm{1}_{\Omega _N } 
      \cdot f\left(
        X^{t_{n}^N, Y^N_n}_T
      \right)
    \right] -
  \mathbb{E} 
    \left[ \mathbbm{1}_{\Omega _N } 
      \cdot 
      f \left(
        X^{t_{n+1}^N, Y^N_{n + 1}}_T
      \right)
    \right]
  \right)
\end{split}     \end{equation}
for every $N$ $\in$ $\mathbb N$.
Since $ \Omega_N \subset \Omega_{N,n+1} $,
we have
\begin{equation}
  \mathbbm{1}_{\Omega_N}
  =
  \mathbbm{1}_{\Omega_{N,n+1}}
  - 
  \mathbbm{1}_{\left(\Omega_{N,n+1} \backslash \Omega_{N}\right)} 
\end{equation}
for every $ n \in \{0,1,\ldots,N-1\} $
and every $ N \in \mathbb{N} $
and therefore
\begin{equation}  \begin{split}
&
  \mathbb{E} 
  \bigg[ \mathbbm{1}_{\Omega _N } 
    \cdot f\left(
      X_T
    \right)
  \bigg] -
  \mathbb{E} 
  \bigg[ \mathbbm{1}_{\Omega _N }
    \cdot f\left(
      Y^N_N
    \right)
  \bigg]
\\&= 
  \sum\limits_{n \, = \, 0}^{N - 1}
     \mathbb{E} 
     \left[ \mathbbm{1}_{\Omega _N }
      \left\{ f
        \left(
          X^{t_{n}^N, Y^N_n}_T
        \right)
        - f
      \left(
        X^{t_{n+1}^N, Y^N_{n + 1}}_T
      \right)
    \right\}
  \right] 
\\&= 
  \sum\limits_{n \, = \, 0}^{N - 1}
     \mathbb{E} 
     \left[ \mathbbm{1}_{\Omega_{N,n+1}}
      \left\{ f
        \left(
          X^{t_{n}^N, Y^N_n}_T
        \right)
        - f
      \left(
        X^{t_{n+1}^N, Y^N_{n + 1}}_T
      \right)
    \right\}
  \right] 
\\&\quad -
  \sum\limits_{n \, = \, 0}^{N - 1}
     \mathbb{E} 
     \left[ 
       \mathbbm{1}_{\left(\Omega_{N,n+1} \backslash \Omega_{N}\right)}   
       \left\{ f
        \left(
          X^{t_{n}^N, Y^N_n}_T
        \right)
        - f
      \left(
        X^{t_{n+1}^N, Y^N_{n + 1}}_T
      \right)
    \right\}
  \right]
\end{split}     \end{equation}
and hence
\begin{equation}  \begin{split}
&
  \left| 
    \mathbb{E} 
    \bigg[ \mathbbm{1}_{\Omega _N } 
      \cdot f\left(
        X_T
      \right)
    \bigg] -
    \mathbb{E} 
    \bigg[ \mathbbm{1}_{\Omega _N }
      \cdot f\left(
        Y^N_N
      \right)
    \bigg]
  \right|
\\&\leq 
  \sum\limits_{n \, = \, 0}^{N - 1}
    \left|
    \mathbb{E} 
    \left[ 
      \mathbbm{1}_{\Omega_{N,n+1}}
      \left\{ 
        f\left(
          X^{t_{n}^N, Y^N_n}_T
        \right)
        - f\left(
        X^{t_{n+1}^N, Y^N_{n + 1}}_T
        \right)
      \right\}
    \right]
    \right| 
\\&\quad+
  \sum\limits_{n \, = \, 0}^{N - 1}
    \mathbb{E}\left[ 
      \mathbbm{1}_{\Omega_{N,n+1}}
      \mathbbm{1}_{ \left(\Omega _N \right)^c }
      \left|
        f\left(
         X^{t_{n}^N, Y^N_n}_T
        \right)
        - f
        \left(
          X^{t_{n+1}^N, Y^N_{n + 1}}_T
        \right)
      \right|
    \right]
\end{split}     \end{equation}
and
\begin{equation}  \begin{split}
&
  \left| 
    \mathbb{E} 
    \bigg[ \mathbbm{1}_{\Omega _N } 
      \cdot f\left(
        X_T
      \right)
    \bigg] -
    \mathbb{E} 
    \bigg[ \mathbbm{1}_{\Omega _N }
      \cdot f\left(
        Y^N_N
      \right)
    \bigg]
  \right|
\\&\leq 
  \sum\limits_{n \, = \, 0}^{N - 1}
    \left|
    \mathbb{E} 
    \left[ 
      \mathbbm{1}_{\Omega_{N,n+1}}
        \mathbb{E}
        \left[ 
          f\left(
            X^{t_{n}^N, Y^N_n}_T
          \right)
          - 
          f\left(
            X^{t_{n+1}^N, Y^N_{n + 1}}_T
          \right)
          \bigg| \;
          \mathcal{F}_{t_{n+1}^N}
        \right]
      \right]
    \right| 
\\&\quad+
  \sum\limits_{n \, = \, 0}^{N - 1}
  \left(
    \mathbb{P}\bigg[\left( \Omega_N \right)^c \bigg]
  \right)^{\frac{1}{2}}
  \left\|
      \mathbbm{1}_{\Omega_{N,n+1}}
      \left|
        f\left(
         X^{t_{n}^N, Y^N_n}_T
        \right)
        -
        f
        \left(
          X^{t_{n+1}^N, Y^N_{n + 1}}_T
        \right)
      \right|
  \right\|_{L^2}
\end{split}     \end{equation}
for every $N \in \mathbb N$.
Therefore, we obtain
\begin{align*}
&
  \left| 
    \mathbb{E} 
    \bigg[ \mathbbm{1}_{\Omega _N } 
      \cdot f\left(
        X_T
      \right)
    \bigg] -
    \mathbb{E} 
    \bigg[ \mathbbm{1}_{\Omega _N }
      \cdot f\left(
        Y^N_N
      \right)
    \bigg]
  \right|
\\&\leq 
  \sum\limits_{n \, = \, 0}^{N - 1}
    \left|
    \mathbb{E} 
    \left[ 
      \mathbbm{1}_{\Omega_{N,n+1}}
        \mathbb{E}
        \left[ 
          f\left(
            X^{t_{n+1}^N, X^{t_{n}^N, Y^N_n}_{t_{n+1}^N}}_T
          \right)
          - 
          f\left(
            X^{t_{n+1}^N, Y^N_{n + 1}}_T
          \right)
          \bigg| \;
          \mathcal{F}_{t_{n+1}^N}
        \right]
      \right]
    \right|
\\&+
  \sum\limits_{n \, = \, 0}^{N - 1}
  C N^{-2}
  \left(
    \left\| 
      \mathbbm{1}_{\Omega_{N,n+1}}
      \cdot f\left(
        X^{t_{n}^N, Y^N_n}_T
      \right)
    \right\|_{L^2}
    +
    \left\|
      \mathbbm{1}_{\Omega_{N,n+1}}
      \cdot f\left(
        X^{t_{n+1}^N, Y^N_{n + 1}}_T
      \right)
    \right\|_{L^2}
  \right)
\end{align*}
and, using the Markov property and inequality~\eqref{e_estimates},
\begin{align*}
&
  \left| 
    \mathbb{E} 
    \bigg[ \mathbbm{1}_{\Omega _N } 
      \cdot f\left(
        X_T
      \right)
    \bigg] -
    \mathbb{E} 
    \bigg[ \mathbbm{1}_{\Omega _N }
      \cdot f\left(
        Y^N_N
      \right)
    \bigg]
  \right|
\\&\leq
  \sum\limits_{n \, = \, 0}^{N - 1}
    \left|
    \mathbb{E} 
    \left[ 
      \mathbbm{1}_{\Omega_{N,n+1}}
        \left( 
          u\left(
            t_{n+1}^N, X^{t_{n}^N, Y^N_n}_{t_{n+1}^N}
          \right)
          - 
          u\left(
            t_{n+1}^N, Y^N_{n + 1}
          \right)
        \right)
      \right]
    \right|
\\&+
  \sum\limits_{n \, = \, 0}^{N - 1}
  C N^{-2}
  \left(
    \left\| 
      \mathbbm{1}_{\Omega_{N,n+1}}
      \cdot f\left(
        X^{t_{n}^N, Y^N_n}_T
      \right)
    \right\|_{L^2}
    +
    \left\|
      \mathbbm{1}_{\Omega_{N,n+1}}
      \cdot f\left(
        X^{t_{n+1}^N, Y^N_{n + 1}}_T
      \right)
    \right\|_{L^2}
  \right)
\\&\leq
  \sum\limits_{n \, = \, 0}^{N - 1}
  10 C^9 N^{-2}
  +
  \sum\limits_{n \, = \, 0}^{N - 1}
  C N^{-2}
  \Bigg(
    \left\| 
      \mathbbm{1}_{\Omega_{N,n+1}}
      \, L\left(
        1 + \left|X^{t_{n}^N, Y^N_n}_T\right|^{\delta}
      \right)
    \right\|_{L^2}
\\&+
    \left\|
      \mathbbm{1}_{\Omega_{N,n+1}}
      \, L\left(
        1 + \left|X^{t_{n+1}^N, Y^N_{n + 1}}_T\right|^{\delta}
      \right)
    \right\|_{L^2}
  \Bigg)
\end{align*}
for every $N \in \mathbb N$.
Finally, 
using $\Omega_{N,n+1}\subseteq\Omega_{N,n}$
for all $ n \in \left\{ 0,1, \dots, N-1\right\}$,
$ N \in \mathbb{N}$
and $|x|^\delta \leq 1+|x|^R$ for all $x\in\mathbb{R}$,
we arrive at
\begin{align*}
&
  \left| 
    \mathbb{E} 
    \bigg[ \mathbbm{1}_{\Omega _N } 
      \cdot f\left(
        X_T
      \right)
    \bigg] -
    \mathbb{E} 
    \bigg[ \mathbbm{1}_{\Omega _N }
      \cdot f\left(
        Y^N_N
      \right)
    \bigg]
  \right|
  \leq
  10 C^9 N^{-1}
  +
  L C N^{-2}
\\&\cdot
  \left\{
    \sum\limits_{n \, = \, 0}^{N - 1}
      \left\| 
        \mathbbm{1}_{\Omega_{N,n}}
        \left( 2 + 
        \left|X^{t_{n}^N, Y^N_n}_T\right|^{R}
        \right)
      \right\|_{L^2} \!\!
      +
      \left\| 
        \mathbbm{1}_{\Omega_{N,n+1}}
        \left( 2 + 
        \left|X^{t_{n+1}^N, Y^N_{n+1}}_T\right|^{R}
        \right)
      \right\|_{L^2}
  \right\}
\\&\leq
  10 C^9 N^{-1}
  + 2
  L C^2 N^{-1}
  \leq
  12
  C^9 N^{-1}
\end{align*}
for every $N \in \mathbb N$ due to \eqref{c_definition}.
This proves the assertion
of Lemma~\ref{lemm_ew}.
\end{proof}

\section*{Acknowledgement}
This work has been partially supported
by the Collaborative 
Research Centre~701 ``Spectral 
Structures and Topological Methods
in Mathematics'' 
and
by the research project
``Numerical solutions of stochastic
differential equations with
non-globally Lipschitz continuous
coefficients'' 
both funded 
by the German Research
Foundation. Additionally, the authors 
thank three anonymous referees 
for their very helpful comments.
%
%
%
%
%
%
%

\end{document}